\documentclass[12pt,reqno]{amsart}
\usepackage{amsmath,amssymb,amsfonts,amsthm,amscd,indentfirst}
\usepackage[a4paper, left=2.5cm, right=2.5cm, top=2cm, bottom=2cm]{geometry}
\usepackage{hyperref}
\usepackage[numbers,sort]{natbib}

\hypersetup{
colorlinks=true,
linkcolor=black
}

\numberwithin{equation}{section}

\newtheorem{proposition}{Proposition}[section]
\newtheorem{theorem}[proposition]{Theorem}
\newtheorem{lemma}[proposition]{Lemma}

\newtheorem{remark}[proposition]{Remark}

\newtheorem{definition}[proposition]{Definition}

\usepackage{mathtools}
\usepackage{comment}
\usepackage{graphicx}

\begin{document}
\title[Viscosity Solutions Of Capillary Parabolic p-Laplacian Equations]{On the Viscosity Solutions of Parabolic p-Laplacian Equations with Capillary-Type Boundary Conditions}
\author[Zhenghuan Gao]{Zhenghuan Gao}
\address{School of Mathematics and Statistics, Xi'an Jiaotong University, Xi'an, China.}
\email{gzhmath@xjtu.edu.cn}
\author[Jin Yan]{Jin Yan}
\address{School of Mathematical Sciences, University of Science and Technology of China, Hefei, China.}
\email{yjoracle@mail.ustc.edu.cn}
\author[Yang Zhou]{Yang Zhou}
\address{School of Mathematical Sciences, University of Science and Technology of China, Hefei, China.}
\email{zy19700816@mail.ustc.edu.cn}
\begin{abstract}
In this paper, we establish the well-posedness and large-time asymptotic behavior of viscosity solutions to singular/degenerate parabolic $p$-Laplacian equations with general capillary-type boundary conditions, including Neumann and prescribed contact angle cases, on strictly convex domains. By establishing a gradient estimate independent of the $C^0$ norm of the solution via the maximum principle, and by analyzing the problem through an approximation procedure together with associated elliptic eigenvalue problems, we prove the existence, uniqueness, and asymptotic behavior of solutions. For the elliptic problem with Neumann boundary conditions, we first focus on flat domains with the zero Neumann condition. By reflecting $u$ across the flat boundary $T_1$ and then using inf- and sup-convolution arguments in the reflected domain, we obtain the $C^{1,\alpha}$ result. For the general elliptic case, we obtain sharp global $C^{1,\alpha}$ regularity by flattening the boundary and employing compactness arguments together with an ``improvement of flatness'' iteration. With an extra condition in the iteration, we can also deal with the singular case $1<p<2$. In the parabolic setting, the spatial H\"older regularity of $Du$ follows from elliptic estimates combined with the Lipschitz continuity of $u$ in time, which in turn yields joint H\"older continuity in $(x,t)$. Extensions to non-convex domains are also discussed by incorporating a suitable forcing term.
\end{abstract}

\maketitle

\section{Introduction}
In this paper, we study viscosity solutions of singular/degenerate parabolic $p$-Laplacian equations with capillary-type boundary conditions, which include the Neumann boundary condition and the prescribed contact angle boundary condition. We first consider the parabolic equations on strictly convex bounded domains
\begin{equation}\label{par eq}
\left\{\begin{aligned}
& u_t=\operatorname{div}(|Du|^{p-2}Du)+f(x,u) &&\text{in~}\Omega\times[0,\infty),\\
& u(x,0)=u_0(x)  &&\text{in~}\Omega,\\
& |Du|^{q-1}u_\nu=-\phi(x) &&\text{on~}\partial\Omega\times[0,\infty),
\end{aligned}\right.
\end{equation}
where $\Omega\subset \mathbb{R}^n$ is a bounded strictly convex domain, $\nu$ denotes the inward unit normal vector, $p>1$, $q\geq0$. 

Throughout this paper, $u_0(x)$ is a smooth function satisfying
$$
|Du_0|^{p-2}Du_0 \in C^1(\overline{\Omega}).
$$
If $q>0$, it further satisfies
$$
\|(|Du_0|^2 + \varepsilon^2)^{\frac{q-1}{2}} D_{\nu}u_0\|_{C^2(\partial\Omega)} \leq L 
$$
for all sufficiently small $\varepsilon \ge 0$. Moreover, the following compatibility condition holds:
\begin{equation}\label{CC1}
|Du_0|^{q-1} u_{0,\nu} = -\phi(x) \quad \text{on } \partial\Omega.
\end{equation}
When $q=0$, the boundary condition should be interpreted as the prescribed contact angle condition
\begin{equation}\label{bdrq=0}
u_{\nu}=-\phi(x)|Du| \quad \text{on~} \partial\Omega\times[0,\infty),
\end{equation} 
when $q=1$, it reduces to the Neumann condition, and when $q=p-1$,  it reduces to the conormal condition.

We always assume that $f$ satisfies
\begin{equation}\label{f}
f\in C^{1,\beta}(\overline{\Omega}\times \mathbb{R}),\quad f_{z}(x,z)\leq0, \quad |f|+|D_{x}f|\leq L,
\end{equation}
for some $0<\beta<1$, and $\phi$ satisfies 
\begin{equation}\label{phi0}
\phi\in C^{2}(\partial{\Omega}),\quad \|\phi(x)\|_{C^2(\partial\Omega)}\leq L.
\end{equation}

The capillary-type problem is a classical subject with deep roots in mathematical physics and geometry. Its fundamental importance originates from its role in modeling capillary surfaces, where the contact angle is determined by material properties. This physical interpretation intrinsically links the problem to the analysis of free boundaries and surfaces of prescribed mean curvature. The study of such problems primarily relies on two approaches: the test function method and the maximum principle. For the former, boundary gradient estimates and corresponding existence theorems for mean curvature equations under positive gravity were established by Ural'tseva \cite{U1973}, Simon-Spruck \cite{SS1976ARMA}, and Gerhardt \cite{G1976Pisa}. Lieberman \cite{L2013} later extended these results to more general quasi-linear equations in divergence form. However, these works mainly address conormal boundary conditions and are difficult to generalize to other boundary types. Therefore, for $p$-Laplacian equations, we employ the maximum principle to obtain the crucial gradient estimate. This technique was introduced by Spruck \cite{S1975CPAM} and further developed by Korevaar \cite{K1988CPDE} and Lieberman \cite{L1988CPDE}. In the book \cite{L2013}, Lieberman treated a broader class of quasi-linear elliptic equations with boundary conditions corresponding to $q=0$ and $q>1$ in the zero-gravity case (see also Xu \cite{X2016CPAA} for a new proof in the mean curvature setting). More recently, Ma-Xu \cite{MX2016CPAAAIM} derived boundary gradient estimates for mean curvature equations with Neumann conditions and established an existence result under positive gravity. For the case $0<q<1$, results on strictly convex domains were obtained by Wang-Wei-Xu \cite{WWX2019CPAA}. We refer the reader to \cite{FL2014JDG, FS2016Invention, MWW2025AIM, MWWX2025MathZ, MW2023JGA, SWX2022JDG, WWX2024MA, WWX2024JFA} for more results on elliptic capillary problems.

The parabolic capillary-type problem has also been extensively studied, with particular attention paid to the mean curvature flow and its limiting case, the level-set flow. We refer the reader to \cite{AW1994CVPDE, AC2009IMJ, GLX2024JFA, GMWW2021JMS, G1996, GOS1999JDE, H1984JDG, H1986Invention, H1989JDE, J2023CVPDE, JKMT2011JMPA, MWW2018JFA, MT2017NODEA, WWX2019CPAA} and the references therein for a detailed account of its development. 

For another important quasi-linear operator, the $p$-Laplacian operator, Lieberman \cite{L1990NA} considered the parabolic $p$-Laplacian equation under conormal boundary conditions. However, there are few analogous results for more general capillary-type conditions, or even for the Neumann boundary condition. Motivated by developments in mean curvature and level-set flow, this paper employs the maximum principle to investigate the well-posedness and large-time behavior of solutions to capillary-type boundary value problems. The key to obtaining the asymptotic behavior is to derive a gradient estimate that is independent of the $C^0$ norm of $u$, which we will establish by means of the maximum principle. This is also the reason why we require the domain to be strictly convex.\vspace{4pt}

From the viscosity perspective, such problems also arise naturally. The following singular/degenerate fully nonlinear equation with uniformly elliptic $F$
$$
|Du|^{\gamma}F(D^2u)=f
$$
was first studied by Birindelli-Demengel in the pioneering works \cite{BD2006ADE, BD2007DCDS, BD2010JDE} for the singular case $-1<\gamma<0$. For the degenerate case $\gamma>0$, the groundbreaking work was carried out by Imbert-Silvestre in \cite{IS2013AIM}, where they established an interior $C^{1,\alpha}$ estimate for solutions. Subsequently, Araújo-Ricarte-Teixeira \cite{ART2015CVPDE} proved an optimal interior $C^{1,\alpha}$ estimate under the assumption that $F$ is convex. We refer the reader to \cite{F2021RE, FRZ2021BLMS, APPT2022AIM, BBLL2024, LLY2024JMPA} and the references therein for further interior regularity results.

For the Neumann problem
$$
\left\{\begin{aligned}
& |Du|^{\gamma}F(D^2u)=f &&\text{in~} \Omega,\\
& u_\nu=g &&\text{on~} \partial\Omega,
\end{aligned}\right.
$$
Milakis-Silvestre \cite{MS2006CPDE} proved $C^{1,\alpha}$ regularity up to the boundary for the case $\gamma=0$ on a flat domain. Later, Li-Zhang \cite{LZ2018ARMA} obtained $C^{1,\alpha}$ regularity estimates on $C^1$ domains for oblique boundary conditions, also with $\gamma=0$. See also the generalization to the parabolic oblique problem by Chatzigeorgiou-Milakis \cite{CM2021RMI}. For the degenerate case $\gamma\ge0$, Banerjee-Verma \cite{BV2022PA} established global $C^{1,\alpha}$ regularity under a nonhomogeneous Neumann condition on $C^2$ domains. Ricarte \cite{R2020NA} proved optimal $C^{1,\alpha}$ regularity on $C^2$ domains when $F$ is convex. Recently, Byun-Kim-Oh \cite{BKO2025CVPDE} obtained sharp $C^{1,\alpha}$ regularity on $C^1$ domains for oblique boundary conditions. For the singular case, Patrizi \cite{P2008JMPA} studied a general class of operators and proved global Lipschitz regularity under the homogeneous Neumann condition on $C^2$ domains, see also Birindelli-Demengel-Leoni \cite{BDL2022NA} for the mixed boundary value problems.

In addition, the viscosity $p$-Laplacian type equation 
$$
|Du|^{\gamma}\left(\Delta u+(p-2)|Du|^{-2}u_iu_ju_{ij}\right)=f
$$
has attracted increasing interest in recent years. Attouchi-Parviainen-Ruosteenoja \cite{APR2017JMPA} proved interior $C^{1,\alpha}$ regularity for $\gamma=0$, and later Attouchi-Ruosteenoja \cite{AR2018JDE} extended the result to the range $\gamma>-1$. For related interior regularity results, we refer the reader to \cite{WYJ2025PA, S2022JMAA, PT2021JDE, SR2020CVPDE} and the references therein.

For the parabolic equation
$$
u_t-|Du|^{\gamma}\left(\Delta u+(p-2)|Du|^{-2}u_iu_ju_{ij}\right)=f,
$$
Jin-Silvestre \cite{JS2017JMPA} obtained interior $C^{1+\alpha,\frac{1+\alpha}{2}}$ estimates for $\gamma=0$ and $f\equiv0$, and later Imbert-Jin-Silvestre \cite{IJS2019ANA} extended the result to $\gamma>-1$ and $f\equiv0$. Subsequently, Attouchi-Ruosteenoja \cite{A2020NA, AR2020DCDS} further extended these results to the case $\gamma>-1$ with $f$ bounded and continuous. For further interior regularity results, we refer the reader to \cite{FZ2023CVPDE, AS2022CVPDE, LY2024JDE} and the references therein.

Besides, Attouchi-Barles \cite{AB2015JMPA} studied the Dirichlet problem with $\gamma=p-2$ and derived asymptotic behavior. However, to the best of our knowledge, there are no results for non-homogeneous Neumann or capillary-type boundary conditions.

We now present the main results of this paper. Due to technical reasons, we must treat the cases $q>0$ and $q=0$ separately.

For the case $q>0$, we have the following existence and uniqueness result for \eqref{par eq}.
\begin{theorem}\label{thm1.1}
Assume $\Omega\subset \mathbb{R}^n$ is a bounded strictly convex $C^{2,\beta}$ domain, $q>0$, $f$ satisfies \eqref{f}, and $\phi$ satisfies \eqref{phi0}. Then we have the following existence results
\begin{itemize}
\item when $q=1$, there exists a viscosity solution $u$ of \eqref{par eq}, which is globally Lipschitz in $(x,t)$, and the spatial gradient $Du$ is globally H\"older in $(x,t)$. To be specific, for any $\alpha\in(0,\alpha_0)$, where $\alpha_0$ is the optimal exponent for interior $C^{1,\alpha_0}$ regularity of weak solutions to elliptic $p$-Laplacian equations (see Section 4.2), 
$$
|u(x,t)-u(x,s)|\leq M(|t-s|+|x-y|),\quad |Du(x,t)-Du(y,s)|\leq M(|x-y|^{\alpha}+|t-s|^{\frac{\alpha}{1+\alpha}}),
$$
for all $t,s\in[0,\infty)$, $x,y\in\overline{\Omega}$.\\
\item when $q=p-1$, the above results hold for $\alpha=\alpha_1
$, where $\alpha_1$ is the optimal exponent for the elliptic conormal problem considered in \cite{L1988NA}.\\
\item when $q\neq 1,p-1$, there exists a viscosity solution $u$ of \eqref{par eq}, which is globally Lipschitz in $(x,t)$, and the spatial gradient $Du$ is locally $C^{\alpha_0}$ in $x$ and globally $C^{\frac{\alpha_0}{1+\alpha_0}}$ in $t$. To be specific, 
$$
|u(x,t)-u(x,s)|\leq M(|t-s|+|x-y|),
$$
for all $t,s\in[0,\infty)$,  $x,y\in\overline{\Omega}$, and
$$
|Du(x,t)-Du(y,s)|\leq M(|x-y|^{\alpha_0}+|t-s|^{\frac{\alpha_0}{1+\alpha_0}}),
$$
for all $t,s\in[0,\infty)$ and $x,y\in\Omega'\subset\subset\Omega$.
\end{itemize}
Moreover, for $p \geq 2$, the Lipschitz solution is unique if either $q = 1$ or $q > 0$ with $|\phi(x)| > 0$ on $\partial\Omega$.
\end{theorem}

In order to overcome the degeneracy or singularity of the $p$-Laplacian operator, we denote $v=\sqrt{\varepsilon^2+|Du|^2}$ for $\varepsilon>0$, and consider the following approximate problem
\begin{equation}\label{app par eq}
\left\{\begin{aligned}
& u_t=\operatorname{div}(v^{p-2}Du)+f(x,u) &&\text{in~}\Omega\times[0,\infty),\\
& u(x,0)=u_{0}(x)  &&\text{in~}\Omega,\\
& u_\nu=-\phi_{\varepsilon}(x)v^{1-q} &&\text{on~}\partial\Omega\times[0,\infty),
\end{aligned}\right.
\end{equation}
with the compatible condition which determines $\phi_{\varepsilon}(x)$
\begin{equation}\label{CC2}
D_{\nu}u_{0}=-\phi_{\varepsilon}(x)\Big(\sqrt{\varepsilon^2+|Du_{0}|^2}\Big)^{q-1}\quad\text{on~}\partial\Omega.
\end{equation}

For the approximate problem, once we get the global gradient estimate, the equation becomes uniformly parabolic, and hence higher-order regularity follows.
\begin{theorem}\label{thm1.2}
Assume $\Omega\subset \mathbb{R}^n$ is a bounded strictly convex $C^{2,\beta}$ domain and $q>0$. If $f$ satisfies \eqref{f} and $\phi$ satisfies \eqref{phi0}, then there exists a unique $C^{2,\sigma}(\overline\Omega\times[0,\infty))$ solution $u_{\varepsilon}$ of \eqref{app par eq}.
\end{theorem}

Next, we aim to study the asymptotic behaviors for the approximate solutions
\begin{equation}\label{app par eq2}
\left\{\begin{aligned}
& u_t=\operatorname{div}(v^{p-2}Du)+\varphi(x) \quad&&\text{in~}\Omega\times[0,\infty),\\
& u(x,0)=u_{0}(x) \quad &&\text{in~}\Omega,\\
& u_\nu=-\phi_{\varepsilon}(x)v^{1-q} \quad&&\text{on~}\partial\Omega\times[0,\infty).
\end{aligned}\right.
\end{equation}
We consider the corresponding elliptic eigenvalue problem 
\begin{equation}\label{app ell eq}
\left\{\begin{aligned}
& -\operatorname{div}(v^{p-2}Du)=\lambda_{\varepsilon}+\varphi(x) &&\text{in~}\Omega,\\
& u_\nu=-\phi_{\varepsilon}(x)v^{1-q} &&\text{on~}\partial\Omega,
\end{aligned}\right.
\end{equation}
and obtain the following result by a further approximation (see Section 5).
\begin{theorem}\label{thm1.3}
Assume $\Omega\subset \mathbb{R}^n$ is a bounded strictly convex $C^{2,\beta}$ domain and $q>0$. Then for any $\varphi(x)$, $\phi(x)\in C^{\infty}(\overline\Omega)$, there exists a unique $\lambda_{\varepsilon}\in \mathbb{R}$ and a function $u_{\varepsilon}\in C^{\infty}(\overline\Omega)$ solving \eqref{app ell eq}. Moreover, $u_\varepsilon$ is unique up to a constant.
\end{theorem}
By Theorem \ref{thm1.3} and a routine argument for uniformly parabolic equations, we get the asymptotic behavior of the solutions to \eqref{app par eq}.
\begin{theorem}\label{thm1.4}
Assume $\Omega\subset \mathbb{R}^n$ is a bounded strictly convex $C^{2,\beta}$ domain and $q>0$. Suppose $\varphi(x)$, $\phi(x)\in C^{\infty}(\overline\Omega)$, then any smooth solutions  $u_1$ and $u_2$ of \eqref{app par eq} with smooth initial data satisfy $\lim_{t\to\infty}|u_1-u_2|_{C^{\infty}(\overline\Omega)}=0$. In particular, the solution $u(x,t)$ of \eqref{app par eq2} satisfies $\lim_{t\to\infty}|u(x,t)-\lambda_{\varepsilon}t-\omega_{\varepsilon}|_{C^{\infty}(\overline\Omega)}=0$, where $(\lambda_{\varepsilon},\omega_{\varepsilon})$ is a solution of \eqref{app ell eq}.
\end{theorem}

Finally, we pass to the limit of $\varepsilon\to0$ to obtain the asymptotic behavior of the solutions to \eqref{par eq}.

As in Theorem \ref{thm1.3}, we consider the corresponding elliptic eigenvalue problem 
\begin{equation}\label{ell eq}
\left\{\begin{aligned}
& -\operatorname{div}(|Du|^{p-2}Du)=\lambda_{0}+\varphi(x) &&\text{in~}\Omega,\\
& |Du|^{q-1}u_\nu=-\phi(x) &&\text{on~}\partial\Omega,
\end{aligned}\right.
\end{equation}
and obtain the following result by a further approximation (see Sections 6).
\begin{theorem}\label{thm1.5}
Assume $\Omega\subset \mathbb{R}^n$ is a bounded strictly convex $C^{2,\beta}$ domain and $q>0$. Then for any $\varphi(x), \phi(x)\in C^{\infty}(\overline{\Omega})$, we have the following existence results
\begin{itemize}
\item when $q=1$ or $p-1$, there exists a $\lambda_{0}\in \mathbb{R}$ and a $C^{1,\alpha}(\overline{\Omega})$ viscosity solution $u$ of \eqref{ell eq}.\\
\item when $q\neq 1,p-1$, there exists a $\lambda_{0}\in \mathbb{R}$ and a $C^{1,\alpha}(\Omega)$ viscosity solution $u$ of \eqref{ell eq}.
\end{itemize}
Moreover, we have the following uniqueness results
\begin{itemize}
\item when  $q=p-1$, then $\lambda_0$ is unique,  and $u$ is unique up to a constant.\\

\item for $p\geq2$, when $q=1$ or $q>0$ with $|\phi(x)|>0$ on $\partial\Omega$, then $\lambda_0$ is unique.
\end{itemize}
\end{theorem}

For the parabolic equation \eqref{par eq}, by using the comparison principle, we have the following long-time behavior.
\begin{theorem}\label{thm1.6}
Assume $\Omega\subset \mathbb{R}^n$ is a bounded, strictly convex $C^{2,\beta}$ domain, $p\geq2$, $f$ depends only on $x$ and satisfies \eqref{f}, and $\phi$ satisfies \eqref{phi0}. If either $q=1$ or $q>0$ with $|\phi(x)|>0$ on $\partial\Omega$, then the unique viscosity solution of \eqref{par eq} satisfies
$$
\lim_{t\to\infty}\frac{u(x,t)}{t}=-\lambda_{0},
$$
where $\lambda_{0}$ is the unique constant in Theorem \ref{thm1.5}.
\end{theorem}

If the boundary data are zero, we can derive a more delicate asymptotic behavior as in \cite{JKMT2011JMPA}.
\begin{theorem}\label{thm1.7}
Assume $\Omega\subset \mathbb{R}^n$ is a bounded strictly convex $C^{2,\beta}$ domain, $q>0$ and $p\geq2$. For $f\equiv\phi\equiv 0$, the unique viscosity solution $u$ of  \eqref{par eq} satisfies $\lim_{t\to\infty}u(x,t)=C$, where $C$ is a constant.
\end{theorem}

For the case $q=0$, i.e., the prescribed contact angle problem with the usual condition
\begin{equation}\label{phi}
\|\phi\|_{C^0(\partial\Omega)}<1,
\end{equation}
we remark that the boundary condition should be interpreted as \eqref{bdrq=0}
\begin{equation*}
u_{\nu}=-\phi(x)|Du| \quad \text{on~} \partial\Omega\times[0,\infty).
\end{equation*}
As in \cite{GMWW2021JMS}, we also need an extra condition on $\phi(x)$, which is the near vertical condition
\begin{equation}\label{nvc}
\|\phi\|_{C^2(\partial\Omega)}\leq \varepsilon_1,
\end{equation}
where $\varepsilon_1>0$ is a small constant. In fact, all  the results for $q>0$ can be derived under this condition.
\begin{theorem}\label{thm1.8}
Assume $\Omega\subset \mathbb{R}^n$ is a bounded strictly convex $C^{2,\beta}$ domain, $f$ satisfies \eqref{f}, and  $\phi(x)$ satisfies \eqref{nvc}. Then all the results in Theorem \ref{thm1.1} to Theorem \ref{thm1.7} hold for $q=0$.
\end{theorem}

Finally, inspired by \cite{JKMT2011JMPA,J2023CVPDE}, for non-convex $\Omega$, we consider the corresponding equation with a forcing term $a(x,u)|Du|^{\tilde p-1}$, where $\tilde p\coloneqq\max\{2,p\}$,
\begin{equation*}
\left\{\begin{aligned}
& u_t=\operatorname{div}(|Du|^{p-2}Du)+a(x,u)|Du|^{\tilde p-1}+f(x,u) &&\text{in~}\Omega\times[0,\infty),\\
& u(x,0)=u_0(x) &&\text{in~}\Omega,\\
& |Du|^{q-1}u_\nu=-\phi(x) &&\text{on~}\partial\Omega\times[0,\infty),
\end{aligned}\right.
\end{equation*}
where $a(x,u)$ satisfying 
\begin{equation}\label{a1-1}
 a\in C^{1,\beta}(\overline{\Omega}\times \mathbb{R}),\quad a_{z}(x,z)\leq0
\end{equation}
and the following forcing condition holds for some $c_1=c_1(n, p, q, L, \Omega)>0$ small and $C_1=C_1(n, p, q, L, \Omega)>0$ large:
\begin{align}\label{a2-1}
\begin{cases}
c_1a(x,u)^2-|Da(x,u)|-C_1|a(x,u)|\gg 1 \quad &\text{if } p\geq2,\\
|a(x,u)|>0 \quad &\text{if } 1<p<2.
\end{cases}
\end{align}
\begin{theorem}\label{thm1.9}
Assume $\Omega\subset \mathbb R^n$ is a bounded $C^{2,\beta}$ domain, $f$ satisfies \eqref{f}, $\phi$ satisfies \eqref{phi0}. If $q>0$ and the forcing term $a(x,u)$ satisfies \eqref{a1-1} and \eqref{a2-1}, then all the results in Theorem \ref{thm1.1} to Theorem \ref{thm1.7} hold for the corresponding equation with the forcing term. Moreover, if $q=0$, then the same results as in Theorem \ref{thm1.8} hold, and the nearly vertical condition \eqref{nvc} can be removed (replaced by \eqref{phi}).
\end{theorem}

We now briefly discuss the main difficulties of this paper. For the gradient estimate, the auxiliary functions are well known; however, by treating $w_i$ as a whole, we avoid the tedious calculations that would come from solving for $u_{1i}$ and then substituting back, thereby simplifying the proof.  We remark here that this method also applies to mean curvature flow or level-set flow. For the uniqueness, we need to verify the conditions for the comparison principle in \cite{B1999JDE}. At first glance, neither the $p$-Laplacian operator nor the boundary condition operator satisfies those conditions. However, with the aid of gradient bounds, we can redefine the equations both in the interior and on the boundary, and prove that the viscosity solutions are always solutions of the new equations, thus obtaining the desired comparison results. For the $C^{1,\alpha}$ regularity under Neumann boundary conditions, we first focus on flat domains with the zero Neumann condition. Since the equation depends on $Du$, the method in \cite{MS2006CPDE} or the doubling variables type argument used in \cite{LZ2018ARMA} does not apply. Nevertheless, thanks to the equivalence between viscosity solutions and weak solutions for the $p$-Laplacian equation, we can ignore the zero-measure set $T_1$ by reflecting $u$ across the flat boundary $T_1$, and then use inf- and sup-convolution arguments in the reflected domain to obtain the result. For the general case, we employ the ``improvement of flatness'' iteration, which has been used in \cite{BV2022PA, R2020NA, BKO2025CVPDE} for degenerate equations of the form $|Du|^{\gamma}F(D^2u)=f$ with $\gamma\geq0$. Moreover, we impose an extra condition that $\frac{\|f\|_{L^{\infty}(\Omega_1)}}{\left(|q|^{2}+|\varepsilon|^{2}\right)^{\frac{p-2}{2}}}$ is sufficiently small in the iteration, and carefully avoid the ``bad range'' of $q$, to extend the results to the singular case $1<p<2$. Finally, for the asymptotic behavior, we need to construct a carefully designed barrier function $w$ (see Theorem~\ref{thm5.1}). For the singular equations, this construction also appears to be new.\vspace{10pt}

The paper is organized as follows. Section 2 collects preliminary materials on viscosity and weak solutions for $p$-Laplacian equations and recalls basic notation. Section 3 is devoted to the crucial gradient estimate and the proofs of existence and Lipschitz regularity, while Section 4 establishes the $C^{1,\alpha}$ regularity for the Neumann case. In Section 5 we study the asymptotic behavior of the approximate problem, while Section 6 passes to the limit $\varepsilon\to0$ and obtains the asymptotic profile for the original equation. Section 7 treats the prescribed contact angle case $q=0$, where the condition \eqref{nvc} from \cite{GMWW2021JMS} plays an important role. Finally, Section 8 extends the results to non-convex domains by means of a carefully designed forcing term, highlighting the parallel with the mean curvature flow analysis in \cite{J2023CVPDE,JKMT2011JMPA}.

\section{Preliminaries}
\subsection{Notions of the p-Laplacian equation}
In this subsection, we recall the notion of viscosity solutions to the $p$-Laplacian equation. The minus $p$-Laplace operator is defined as
$$
-\Delta_p u\coloneqq-\operatorname{div}\big( |Du|^{p-2}Du\big)\coloneqq F(Du,D^2u).
$$
We consider the following general degenerate (or singular) elliptic equations of $p$-Laplacian type for continuous $f$:
\begin{equation}\label{ell}
-\operatorname{div}\big( |Du|^{p-2}Du\big)=f(x, u, Du),
\end{equation}
We first recall the notion of semi-jets, which are the
natural generalizations of the sub- and super-differential, see \cite{BCESS1997,CIL1992,FRZ2024MA} for more details.
\begin{definition} The subjet of $ u : \Omega \to \mathbb{R} $ at $ x $ is given by letting $ (\eta, X) \in J^{2,-}u(x) $ if
$$
u(y) \geq u(x) + \eta \cdot (y - x) + \frac{1}{2} \langle X(y - x), (y - x) \rangle + o(|y - x|^2)
$$
as $y\to x$. If, in addition, we require $y\in O$ where $O$ is a locally compact neighborhood of x, then denote it by $J^{2,-}_O$. The closure of a subjet is defined by $(\eta,X)\in\overline{J}^{2,-}u(x)$ if there exists a sequence $(\eta_{j},X_{j})\in J^{2,-}u(x_{j})$ such that $(x_{j},u(x_{j}),\eta_{j},X_{j})\to(x,r,\eta,X)$ with some $r\in\mathbb{R}$ as $j\to\infty$. Obviously, $r=u(x)$ if $u$ is continuous. The superjet $J^{2,+}$ and its closure $\overline{J}^{2,+}$ are defined in a similar way, with the opposite inequality.  
\end{definition}

Let us now state the different types of solutions to \eqref{ell} that we will consider.

The following definition of viscosity solutions for the $p$-Laplacian equation was introduced by Juutinen-Lindqvist-Manfredi \cite{JLM2001SIAMMA} for the case $f\equiv 0$, extended by Julin-Juutinen \cite{JJ2012CPDE} to $f=f(x)$, and further generalized by Medina-Ochoa \cite{MO2019ANA} to the general form $f=f(x,u,Du)$. For $p \geq 2$, this definition coincides with the usual notion of viscosity solutions for degenerate elliptic equations.

\begin{definition}[Viscosity solution]\label{vis-ell-def-1}
A continuous function $u:\Omega\to (-\infty,+\infty]$ is a viscosity supersolution to \eqref{ell} if whenever $\varphi \in C^2(\Omega)$ is such that
$$
\begin{cases}
\varphi(x_0) = u(x_0), \\
\varphi(x) < u(x) & \text{when } x \neq x_0, \\
D\varphi(x) \neq 0 & \text{when } x \neq x_0,
\end{cases}
$$
we have
$$
\limsup_{\substack{x \to x_0,\\ x\neq x_0}} \bigl( - \Delta_p \varphi(x) - f(x, u, D\varphi(x)) \bigr) \geq 0.
$$

A continuous function $u:\Omega\to (-\infty,+\infty]$ is a viscosity subsolution to \eqref{ell} if whenever $\varphi \in C^2(\Omega)$ is such that
$$
\begin{cases}
\varphi(x_0) = u(x_0), \\
\varphi(x) > u(x) & \text{when } x \neq x_0, \\
D\varphi(x) \neq 0 & \text{when } x \neq x_0,
\end{cases}
$$
we have
$$
\liminf_{\substack{x \to x_0,\\x\neq x_0}} \bigl( - \Delta_p \varphi(x) - f(x, u, D\varphi(x)) \bigr) \leq 0.
$$
A function is a viscosity solution if and only if it is both viscosity sub- and supersolution.
\end{definition}

The definition is natural, but seems hard to understand what happens when $Du=0$. Another definition of solutions was introduced by Birindelli-Demengel in \cite{BD2006ADE, BD2007DCDS, BD2010JDE}. This definition adapts to the elliptic setting the notion of solution originally developed by Chen-Giga-Goto \cite{CGG1991JDG} and Evans-Spruck \cite{ES1991JDG} for parabolic singular problems. It constitutes a variant of the usual viscosity solution concept, avoiding the use of test functions with vanishing gradient at the point of testing.

\begin{definition}[Alternative definition]\label{vis-ell-def-2}
A continuous function $u:\Omega\to (-\infty,+\infty]$ is a viscosity supersolution to \eqref{ell} if 
\begin{itemize}
\item either $u$ is not constant near $x_0\in\Omega$, and for every $\varphi \in C^2(\Omega)$ satisfying
$$
\begin{cases}
\varphi(x_0) = u(x_0),\\
\varphi(x) < u(x) & \text{for } x \neq x_0,\\
D\varphi(x_0) \neq 0,
\end{cases}
$$
we have
$$
-\Delta_p \varphi(x_0) - f(x_0, u(x_0), D\varphi(x_0)) \ge 0;
$$
\item or $u$ is constant near $x_0$, then
$$
-f(x_0, u(x_0), 0) \ge 0.
$$
\end{itemize}
Viscosity subsolutions and solutions are defined similarly.
\end{definition}
The equivalence of the two definitions was proved by \cite{AR2018JDE} (see also \cite{DFQ2010CVPDE}); we will reformulate it for the reader's convenience and for use in the parabolic setting.
\begin{lemma}\label{vis-ell-def-equal}
Definition \ref{vis-ell-def-1} is equivalent to Definition \ref{vis-ell-def-2}.
\end{lemma}
\begin{proof}
When $p\geq2$, there is nothing to prove since the singularity does not occur, so we only prove the equivalence for $1<p<2$.\vspace{4pt}

\textbf{Step 1.} Assume $u$ is a viscosity supersolution in the sense of Definition \ref{vis-ell-def-1}. We show it also satisfies Definition \ref{vis-ell-def-2}.

If $u$ is not constant, the result follows easily since $|D\varphi(x_0)|\neq0$. Hence we may assume $u$ is constant near $x_0$.

For any $\varphi\in C^2$ that touches $u$ from below at $x_0$ and $D\varphi(x_0)=0$, consider
$$
\tilde{\varphi}(x)\coloneqq u(x_0)-|x-x_0|^{q},
$$
where $q>\frac{p}{p-1}>2$. Then $u-\tilde{\varphi}$ also attains its strict minimum $x_0$, and $x_0$ is the isolated
critical point. By Definition \ref{vis-ell-def-1}, 
$$
0\leq \limsup_{\substack{x\to x_0,\\x\neq x_0}}  - \Delta_p \tilde{\varphi}(x) - f(x, u, 0)=-f(x_0,u(x_0),0).
$$
Thus $u$ satisfies the constant case condition in Definition \ref{vis-ell-def-2}.\vspace{4pt}

\textbf{Step 2.} Conversely, assume $u$ is a viscosity supersolution in the sense of Definition \ref{vis-ell-def-2}. Let $\varphi\in C^2$ be such that $u-\varphi$ attains a strict local minimum at $x_0$, with $D\varphi(x_0)=0$ and $D\varphi(x)\neq0$ near $x_0$.\vspace{4pt}

\textbf{Case 1:} $u$ is not constant near $x_0$. 

For any $y\in B_r(0)$ ($r>0$ small), set $\phi_y(x)\coloneqq\varphi(x+y)$. Since $x_0$ is a strict local minimum of $u-\varphi$, $u-\phi_y$ attains a local minimum at some $x_y\in B_\eta(x_0)$ with $\eta\to0$ as $r\to0$.  

If there exists a sequence $y_k\to0$ such that $x_{y_k}+y_k\neq x_0$, then $D\phi_{y_k}(x_{y_k})=D\varphi(x_{y_k}+y_k)\neq0$. By subtracting $\varepsilon|x-x_0|^2$ and letting $\varepsilon\to0$, we may assume $u-\phi_{y_k}$ attains a strict local minimum at $x_{y_k}$. By Definition \ref{vis-ell-def-2}, 
\begin{eqnarray*}
&&\limsup_{\substack{x \to x_0,\\x\neq x_0}} \bigl( - \Delta_p \varphi(x) - f(x, u, D\varphi(x)) \bigr) \\
&\geq& \limsup_{k\to\infty}\bigl(-\Delta_p\varphi(x_{y_k}+y_k)-f(x_{y_k}+y_k, u(x_{y_k}+y_k), D\varphi(x_{y_k}+y_k))\bigr)\\[4pt]
&=& \limsup_{k\to\infty}-\Delta_p \varphi_{y_k}(x_{y_k}) - f(x_0, u(x_0), D\varphi(x_0))\\[4pt]
&\geq& \limsup_{k\to\infty}f(x_{y_k},u(x_{y_k}),D\varphi_{y_k}(x_{y_k}))- f(x_0, u(x_0), D\varphi(x_0))=0.
\end{eqnarray*}
It follows that $u$ is a viscosity supersolution in the sense of Definition \ref{vis-ell-def-1}.

If no such sequence exists, then we must have $y+x_y=x_0$ for any $y$ in a perhaps smaller ball (we still denote by $r$). Moreover, for all $y\in B_r(0)$ and any $x\in B_{\eta}(x_0)$, we have
$$
u(x_0-y)-\varphi(x_0)=u(x_y)-\varphi_y(x_y)\leq u(x)-\varphi_{y}(x)=u(x)-\varphi(x+y),
$$
rearranging gives
$$
u(x_0-y)-u(x)\leq\varphi(x_0)-\varphi(x+y).
$$
For any $z\in B_{\min\{r,\eta\}}(0)$ and any unit vector $e$, choosing $x=x_0+z$, $y=-z-se$ with small $s>0$, dividing by $s$ and taking $\limsup$, then taking $x=x_0+z+se$, $y=-z$, dividing by $s$ and taking $\liminf$, we obtain $Du\equiv0$ near $x_0$, which contradicts the assumption that $u$ is not constant.\vspace{4pt}

\textbf{Case 2:}  If $u$ is constant near $x_0$. 

We claim 
$$
\liminf_{\substack{x \to x_0,\\x\neq x_0}}   \Delta_p \varphi(x)   \leq 0.
$$
If not, then $\varphi$ is $p$-subharmonic in the weak sense in $B_r(x_0)\setminus\{x_0\}$ for some $r>0$ small. By a standard approximation argument, $\varphi$ is also $p$-subharmonic in $B_r(x_0)$, so by the strong maximum principle it cannot attain a maximum in the interior unless it is constant, a contradiction. It follows that
$$
0\leq -f(x_0,u(x_0),0)\leq \limsup_{\substack{x \to x_0,\\x\neq x_0}} \bigl( - \Delta_p \varphi(x) - f(x, u, D\varphi(x)) \bigr).
$$
This completes the proof.
\end{proof}

\begin{definition}[Weak solution] A function $u \in W^{1, p}_{\text{loc}}(\Omega)$ is a weak supersolution to \eqref{ell} if
\begin{equation}
\int_\Omega|\nabla u|^{p-2}\nabla u \cdot \nabla \psi \geq \int_\Omega f(x, u, \nabla u)\psi
\end{equation}
for all nonnegative $\psi \in C_0^{\infty}(\Omega)$.  For weak subsolutions, the inequality is reversed. A function that is both weak sup- and subsolution is a weak solution.
\end{definition}

With the above notation, it is natural to ask about the relation between viscosity solutions and weak solutions; in fact, they are equivalent under various conditions on $f$. We list some of these results that will be used in the convergence of $u_{\varepsilon}\to u$. The results are stated for supersolutions, but they hold for subsolutions as well.
\begin{theorem}[Theorem 1.4 in \cite{MO2019ANA}]\label{vis to weak}
Let $1 < p < \infty$. Assume that  $f=f(x,s,\eta)$ is uniformly continuous in $\Omega \times \mathbb{R}\times \mathbb{R}^{n}$, non-increasing in $s$, and satisfies the following growth condition
\begin{eqnarray}\label{growth cond}
|f(x,s, \eta)|\leq \gamma(|s|)|\eta|^{p-1} + \phi(x),\end{eqnarray}  
where $\gamma \geq 0$ is  continuous, and $\phi \in L^{\infty}_{\text{loc}}(\Omega)$. Then, 
if $u \in L^{\infty}_{\text{loc}}(\Omega)$ is a viscosity supersolution to \eqref{ell}, it is also a weak supersolution to \eqref{ell}.
\end{theorem}

A converse of Theorem \ref{vis to weak} is stated below.

\begin{theorem}[Theorem 1.8 in \cite{MO2019ANA}]\label{weak to vis}
Let $1 < p < \infty$. Suppose that $f=f(x, s)$ is continuous in $\Omega \times \mathbb{R}$ and non-increasing in $s$.   If $u \in  W^{1, p}_{\text{loc}}(\Omega) \cap C(\Omega)$ is a weak supersolution to \eqref{ell}, then it is a viscosity supersolution to \eqref{ell}.
\end{theorem}
\begin{remark}\label{rmk1}
For general $f(x,s,\eta)$, the question of whether a weak solution is necessarily a viscosity solution is more involved. In fact, as noted in \cite{FRZ2024MA}, this implication holds if the pair $(u,f)$ fulfills the comparison principle property (CPP, Definition \ref{CPP}).  We refer to \cite{MO2019ANA, MV2023BMS, MV2024BMS, PS2007} for more information on this topic. In particular, by Theorem 3.5.1 in \cite{PS2007}, if $1<p<2$ and $f(x,s,\eta)$ is Lipschitz in $\eta$ and non-increasing in $s$, then $(u,f)$ satisfies (CPP), and hence a weak solution is a viscosity solution.
\end{remark}
\begin{definition}\label{CPP}
 Suppose that $u$ is a weak supersolution to \eqref{ell} in $\Omega' \subset \Omega$. If for any weak subsolution $v$ of \eqref{ell} such that $v \le u$ a.e. in $\partial \Omega'$ there holds that $v \le u$ a.e. in $\Omega'$, then we say that $(u, f)$ fulfills the comparison principle property (CPP) in $\Omega'$.   
\end{definition}

\subsection{Parabolic capillary-type p-Laplacian problem}
In this subsection, we indicate how to extend the notion of elliptic problems to parabolic problems.
\begin{equation}\label{par}
u_t  -\Delta_p u= f(x,u,Du),
\end{equation}
where now $u$ is a function of $(x,t)$ and $Du$, $D^2u$ mean $D_xu(x,t)$ and $D_x^2u(x,t)$.  Let $T > 0$, and $\Omega_T = \Omega\times(0,T)$. We denote by $P^{2,+}$, $P^{2,-}$ the parabolic variants of the semijets $J^{2,+}$, $J^{2,-}$; to be specific, we say that $(a,\eta,X) \in \mathbb{R} \times \mathbb{R}^N \times S(N)$ belongs to $P^{2,+}u(s,z)$ for $(s,z) \in \Omega_T$ if
$$
u(t,x) \leq u(s,z) + a(t-s) + \langle \eta,x-z \rangle + \frac{1}{2}\langle X(x-z),x-z \rangle + o(|t-s| + |x-z|^2) 
$$
as $(t,x) \to (s,z)$. Similarly, we define $P^{2,-}u = -P^{2,+}(-u)$. The corresponding definitions of $\overline{P}^{2,+}$, $\overline{P}^{2,-}$ are then clear.
 
 The following definition of the viscosity solutions for parabolic $p$-Laplacian equations was also studied by Juutinen-Lindqvist-Manfredi \cite{JLM2001SIAMMA} for $f\equiv 0$, and later extended by Siltakoski \cite{S2021JEE} to the case $f=f(Du)$, and by Fang-Zhang \cite{FZ2022MMAS} for general $f=f(x,u,Du)$.
\begin{definition}[Parabolic viscosity solution]
A continuous function $u:\Omega\times[0,\infty)\to (-\infty,+\infty]$ is a viscosity supersolution to \eqref{par}, if whenever $ \varphi \in C^2(\Omega\times[0,\infty)) $ and $(x_0, t_0) \in \Omega\times [0,\infty) $ are such that
$$
\begin{cases}
\varphi(x_0, t_0) = u(x_0, t_0), \\
\varphi(x, t) < u(x, t) & \text{when } (x, t) \neq (x_0, t_0), \\
D\varphi(x, t) \neq 0 & \text{when } x \neq x_0,
\end{cases}
$$
we have
$$
\limsup_{\substack{(x,t) \to (x_0, t_0),\\ x\neq x_0}} \bigl( \partial_t \varphi(x, t) - \Delta_p \varphi(x, t) - f(x, u, D\varphi(x, t)) \bigr) \geq 0.
$$
A continuous function $u:\Omega_T\to (-\infty,+\infty]$ is a viscosity subsolution to \eqref{par}, if whenever $ \varphi \in C^2(\Omega\times[0,\infty)) $ and $(x_0, t_0) \in \Omega\times [0,\infty)$ are such that
$$
\begin{cases}
\varphi(x_0, t_0) = u(x_0, t_0), \\
\varphi(x, t) > u(x, t) & \text{when } (x, t) \neq (x_0, t_0), \\
D\varphi(x, t) \neq 0 & \text{when } x \neq x_0,
\end{cases}
$$
we have
$$
\liminf_{\substack{(x,t) \to (x_0, t_0),\\ x\neq x_0}} \bigl( \partial_t \varphi(x, t) - \Delta_p \varphi(x, t) - f(x, u, D\varphi(x, t)) \bigr) \leq 0.
$$
A function is a viscosity solution if and only if it is both viscosity sub- and supersolution.
\end{definition}
\begin{remark}
There are at least three types of definitions for singular parabolic problems, and it seems that a simple description of the singular case $D\varphi = 0$ as in Definition \ref{vis-ell-def-2} cannot be expected. See \cite{OS1997CPDE, D2011PA} and the references therein for further details. Moreover, it was proved in \cite{JLM2001SIAMMA, D2011PA} that these definitions are equivalent when $f \equiv 0$. 
\end{remark}

\begin{definition}[Parabolic weak solution]
A function $u:\Omega\times[0,\infty)\to\mathbb{R}$ is a weak supersolution to \eqref{par} provided $u\in L^{q}(t_{1},t_{2};W^{1,q}(\Omega))$ whenever $\Omega'\times(t_{1},t_{2})\subset\subset\Omega\times[0,\infty)$, and
$$
\int_{\Omega\times[0,\infty)}-u\cdot\partial_{t}\varphi+|Du|^{p-2}Du\cdot D\varphi \,dz\geq\int_{\Omega\times[0,\infty)}\varphi f(x,u,Du)\,dz
$$
for all nonnegative functions $\varphi\in C^{\infty}_{c}(\Omega\times(t_{1},t_{2}))$. For weak subsolutions, the inequality is reversed. A function that is both weak sup- and subsolution is a weak solution.
\end{definition}
Parabolic viscosity solutions and parabolic weak solutions are also equivalent under certain conditions on $f$, and the equivalence always holds if we have a priori bounds on $u_t$ and $|Du|$.
\begin{theorem}[Theorem 3.2, Lemma 3.4 in \cite{FZ2022MMAS}]\label{thm par vis equ weak}
If $f$ satisfies the following assumptions 
\begin{itemize}
\item[(1)] $ f(x, s, \eta) $ is non-increasing with respect to $ s $ and is uniformly continuous.
\item[(2)] $ f(x, s, \eta) $ satisfies the growth condition that
$$
|f(x, s, \eta)| \leq \gamma(|s|)(1 + |\eta|^\beta) + \phi(x, t),
$$
where $ 1 \leq \beta < p, \phi \in L_{\text{loc}}^{\infty}$ and $ \gamma(\cdot) \geq 0 $ is continuous in $ \mathbb{R}_+ $.
\item[(3)] $ f(x, s, \eta) $ is locally Lipschitz continuous with respect to $ \eta $,
\end{itemize}
then the bounded viscosity solutions are the same as the bounded weak solutions with continuity to \eqref{par}.
\end{theorem}
\begin{remark}
Consider $f=\varphi(x,u)+g(x,u)|Du|^{\alpha}$, where both $\varphi,g$ are non-increasing in $u$ and $|Du|\leq C$. To ensure condition (3) in Theorem \ref{thm par vis equ weak}, we require $\alpha\geq1$.
\end{remark}

We now give the precise definition of a viscosity solution to problem \eqref{par eq}, which incorporates the capillary-type boundary condition.
\begin{definition}[Viscosity solution of the parabolic boundary problem]\label{vis-par-bdy-def-1}
A continuous function $u:\overline{\Omega}\times[0,\infty)\to (-\infty,+\infty]$ is said to be a viscosity supersolution of \eqref{par eq} if $ u(\cdot, 0) \geq u_0 $ on $\overline{\Omega}$, and whenever $ \varphi \in C^2(\overline\Omega\times[0,\infty)) $ and $(x_0, t_0) \in \Omega\times[0,\infty) $ are such that 
$$
\begin{cases}
\varphi(x_0, t_0) = u(x_0, t_0), \\
\varphi(x, t) < u(x, t) & \text{when } \overline\Omega\times[0,\infty)\ni(x, t) \neq (x_0, t_0), \\
D\varphi(x, t) \neq 0 & \text{when } x \neq x_0,
\end{cases}
$$
then
\begin{itemize}
\item if $x_0\in\Omega$,
$$
\limsup_{\substack{(x,t)\to(x_0,t_0),\\ x\neq x_0}} \bigl( \partial_t\varphi(x,t)-\Delta_p\varphi(x,t)-f(x,u,D\varphi(x,t)) \bigr) \ge 0;
$$
\item if $x_0\in\partial\Omega$,
\begin{equation}\label{sup-vis-par-bdy-cond}
\limsup_{\substack{(x,t)\to(x_0,t_0),\\ x\neq x_0}} \max\Bigl\{ \partial_t\varphi(x,t)-\Delta_p\varphi(x,t)-f(x,u,D\varphi(x,t)),\; -|D\varphi|^{q-1}\varphi_\nu-\phi(x) \Bigr\} \ge 0.
\end{equation}
\end{itemize}

A continuous function $u:\overline{\Omega}\times[0,\infty)\to (-\infty,+\infty]$ is said to be a viscosity subsolution of \eqref{par eq} if $ u(\cdot, 0) \leq u_0 $ on $\overline{\Omega}$, and whenever $ \varphi \in C^2(\overline\Omega\times[0,\infty)) $ and $(x_0, t_0) \in \Omega\times[0,\infty) $ are such that 
$$
\begin{cases}
\varphi(x_0, t_0) = u(x_0, t_0), \\
\varphi(x, t) > u(x, t) & \text{when } \overline\Omega\times[0,\infty)\ni(x, t) \neq (x_0, t_0), \\
D\varphi(x, t) \neq 0 & \text{when } x \neq x_0,
\end{cases}
$$
then
\begin{itemize}
\item if $x_0\in\Omega$,
$$
\liminf_{\substack{(x,t)\to(x_0,t_0),\\ x\neq x_0}} \bigl( \partial_t\varphi(x,t)-\Delta_p\varphi(x,t)-f(x,u,D\varphi(x,t)) \bigr) \le 0;
$$
\item if $x_0\in\partial\Omega$,
\begin{equation}\label{sub-vis-par-bdy-cond}
\liminf_{\substack{(x,t)\to(x_0,t_0),\\ x\neq x_0}} \min\Bigl\{ \partial_t\varphi(x,t)-\Delta_p\varphi(x,t)-f(x,u,D\varphi(x,t)),\; -|D\varphi|^{q-1}\varphi_\nu-\phi(x) \Bigr\} \le 0.
\end{equation}
\end{itemize}

A function is a viscosity solution if and only if it is both a viscosity subsolution and a viscosity supersolution.
\end{definition}
\begin{remark}
The boundary condition in the viscosity problem is different from the classical boundary value problems, however, it was shown in \cite{CIL1992} Proposition 7.11 that if $F$ is nondegenerate, then every classical solution is also a viscosity solution.
\end{remark}
In fact, when $x_0\in\partial\Omega$, we can remove the singularity in the definition (the same conclusion also holds for elliptic boundary problems).
\begin{lemma}\label{par-vis-bdy-def-2}
The boundary conditions \eqref{sup-vis-par-bdy-cond} need only hold for all $ \varphi \in C^2(\overline\Omega\times[0,\infty)) $ and $(x_0, t_0) \in \partial\Omega\times[0,\infty)$ such that 
$$
\begin{cases}
\varphi(x_0, t_0) = u(x_0, t_0), \\
\varphi(x, t) < u(x, t) & \text{when } \overline\Omega\times[0,\infty)\ni(x, t) \neq (x_0, t_0), \\
D\varphi(x_0, t_0) \neq 0,
\end{cases}
$$
\end{lemma}
\begin{proof}
When $p\geq2$, there is nothing to prove since the singularity does not occur, so we only prove the result for $1<p<2$. 
    
For any $\varphi\in C^2$ and $x_0\in\partial\Omega$ such that $u-\varphi$ attains a strict local minimum (among $\overline{\Omega}\times[0,\infty)$) at $(x_0,t_0)$ with $D\varphi(x_0,t_0)=0$ and $D\varphi(x,t)\neq0$ for $x\neq x_0$, we aim to prove
$$
\limsup_{\substack{\overline\Omega\times[0,\infty)\ni(x,t) \to (x_0, t_0),\\ x\neq x_0}} \max\Big\{ \partial_t \varphi(x, t) - \Delta_p \varphi(x, t) - f(x, u, D\varphi(x, t)), -|D\varphi|^{q-1}\varphi_{\nu}-\phi(x) \Big\} \geq 0;
$$

We denote $\phi_y(x,t)\coloneqq\varphi((x,t)+y)$, and choose $y_k=(\frac{1}{k}\nu(x_0),0)$ for $k>0$ large enough. We have $u(x,t)-\phi_{y_k}(x,t)$ attains its local minimum at some $(x_y,t_y)\in (B_{\eta}(x_0)\cap\overline\Omega)\times (t_0-\eta,t_0+\eta)$ with $\eta\to0$ as $r\to0$. Clearly $x_{y_k}+\frac{1}{k}\nu(x_0)\neq x_0$ for $k$ sufficiently large. Therefore,  $D\phi_{y_k}(x_{y_k},t_{y_k})=D\varphi((x_{y_k}+\frac{1}{k}\nu(x_0)),t_k)\neq0$. By a similar argument to that in Lemma \ref{vis-ell-def-equal}, we obtain the result.
\end{proof}

\subsection{Stability and comparison principle}
We next discuss the stability of viscosity solutions to $p$-Laplacian equations. The stability results for degenerate equations are well understood; see for instance \cite{BCESS1997,CIL1992}. For the $p$-Laplacian equations, an unbounded singularity occurs when $1<p<2$, and it is not known whether stability results hold for general $f$. However, we can pass to the limit using the equivalence of weak and viscosity solutions together with gradient estimates, as in the following lemma. When $p\ge 2$, the exclusion of $Du=0$ in the definition of viscosity solutions does not affect stability, as in the degenerate case. We only need to note that for any $p\neq0$, we can choose $p_n\neq0$ such that $(x_n, u_n(x_n), p_n, X_n) \to (z, u(z), p, X)$. Then, following the proof of Theorem 8.3 in \cite{BCESS1997} (page 21), we obtain stability.

\begin{lemma}[Stability]\label{lemstb} For $1<p<2$, define $F_{\varepsilon}[u]\coloneqq F_{\varepsilon}(Du,D^2u)\coloneqq\operatorname{div}((\varepsilon^2+|Du|^2)^{\frac{p-2}{2}}Du)$ in $\Omega$, and let $u_\varepsilon\in C^2$ be the solutions of $F_{\varepsilon}[u]=f(x,u,Du)$. Suppose $u_0\coloneqq\lim_{\varepsilon\to 0} u_\varepsilon$ is finite, and additionally, $|Du_{\varepsilon}|\leq C$ and $(u,f)$ satisfies (CPP). Then $u_0$ is a viscosity solution of $F[u]=f(x,u,Du)$ in $\Omega$. Moreover, stability for the parabolic $p$-Laplacian equations can be proved similarly provided that $Du$ and $u_t$ are bounded.
\end{lemma}
\begin{proof}
By classical regularity theory for divergence-type quasilinear equations, for any $\Omega'\subset\subset\Omega$ we have $|u_{\varepsilon}|_{C^{1,\alpha_0}(\Omega')}\le C$ for some $0<\alpha_0<1$ and some $C$ independent of $\varepsilon$. Consequently, $u_0$ is a weak solution of $F[u]=f(x,u,Du)$, and hence a viscosity solution by Remark \ref{rmk1} or Theorem \ref{thm par vis equ weak}.
\end{proof}

Finally, for $p\ge 2$, we state the comparison principle for viscosity solutions of degenerate parabolic $p$-Laplacian equations; see \cite{B1993JDE, B1999JDE, BC2026NODEA, GS1993DIE, IL1990JDE} for more details.
\begin{lemma}[Comparison Principle]\label{lem cpp}
Assume $\Omega\subset \mathbb{R}^n$ is a bounded $C^{2,\beta}$ domain, $p\geq2$, $f$ satisfies \eqref{f}, and $\phi$ satisfies \eqref{phi0}. Let continuous functions $u_1,u_2$ be a viscosity subsolution and a viscosity supersolution to \eqref{par eq}, respectively, with $\operatorname{Lip}u_1(\cdot,t)$ and $\operatorname{Lip}u_2(\cdot,t)$ bounded (independently of $t$).
\begin{itemize}
\item If $q=0$ or $q=1$, and $u_1(x,0)\leq u_2(x,0)$ on $\overline{\Omega}$, then $u_1(x,t)\leq u_2(x,t)$ on $\overline{\Omega}\times [0,T)$ for any $T>0$.
\item If $q>0$ and $|\phi(x)|>0$ on $\partial\Omega$, then the above comparison result holds.
\end{itemize}
\end{lemma}
\begin{proof}
We may assume $p>2$. The proof is divided into 4 steps. The first three steps are for $q=0$ or $q=1$, and the final step handles the general case $q>0$ with the condition $|\phi(x)|>0$ on $\partial\Omega$.

\textbf{Step 1.} For $q=0$ or $q=1$, let $u$ be a continuous viscosity subsolution or supersolution to \eqref{par eq}. We aim to prove that it satisfies the boundary condition in the strong viscosity sense; i.e., $u$ is a viscosity subsolution or supersolution of the boundary condition.

We only prove the case $q=1$ for a subsolution $u$, the other cases are similar. We argue by contradiction: suppose $u$ is not a viscosity subsolution of $-u_\nu-\phi(x)=0$ at some $(x_0,t_0)\in\partial\Omega\times[0,\infty)$. Then there exists a $\varphi\in C^2$ such that $u-\varphi$ has a local maximum at $(x_0,t_0)$, and 
$$
-\varphi_{\nu}(x_0,t_0)-\phi(x_0,t_0)>0.
$$  
Consider 
$$
\tilde{\varphi}(x,t)\coloneqq\varphi(x,t)+a\operatorname{d}(x)-b\operatorname{d}(x)^2,
$$
where $\operatorname{d}(x)\coloneqq\operatorname{dist}(x,\partial\Omega)$ and $a,b>0$.

Then $u-\tilde{\varphi}$ also has a local maximum point at $x_0$, and $-\tilde{\varphi}_{\nu}(x_0)-\phi(x_0)>0$ provided that $a>0$ is sufficiently small. Moreover, we can choose $a$ appropriately so that $|D\tilde{\varphi}(x_0,t_0)|=|D\varphi(x_0,t_0)+a\nu(x_0)|\geq C_a>0$ (from now on $a$ is fixed). Since $u$ is a viscosity subsolution, writing $\Delta_p u=|Du|^{p-2}G(Du,D^2u)$, we have at $(x_0,t_0)$,
$$
\partial_t{\varphi}-|D\tilde{\varphi}|^{p-2}G\Bigl(D\tilde{\varphi}, D^2\varphi+aD^2\operatorname{d}-2bD\operatorname{d}\otimes D\operatorname{d}\Bigr)\leq f\leq L.
$$
Because $G$ is uniformly elliptic, this cannot happen if $b$ is sufficiently large. Therefore, $u$ is the viscosity subsolution of the boundary condition.\vspace{4pt}

\textbf{Step 2.} For $q=0$ or $q=1$, according to Theorem 3.1 in \cite{B1999JDE}, we first verify condition (H5-1) for bounded $\xi$. That is, for any $R,K>0$, there exists a function $m_{R,K}:\mathbb{R}^+\to\mathbb{R}$ such that $m_{R,K}(t)\to0$ as $t\to0$ and, for all $\eta>0$,
\begin{align}
&\quad F(y,u,\xi_2,Y)-F(x,u,\xi_1,X) \notag\\
&\qquad \leq m_{R,K}\left( \eta + |x-y|(1+\max\{|\xi_1|,|\xi_2|\}) + \frac{|x-y|^2}{\varepsilon^2} \right)\label{B1999JDE-1}
\end{align}
for any $x,y\in\overline{\Omega}$, $\xi_1,\xi_2\in\mathbb{R}^N$, $|\xi_1|+|\xi_2|\leq A$ for some large $A$, $|u|\leq R$, and for any matrices $X,Y\in\mathcal{S}^N$ satisfying the following properties:
\begin{align}
-\frac{K}{\varepsilon^2}\mathrm{Id} &\leq 
\begin{pmatrix}
X & 0 \\
0 & -Y
\end{pmatrix}
\leq \frac{K}{\varepsilon^2} 
\begin{pmatrix}
\mathrm{Id} & -\mathrm{Id} \\
-\mathrm{Id} & \mathrm{Id}
\end{pmatrix}
+ K\eta\,\mathrm{Id}, \label{B1999JDE-2}\\[4pt]
|\xi_1-\xi_2| &\leq K\eta\varepsilon(1+\min\{|\xi_1|,|\xi_2|\}), \quad |x-y| \leq K\eta\varepsilon,\label{B1999JDE-3}
\end{align}
where $F=-\Delta_pu$.

Rewrite $F$ as 
$$
F(\xi,X)=-\operatorname{Tr}\left(B(\xi)B^{T}(\xi)X\right),
$$\vspace{4pt}
where $B(\xi)\coloneqq|\xi|^{\frac{p-2}{2}}\left(\mathrm{Id}+(\sqrt{p-1}-1)\frac{\xi\xi^T}{|\xi|^2}\right)$.\vspace{4pt}

Multiplying the rightmost inequality in \eqref{B1999JDE-2} by the positive semidefinite matrix
\begin{align}
\begin{pmatrix}
B^T(\xi_1)B(\xi_1) & B^T(\xi_2)B(\xi_1) \\[10pt]
B^T(\xi_1)B(\xi_2) & B^T(\xi_2)B(\xi_2)
\end{pmatrix},
\end{align}
we have
\begin{align}
&\begin{pmatrix}
B^T(\xi_1)B(\xi_1)X & -B^T(\xi_2)B(\xi_1)Y\\[10pt]
B^T(\xi_1)B(\xi_2)X & -B^T(\xi_2)B(\xi_2)Y
\end{pmatrix} \notag \\[10pt]
&\qquad \leq \frac{K}{\varepsilon^2}
\begin{pmatrix}
B^T(\xi_1)B(\xi_1)-B^T(\xi_2)B(\xi_1) & -B^T(\xi_1)B(\xi_1)+B^T(\xi_2)B(\xi_1)\\[10pt]
B^T(\xi_1)B(\xi_2)-B^T(\xi_2)B(\xi_2) & -B^T(\xi_1)B(\xi_2)+B^T(\xi_2)B(\xi_2)
\end{pmatrix} \notag \\[10pt]
&\qquad\quad + K\eta
\begin{pmatrix}
B^T(\xi_1)B(\xi_1) & B^T(\xi_2)B(\xi_1) \\[10pt]
B^T(\xi_1)B(\xi_2) & B^T(\xi_2)B(\xi_2)
\end{pmatrix}.
\end{align}
Taking traces yields
\begin{align}
F(\xi_2,Y)-F(\xi_1,X)\leq& \frac{K}{\varepsilon^2}\operatorname{Tr}\Bigl(\left(B^T(\xi_1)-B^T(\xi_2)\right)\left(B(\xi_1)-B(\xi_2)\right)\Bigr)\notag\\[2pt]
&+K\eta\operatorname{Tr}\Bigl(B(\xi_1)^T B(\xi_1)-B(\xi_2)^T B(\xi_2)\Bigr).\label{Fxi2-Fxi1}
\end{align}
For $p>2$, since $|\xi_1|+|\xi_2|\leq A$, it is not hard to see that the right‑hand side of \eqref{Fxi2-Fxi1} tends to $0$ as $\eta\to0$. Hence (H5-1) is satisfied for bounded $\xi_1,\xi_2$.\vspace{4pt}

\textbf{Step 3.}
For $q=0$ or $q=1$, define $\tilde{F}(\xi,X)$ such that $\tilde{F}(\xi,X)=F(\xi,X)$ for $|\xi|\leq A$, and such that $\tilde{F}$ satisfies condition (H5-1) and is degenerate elliptic for all $\xi\in\mathbb{R}^n$.

We claim that if $\operatorname{Lip}(u)\leq\frac{A}{2}$ and $u$ is a viscosity subsolution or supersolution of 
$$
\left\{\begin{aligned}
& u_t=F(Du,D^2u)+f(x,u,Du) &&\text{in~}\Omega\times[0,\infty),\\
& -u_\nu-\phi(x)=0 &&\text{on~}\partial\Omega\times[0,\infty),
\end{aligned}\right.
$$
then $u$ is also a viscosity subsolution or supersolution of 
$$
\left\{\begin{aligned}
& u_t=\tilde{F}(Du,D^2u)+f(x,u,Du) &&\text{in~}\Omega\times[0,\infty),\\
& -u_\nu-\phi(x)=0 &&\text{on~}\partial\Omega\times[0,\infty).
\end{aligned}\right.
$$
Indeed, for any $x_0\in\Omega$ and any $\varphi\in C^2$ touching $u$ from above or below at $(x_0,t_0)$, we have $|D\varphi(x_0)|\le A/2$. Hence $F(D\varphi(x_0),D^2\varphi(x_0))=\tilde{F}(D\varphi(x_0),D^2\varphi(x_0))$. Together with the fact that $u$ is a strong viscosity subsolution or supersolution of the boundary condition, the claim follows.

Therefore, under the assumptions of the lemma, we may apply the comparison principle for $\tilde{F}$ (see Theorem 3.1 and the discussion on page 203 of \cite{B1999JDE}) to conclude $u_1(x,t)\le u_2(x,t)$.\vspace{4pt}

\textbf{Step 4.} For general $q>0$, denote $B(x,\xi)\coloneqq-|\xi|^{q-1}\nu(x)\cdot \xi-\phi(x)$, and define a new boundary operator $\tilde{B}(x,\xi)$ such that $\tilde{B}(x,\xi)=B(x,\xi)$ for $x\in\partial\Omega$ and for $\xi$ satisfying $A^{-1}\leq|\xi|\leq A$ . Moreover, $\tilde{B}$ satisfies the strictly oblique condition (H1) in \cite{B1999JDE}:
$$
\tilde{B}(x,\xi-\lambda\nu)-\tilde{B}(x,\xi)\geq c_1\lambda \quad\text{for all } \lambda>0,\ \xi\in\mathbb R^n,
$$
for some $c_1=c_1(A)>0$, and is Lipschitz in $x$ and $\xi$ (condition (H2-1) in \cite{B1999JDE}):
$$
|\tilde{B}(x,\xi_1)-\tilde{B}(y,\xi_2)|\leq C_2(|x-y|+|\xi_1-\xi_2|) \quad\text{for all } x,y\in\partial\Omega,  ~\xi_1,\xi_2\in\mathbb R^n,
$$
for some $C_2=C_2(A)>0$.

We claim that if $|\phi(x)|\geq c>0$ on $\partial\Omega$, $A$ is sufficiently large, $\operatorname{Lip}(u)\leq\frac{A}{2}$, and $u$ is a viscosity subsolution or supersolution of 
\begin{equation}\label{F-B}
\left\{\begin{aligned}
& u_t=F(Du,D^2u)+f(x,u,Du) &&\text{in~}\Omega\times[0,\infty),\\
& B(x,Du)=0 &&\text{on~}\partial\Omega\times[0,\infty),
\end{aligned}\right.
\end{equation}
then $u$ is also a viscosity subsolution or supersolution of 
\begin{equation}\label{tildeF-B}
\left\{\begin{aligned}
& u_t=\tilde{F}(Du,D^2u)+f(x,u,Du) &&\text{in~}\Omega\times[0,\infty),\\
& \tilde{B}(x,Du)=0 &&\text{on~}\partial\Omega\times[0,\infty).
\end{aligned}\right.
\end{equation}
In view of \textbf{Step 3}, we only need to consider the boundary case. Suppose $u$ is a viscosity supersolution of \eqref{F-B}. For any $\varphi\in C^2$ touching $u$ from below at $(x_0,t_0)\in\partial\Omega\times[0,\infty)$, write $D\varphi(x_0,t_0)=a\nu(x_0)+\mathbf{b}$, where $\mathbf{b}$ is tangential. Then clearly $|\mathbf{b}|\leq\frac{A}{2}$ and $a\leq\frac{A}{2}$.

On the other hand, since $\tilde{B}$ is strictly oblique, we have
$$
\tilde{B}(x_0,D\varphi(x_0,t_0))\geq \tilde{B}(x_0,-\frac{A}{2}\nu(x_0)+\mathbf{b})=B(x_0,-\frac{A}{2}\nu(x_0)+\mathbf{b})\geq0
$$
if $a\leq -\frac{A}{2}$ and $A$ is sufficiently large. Hence we may assume $|D\varphi(x_0,t_0)|\leq A$.

Since $u$ is a viscosity supersolution of \eqref{F-B}, we have
$$
\limsup_{\substack{(x,t)\to(x_0,t_0),\\ x\neq x_0}} \max\Bigl\{ \partial_t\varphi(x,t)-F(D\varphi,D^2\varphi)-f(x,u,D\varphi(x,t)),\; B(x,D\varphi) \Bigr\} \ge 0.
$$
\begin{itemize}
\item If
$$
\limsup_{\substack{(x,t)\to(x_0,t_0),\\ x\neq x_0}} \partial_t\varphi(x,t)-\Delta_p\varphi(x,t)-f(x,u,D\varphi(x,t))\geq0,
$$
then since $\tilde{F}(\xi,X)=F(\xi,X)$ for $|\xi|\leq A$, we conclude that $u$ is a viscosity supersolution of \eqref{tildeF-B}.\\
\item If $B(x,D\varphi)\geq0$, we may assume $|D\varphi|\leq A^{-1}$; otherwise 
$$
\tilde{B}(x,D\varphi)=B(x,D\varphi)\geq0,
$$
hence $u$ is a viscosity supersolution of \eqref{tildeF-B}. Since $|\phi(x)|\geq c>0$, $|D\varphi|\leq A^{-1}$, and $B(x,D\varphi)>0$, we have $\phi(x_0)<-c<0$. 

Let $a_0>0$ be the small constant such that $|D\varphi(x_0)+a_0\nu(x_0)|=A^{-1}$. Therefore,
$$
\tilde{B}(x_0,D\varphi)\geq\tilde{B}(x_0,D\varphi+a_0\nu)=B(x_0,D\varphi+a_0\nu)\geq0,
$$
and $u$ is also a viscosity supersolution of \eqref{tildeF-B}.
\end{itemize}
Thus, applying the comparison principle for $\tilde{F}$ and $\tilde{B}$ yields $u_1(x,t)\le u_2(x,t)$.
\end{proof}

\subsection{Basic calculations}
At the end of this section, we introduce some notation and perform preliminary calculations that will be used repeatedly in the following sections.

$\bullet$ Defining functions.

When $\Omega$ is a strictly convex domain with $C^{2,\beta}$ boundary, there exists a $C^{2,\beta}$ defining function $h$ for $\Omega$ satisfying
$$
h < 0,\quad \{h_{ij}\} \ge \kappa_0 \{\delta_{ij}\} \quad \text{in } \Omega,
$$
and
$$
h = 0,\quad -h_\nu = |Dh| = 1 \quad \text{on } \partial\Omega.
$$
Moreover, the principal curvature matrix on $\partial\Omega$ satisfies $\{\kappa_{ij}\} \ge \kappa_0 \{\delta_{ij}\}_{1\le i,j\le n-1}$ for some positive constant $\kappa_0$.

$\bullet$ Coordinate selection.

Throughout this paper, we denote
$$
v(Du) \coloneqq \sqrt{\varepsilon^2 + |Du|^2}, \qquad 
\tilde{v}(Du) \coloneqq \sqrt{\varepsilon^2 + (p-1)|Du|^2},
$$
$$
a^{ij}(Du) \coloneqq v^{p-2}\delta_{ij} + (p-2)v^{p-4}u_i u_j.
$$
When no ambiguity arises, we omit the vector in parentheses. Under this notation, the approximate $p$-Laplacian operator reduces to $a^{ij}u_{ij}$.

In the following sections, when we compute at a fixed point $(x_0,t_0)$ (or $x_0$) lying in the interior of $\Omega\times[0,\infty)$ (or $\Omega$), we always choose coordinates such that $|Du| = u_1$ and the matrix $\{u_{\alpha\beta}\}_{\alpha,\beta>1}$ is diagonal at that point. Consequently,
$$
a^{11}(x_0,t_0) = v^{p-4}\tilde{v}^2,\qquad a^{1\alpha}(x_0,t_0) = 0 \quad (\alpha \geq 2), \qquad 
a^{\alpha\alpha}(x_0,t_0) = v^{p-2} \quad (\alpha\geq2).
$$
Furthermore, 
\begin{equation}\label{aijkuij}
a^{ij}_{~,k}u_{ij}=(p-2)v^{p-4}u_{1}u_{1k}\Delta u+(p-2)(p-4)v^{p-6}u_{1}^{3}u_{11}u_{1k}+2(p-2)v^{p-4}u_{1}\sum_{i}u_{1i}u_{ik}.
\end{equation}

\section{Lipschitz regularity}

In this section, we study the Lipschitz regularity of solutions to the approximate problem \eqref{app par eq} for fixed $\varepsilon>0$. We first focus on the estimate for $u_t$.
\begin{lemma}\label{lem3.1}
Let $u$ be a solution of \eqref{app par eq}. Then for any finite $T>0$, we have
$$
\sup_{\overline{\Omega}\times[0,T]}|u_t|=\sup_{\overline{\Omega}\times\{0\}}|u_t|,
$$
and hence 
$$
\sup_{\overline{\Omega}\times[0,T]}|u(x,t)|\leq \sup_{\overline{\Omega}}|u_0(x)| + CT,
$$
where $C$ is a constant independent of $T$ and $\varepsilon$.
\end{lemma}
\begin{proof}
Rewrite \eqref{app par eq} as  
$$u_t=a^{ij}u_{ij}+f(x,u).$$
Then
$$u_{tt}=a^{ij}u_{ijt}+a^{ij}_{~,t}u_{ij}+f_u u_t,$$
where
\begin{eqnarray*}
a^{ij}_{~,t}&=&(p-2)v^{p-4}u_mu_{mt}\delta_{ij}+(p-2)v^{p-4}(u_{it}u_j+u_iu_{jt})+(p-2)(p-4)v^{p-6}u_mu_{ml}u_iu_j\\
&=&(p-2)\frac{u_m u_{mt}}{v^2}a^{ij}-2(p-2)\frac{u_mu_{mt}}{v^2}v^{p-2}\frac{u_iu_j}{v^2}+(p-2)\frac{u_{it}u_j+u_iu_{jt}}{v^2}v^{p-2}.
\end{eqnarray*}
Let $\psi=\pm u_t$; then we have
\begin{eqnarray*}
\psi_t&=&a^{ij}\psi_{ij}+(p-2)v^{-2}(a^{ij}-2v^{p-4}u_iu_j)u_{ij}u_m\psi_m+2(p-2)v^{p-4}u_{ij}u_j\psi_i+f_u \psi\\
&=&a^{ij}\psi_{ij}+b_i\psi_i+f_u \psi
\end{eqnarray*}
for some bounded coefficients $b_i$. Since $f_u\leq 0$, the parabolic maximum principle tells us that the maximum of $u_t$ is attained on the parabolic boundary $\Omega\times\{0\}\cup \partial\Omega\times[0,T]$.

If the maximum occurs at $(x_0,t_0) \in \partial\Omega \times [0,T]$, then by Hopf's lemma we have $u_{t\nu}(x_0,t_0) < 0$. 
Choose coordinates in $\mathbb{R}^n$ such that the positive $x_n$-axis is the interior normal direction to $\partial\Omega$ at $(x_0,t_0)$. 
Differentiating the boundary condition $u_\nu = -\phi_{\varepsilon}(x) v^{1-q}$ with respect to $t$ yields
$$
u_{t\nu} = (1-q) u_\nu^2 v^{-2} u_{t\nu},
$$
which is a contradiction since $q>0$. 
Hence the maximum of $u_t$ occurs at $t_0 = 0$, and therefore $|u_t| \leq C$.

\end{proof}

We now give the crucial gradient estimate. In the following theorem, the small constant $c>0$ and the large constant $C>0$ are independent of $\varepsilon$, $T$, $|u|_{C^0}$, and may change from line to line. 
\begin{theorem}\label{thm3.2}
Suppose $\Omega\subset \mathbb{R}^n$ is a bounded strictly convex $C^{2,\beta}$ domain and $q>0$. If $f_{u}(x,u)\leq 0$, then the solution of \eqref{app par eq} satisfies
$$
\sup_{\overline{\Omega}\times[0,T]}|Du|\leq C(n,p,q,L,\Omega).
$$
\end{theorem}
\begin{proof}
By abuse of notation, we write $\phi$ in place of $\phi_{\varepsilon}$, since they enjoy similar properties. Following the ideas of \cite{MWW2018JFA} and \cite{WWX2019CPAA}, we choose auxiliary function
$$
\Phi(x,t)=\log w +\alpha h,
$$
where
$$
w=v^{q+1}-(q+1)\sum_{l=1}^n \phi u_l h_l.
$$

Suppose $\Phi(x,t)$ attains its maximum at $(x_0,t_0)\in \overline{\Omega}\times[0,T]$.

Without loss of generality, we may assume $t_0>0$ and $|Du(x_0,t_0)|$ is large enough. We further use the notation $A\sim B$ if $c\leq \frac{A}{B}\leq C$ for $|Du|$ large enough, and $A\approx B$ if $\frac{A}{B}\to 1$ as $|Du|\to \infty$. With these notation, we have
$$
|Du|\approx v,\quad w\approx v^{q+1},\quad a^{11}\sim a^{ii}\sim v^{p-2}. 
$$

\textbf{Case 1:} $x_0\in \partial\Omega$.

The proof for the boundary case follows a similar approach to Theorem 2.2 in \cite{WWX2019CPAA}, so we only sketch the argument here.

We choose coordinates in $\mathbb{R}^n$ such that the positive $x_n$-axis is the interior normal direction to $\partial \Omega$ at $x_0$. More specifically, $u_n$ denotes the unit inner normal derivative and $u_i$, $1 \leq i \leq n-1$ denote the $n-1$ tangential derivatives of $u$ on the boundary.  $D$ denotes the derivative in $\mathbb{R}^n$ and $\nabla$ denotes the tangential derivative on the boundary. We temporarily write $\nabla_i(u_n) \coloneqq u_{ni}$ for $1 \leq i \leq n-1$. Thus $|D'u|^2 = \sum_{k=1}^{n-1} u_k^2$. By the Gauss-Weingarten formula,  
$$
D_{kn}u = u_{nk} + \kappa_{kj}u_j,
$$
where $\kappa_{kj}$ is the curvature matrix of the boundary.

For $i\neq n$, from 
$$0=\frac{w\Phi_i}{1+q}=v^{q}v_{i}+\phi_{i}u_{n}+\phi u_{ni}-\phi\sum_{l=1}^{n}u_{l}h_{li}$$
and the tangential derivative of boundary condition
$$
u_{ni}=-\phi_iv^{1-q}-(1-q)\phi v^{-q}v_i,
$$
we can solve $u_{ni}$ as
$$u_{ni}=\frac{1}{1+(q-1)v^{-2q}\phi^{2}}\bigg{[}\phi(q-1)v^{-2q}(-\phi_{i}u_{n}+\phi\sum_{l=1}^{n}u_{l}h_{li})-v^{1-q}\phi_{i}\bigg{]}.$$
Hence
\begin{eqnarray}
0&\geq& \Phi_{\nu}(x_{0},t_{0})
= \frac{q+1}{w}v^{q-1}\sum_{k=1}^{n-1}u_{k}(u_{nk}+\sum_{i=1}^{n-1}u_{i}{\kappa}_{ik})+\frac{q+1}{w}[\phi_{n}u_{n}-\phi\sum_{l=1}^{n}u_{l}h_{ln}]-\alpha\notag\\
&=& \frac{q+1}{w}\cdot\frac{v^{q-1}}{1+(q-1)v^{-2q}\phi^{2}}\sum_{k=1}^{n-1}u_{k}\big{[}\phi(q-1)v^{-2q}(\phi\sum_{l=1}^{n}u_{l}h_{lk}-\phi_{k}u_{n})-v^{1-q}\phi_{k}\big{]}\notag\\
&&+\frac{q+1}{w}v^{q-1}\sum_{k,i=1}^{n-1}u_{k}u_{i}\kappa_{ik}+\frac{q+1}{w}[\phi_{n}u_{n}-\phi\sum_{l=1}^{n}u_{l}h_{ln}]-\alpha\label{bdy est}\\
&\geq&\frac{(q+1)\kappa_0}{2}-\alpha-Cv^{-q}.\notag
\end{eqnarray}
It follows that $|Du(x_0,t_0)|\leq C$ if we select $\alpha>0$ small enough.\vspace{4pt}

\textbf{Case 2:} $x_0\in\Omega$.

We choose coordinates such that $|Du|=u_1$ and the matrix $\{u_{\alpha\beta}\}_{\alpha,\beta>1}$ is diagonal at $(x_0,t_0)$. Since $\Phi$ attains its maximum at this point, we have
\begin{eqnarray*}
0&\leq& \Phi_t(x_0,t_0)=(q+1)w^{-1}\sum_{k=1}^n(v^{q-1}u_k-\phi h_k)u_{tk},\\ 
0&=&\Phi_i(x_0,t_0)=w^{-1}w_i+\alpha h_i,
\end{eqnarray*}
and
\begin{eqnarray*}
0\geq \Phi_{ii}(x_0,t_0)=w^{-1}w_{ii}-\left|\frac{w_i}{w}\right|^2+ \alpha h_{ii}.
\end{eqnarray*}
Hence
\begin{eqnarray*}
0&\geq&a^{ii}\Phi_{ii}(x_0,t_0)-\Phi_t(x_0,t_0)\\
&=&w^{-1}a^{ii}w_{ii}-a^{ii}\left|\frac{w_i}{w}\right|^2+\alpha a^{ii}h_{ii}-(q+1)w^{-1}\sum_{k=1}^n(v^{q-1}u_k-\phi h_k)u_{tk}\\
&=&w^{-1}a^{ii}w_{ii}+A a^{ii}\left|\frac{w_i}{w}\right|^2-(q+1)w^{-1}\sum_{k=1}^n(v^{q-1}u_k-\phi h_k)u_{tk}\\
&&+\alpha a^{ii}h_{ii}-(A+1)\alpha^{2}a^{ii}|h_i|^2\\[4pt]
&=& \mathrm{I}+\mathrm{II},
\end{eqnarray*}
where $\mathrm{I}$ denotes the three terms in the third line, and $\mathrm{II}$ denotes the two terms in the fourth line. The large constant $A$ is independent of $\varepsilon$, $T$, $|u|_{C^0}$, and will be chosen later.

By direct computation, we have
$$
w_i=(q+1)v^{q-1}\sum_k u_k u_{ki}-(q+1)\phi_i \sum_k u_k h_k-(q+1)\phi \sum_k u_{ki} h_k-(q+1)\phi \sum_k u_k h_{ki},
$$
and
\begin{eqnarray*}
w_{ii}&=&(q+1)(q-1)v^{q-3}\left(\sum_k u_k u_{ki}\right)^2+(q+1)v^{q-1}\sum_{k} u_{ki}^2+(q+1)v^{q-1}\sum_k u_k u_{kii}\\
&&-(q+1)\phi_{ii}\sum_k u_k h_k-2(q+1)\phi_i\sum_k u_{ki} h_k-2(q+1)\phi_{i}\sum_k u_k h_{ki}\\
&&-(q+1)\phi \sum_k u_{kii} h_k-2(q+1)\phi \sum_k u_{ki} h_{ki}-(q+1)\phi \sum_k u_k h_{kii}.
\end{eqnarray*}
Denote $M\coloneqq\max|D^2u(x_0,t_0)|$ and $M_1\coloneqq\max_i|u_{1i}(x_0,t_0)|$. Note that 
\begin{eqnarray*}
w_i&=&(q+1)\sum_k \left(v^{q-1}u_k-\phi h_k\right)u_{ki}-(q+1)\left(\phi_i h_1+\phi h_{1i}\right)u_1\\
&=&(q+1) \left(v^{q-1}u_1-\phi h_1\right)u_{1i}-(q+1)\sum_{k\neq1}\phi h_k u_{ki}-(q+1)\left(\phi_i h_1+\phi h_{1i}\right)u_1,
\end{eqnarray*}
hence 
$$|w_i|^2\geq \frac{1}{2}v^{2q}u_{1i}^2-C\sum_k u_{ki}^{2}-Cv^2,$$
and
\begin{equation}\label{w-2aiiwi2}
w^{-2}a^{ii}|w_i|^2\geq cv^{p-4}M_1^{2} -Cv^{p-2q-4}M^2-Cv^{p-2q-2}.
\end{equation}
Now compute the remaining part of $\mathrm{I}$.
\begin{eqnarray*}
&&a^{ii}w_{ii}-(q+1)\sum_k v^{q-1}\left(u_k-\phi h_k\right)u_{tk}\\
&\geq&c v^{q+p-3}M^2+(q+1)\sum_k\left(v^{q-1}u_k-\phi h_k\right)\left(a^{ij}u_{ijk}-u_{tk}\right)-Cv^{p-2}M-Cv^{p-1}\\
&\geq& c v^{q+p-3}M^2-Cv^{p-1}+(q+1)\sum_k\left(v^{q-1}u_k-\phi h_k\right)\left(-f_k-f_u u_k-a^{ij}_{~,k}u_{ij}\right),
\end{eqnarray*}
where we have used differentiating equation \eqref{app par eq}.Recall that $f_u \leq 0$; from the computation \eqref{aijkuij} and the Cauchy inequality, we further have
\begin{eqnarray*}
&&a^{ii}w_{ii}-(q+1)\sum_k v^{q-1}\left(u_k-\phi h_k\right)u_{tk}\\
&\geq& c v^{q+p-3}M^2-C\left(v^{p-1}+v^{q}\right)-Cv^q\left|a^{ij}_{~,1}u_{ij}\right|-C\sum_{k\neq1}\left|a^{ij}_{~,k}u_{ij}\right|\\
&\geq& c v^{q+p-3}M^2-C\left(v^{p-1}+v^{q}\right)-Cv^{q+p-3}M_1^2-Cv^{p-3}M_1M\\[8pt]
&\geq& c v^{q+p-3}M^2-C\left(v^{p-1}+v^{q}\right)-Cv^{q+p-3}M_1^2.\label{aiiwii}
\end{eqnarray*}
It follows that
\begin{equation}\label{w-1aiiwii}
w^{-1}a^{ii}w_{ii}-(q+1)w^{-1}\sum_k v^{q-1}\left(u_k-\phi h_k\right)u_{tk}\geq c v^{p-4}M^2-Cv^{p-4}M_1^2-C\left(v^{p-q-2}+v^{-1}\right).
\end{equation}
Recalling \eqref{w-2aiiwi2}, for $A$ sufficiently large we obtain
\begin{equation}\label{I}
    \mathrm{I}\geq -C\left(v^{p-q-2}+v^{-1}\right).
\end{equation}
Finally, for $\alpha>0$ small enough, the strict convexity of $h$ implies
\begin{equation}\label{II}
    \mathrm{II}\geq cv^{p-2}.
\end{equation}
Combining \eqref{I} and \eqref{II} yields
$$
0\geq a^{ii}\Phi_{ii}(x_0,t_0)-\Phi_t(x_0,t_0)\geq cv^{p-2}-C\left(v^{p-2-q}+v^{-1}\right),
$$
Since $q>0$ and $p>1$, we obtain $|Du(x_0,t_0)|\leq C$.

Combining with \textbf{Case 1}, we conclude
$$
\sup_{\overline{\Omega}\times[0,T]}|Du|\leq C(n,p,q,L,\Omega).
$$
\end{proof}
The proof of Theorem \ref{thm1.2} is standard.
\begin{proof}[\textbf{Proof of Theorem \ref{thm1.2}.}] 
By the crucial gradient estimate in Theorem \ref{thm3.2}, equation \eqref{app par eq} is uniformly parabolic for any fixed $T>0$. Hence long time existence and uniqueness follow from the standard parabolic theory.
\end{proof}

\section{Global H\"{o}lder regularity of $Du$}
In this section, we prove the $\varepsilon$-independent global H\"{o}lder regularity of the spatial gradient $Du$ for the solution of \eqref{app par eq} with $q=1$, i.e., the Neumann boundary condition.

We first discuss the interior regularity (these results are already known) and then turn to the boundary. The proof is inspired by \cite{JJ2012CPDE} and \cite{BKO2025CVPDE}, see also \cite{BRS2026NA, BV2022PA, CC1995, IS2013AIM, LZ2018ARMA, MS2006CPDE, R2020NA} and the references therein for related results.

\subsection{Local H\"{o}lder regularity of $Du$}~

In this subsection, we consider the local $\varepsilon$-independent H\"older regularity of the spatial gradient $Du$ for the approximate equation
\begin{equation}\label{app-int-par-eq}
u_{t}=\operatorname{div}(v^{p-2}Du)+f(x,u,Du)\quad \text{in}~\Omega\times[0,\infty).
\end{equation}
 
We now proceed with some discussion. As noted before, this equation can be studied from the perspectives of both weak solutions and viscosity solutions. From the weak viewpoint, DiBenedetto-Friedman \cite{DF1984JRAM,DF1985JRAM,D1993} established the local H\"older regularity of $Du$ under the condition $\frac{2n}{n+2}<p<\infty$ and certain structural assumptions on $f$. As they pointed out, the restriction $\frac{2n}{n+2}<p$ is used in a Moser-DeGiorgi-type iteration (which is significantly more involved due to the presence of the time variable)  to obtain the local boundedness of $Du$.

From the viscosity perspective, Imbert-Jin-Silvestre \cite{JS2017JMPA,IJS2019ANA} proved the local H\"older regularity of $Du$ for $1<p<\infty$ in the case $f\equiv0$. Subsequently, Attouchi-Ruosteenoja \cite{A2020NA,AR2020DCDS} extended this result to merely continuous and bounded $f$. Their main tool is the so-called ``improvement of flatness" method, which will also be employed in the next subsection. We also refer the reader to \cite{A2016DIE, FZ2021JDE, AS2022CVPDE} and the references therein for related results. 

In summary, we have the following conclusion:

\begin{theorem}[Theorem 1.1 in \cite{AR2020DCDS}]\label{thm-local-C1alpha}
Let $u_{\varepsilon}$ be a solution of \eqref{app-int-par-eq} and let $f$ be continuous and bounded. Then for any $ \Omega' \subset\subset \Omega $ and $ \delta > 0 $, there exists $0<\alpha<1$ such that
\begin{equation}\label{local-C1alpha-inx}
\|Du_{\varepsilon}\|_{C^{\alpha, \frac{\alpha}{2}}(\Omega' \times (\delta, T-\delta))} \leq C(\underset{\Omega\times(0,T)}{\operatorname{osc}} u_{\varepsilon} + \|f\|_{L^{\infty}(\Omega\times(0,T))}),
\end{equation}
where $ C = C(p, n, d, \delta, d') $, $ d' = \mathrm{dist}(\Omega', \partial\Omega) $ and $ d = \mathrm{diam}(\Omega) $.
\end{theorem}

\subsection{Boundary regularity for elliptic problem}~

For the boundary regularity when $q=1$, we first introduce some notation.

For $ r > 0 $, we denote by $ B_r = \{ |x| < r \} $ the ball centered at $ 0 $ with radius $ r $, by $ B_r^+ = B_r \cap \{ x_n > 0 \} $ the upper half-ball, by $ B_r^- = B_r \cap \{ x_n < 0 \} $ the lower half-ball, and by $ T_r = B_r \cap \{ x_n = 0 \} $ the flat boundary of $ B_r^+ $. For $ x_0 \in \mathbb{R}^n $, we write $ B_r(x_0) = B_r + x_0 $ for the ball centered at $ x_0 $. For a domain $ \Omega \subset \mathbb{R}^n $ and a point $ x_0 \in \mathbb{R}^n $, we set $ \Omega_r = \Omega \cap B_r $, $ \partial \Omega_r = \partial \Omega \cap B_r $, and $ \Omega_r(x_0) = \Omega \cap B_r(x_0) $, $ \partial \Omega_r(x_0) = \partial \Omega \cap B_r(x_0) $. We denote by $ S(n) $ the space of real symmetric $ n \times n $ matrices and by $ I $ the identity matrix. 

Furthermore, according to the classical weak theory of $p$-Laplacian equations (see, e.g., \cite{D1983NA, L1983IMJ, T1984JDE, L1988NA}; 
see also \cite{ACF2023MA, A2026arXiv} for modern approaches to more general divergence-type operators),  we may define the following constant.

Let $\alpha_0$ be the optimal exponent such that for all sufficiently small $\varepsilon\geq0$ and for all weak solutions of 
$$
\operatorname{div}(v^{p-2}Du)=f \quad \text{in }B_1,
$$
where $f$ is bounded, we have
$$
\|u\|_{C^{1,\alpha_0}(B_{1/2})}\leq C\left(\|u\|_{L^{\infty}(B_1)}+\|f\|_{L^{\infty}(B_1)}\right).
$$
Since $u(x)=|x|^{\frac{p}{p-1}}$ is a solution of $\Delta_p u=C_{n,p}$, we know $\alpha_0\leq \frac{1}{p-1}$.

We first consider the elliptic problem in flat domains. We remark here that it is nontrivial to show that a viscosity solution to the homogeneous Neumann boundary problem is also a solution in the weak sense.

\begin{theorem}\label{thm-global-flat-elliptic}
 Let $p >1$ and let $u$ be a viscosity solution of  
$$\begin{cases} 
\operatorname{div}\left(|Du|^{p-2}Du\right)=0   &\text{in }B_1^+,\\
u_{\nu}=0 &\text{on } T_1. 
\end{cases}$$
Then $u \in C^{1,\alpha_0}(\overline{B_{1/2}^+})$, and  
$$\|u\|_{C^{1,\alpha_0}(\overline{B_{1/2}^+})} \leq C \left(\|u\|_{L^{\infty}(\overline{B_1^+})} + 1\right),$$
where $C$ depends only on $n$, $p$ and $\alpha_0$.   
\end{theorem}
\begin{proof}
\textbf{Step 1.} Reflect $u$ over $T_1$ and denote the reflected function still by $u$:
$$
u(x)=
\begin{cases}
u(x) & \text{if } x_n \ge 0,\\[2pt]
u(x',-x_n) & \text{if } x_n < 0,
\end{cases}
$$
where $x=(x',x_n)$. It is then easy to verify that
$$
\Delta_p u = 0 \quad \text{in } B_1^+ \cup B_1^- \text{ in the viscosity sense}.
$$

Define the sup- and inf-convolutions, which are slightly different from the usual ones:
\begin{align}
u^{\varepsilon}(x) &\coloneqq \sup_{y\in \overline{B_{1}}}\left\{ u(y)-\frac{|x-y|^q}{q\varepsilon^{q-1}} \right\},\label{sup-convolution}\\
u_{\varepsilon}(x) &\coloneqq \inf_{y\in \overline{B_{1}}}\left\{ u(y)+\frac{|x-y|^q}{q\varepsilon^{q-1}} \right\},\label{inf-convolution}
\end{align}
where $q = \max\{p/(p-1),2\}$.

The above sup- and inf-convolutions are well studied; see e.g.,  \cite{CC1995, MS2006CPDE, APR2017JMPA}. They satisfy the following properties:
\begin{enumerate}
\item There exists $r(\varepsilon) > 0$ such that
$$
u^\varepsilon(x) = \sup_{y \in B_{r(\varepsilon)}(x)\cap B_1} \left( u(y) - \frac{|x-y|^q}{q\varepsilon^{q-1}} \right).
$$
Moreover, $r(\varepsilon) \to 0$ as $\varepsilon \to 0$.
\item $u^\varepsilon \in C(B_{1-r(\varepsilon)})$ and is locally Lipschitz in $B_{1-r(\varepsilon)}$.
\item The sequence $(u^\varepsilon)$ is decreasing and $u^\varepsilon \to u$ pointwise in $B_1$.
\item For any $x_0 \in B_{1-r(\varepsilon)}$, there exists a concave paraboloid touching $u^\varepsilon$ from below. Hence $u^\varepsilon$ is $C^{1,1}$ from below.
\item Denote by $Y^\varepsilon(x)$ the set of points $y$ such that $u^{\varepsilon}(x)=u(y)-\frac{|x-y|^q}{q\varepsilon^{q-1}}$. The function $x \mapsto \max_{y \in Y^\varepsilon(x)} |y - x|$ is upper semicontinuous.
\item If the gradient $Du^\varepsilon(x)$ exists, then for every $y \in Y^\varepsilon(x)$ we have
$$
\left( \frac{|x - y|}{\varepsilon} \right)^{q-1} \le |Du^\varepsilon(x)|.
$$
\end{enumerate}

\textbf{Step 2.} In this step we claim that if $\varphi \in C^2(B_1)$ touches $u$ from above at some $x_0 \in T_1$, then $\varphi_\nu(x_0)=0$.

Indeed, by definition we have $\varphi_\nu(x_0) \ge 0$. Conversely, consider $\tilde{\varphi}(x) \coloneqq \varphi(x',-x_n)$. This function also touches $u$ from above at $x_0$, so $\tilde{\varphi}_\nu(x_0) \ge 0$. But $\tilde{\varphi}_\nu(x_0) = -\varphi_\nu(x_0)$, which gives $\varphi_\nu(x_0) \le 0$. Hence $\varphi_\nu(x_0)=0$.

\textbf{Step 3.} We now prove that for any $x_0 \in B_{1-r(\varepsilon)}^+ \cup B_{1-r(\varepsilon)}^-$ and any $y_0$ satisfying
$$
u^{\varepsilon}(x_0) = u(y_0) - \frac{|x_0-y_0|^q}{q\varepsilon^{q-1}},
$$
we necessarily have $y_0 \notin T_1$.

Suppose, for contradiction, that $y_0 \in T_1$. Then for any $x_0$ as above, define
$$
v(y) \coloneqq u(y_0) + \frac{|x_0-y|^q}{q\varepsilon^{q-1}} - \frac{|x_0-y_0|^q}{q\varepsilon^{q-1}}.
$$
One checks that $v(y_0)=u(y_0)$ and $v(y) \ge u(y)$ for all $y \in B_{1-r(\varepsilon)}$, i.e., $v$ touches $u$ from above at $y_0$. By \textbf{Step 2}, we must have $v_\nu(y_0)=0$. However, a direct computation gives
$$
v_\nu(y_0) = \varepsilon^{1-q} |x_0-y_0|^{q-2} \langle y_0 - x_0, \nu \rangle,
$$
where $\nu$ is the inward unit normal to $T_1$. Since $x_0 \notin T_1$ and $y_0 \in T_1$, we have $\langle y_0 - x_0, \nu \rangle \neq 0$, and therefore $v_\nu(y_0) \neq 0$, a contradiction. Hence $y_0 \notin T_1$.

\textbf{Step 4.} In this step we prove that $-\Delta_p u^\varepsilon \le 0$ in $B_{1-r(\varepsilon)}^+ \cup B_{1-r(\varepsilon)}^-$ in the viscosity sense.

Let $x_0 \in B_{1-r(\varepsilon)}^+ \cup B_{1-r(\varepsilon)}^-$ and assume that $u^\varepsilon$ is not constant near $x_0$. Suppose $\varphi \in C^2(B_1)$ touches $u^\varepsilon$ from above at $x_0$ and $D\varphi(x_0) \neq 0$. By the properties of the sup-convolution, there exists $y_0 \in B_{r(\varepsilon)}(x_0)$ such that
$$
u^\varepsilon(x_0) = u(y_0) - \frac{|x_0-y_0|^q}{q\varepsilon^{q-1}}.
$$
Moreover, by \textbf{Step 3} we have $y_0 \notin T_1$. Therefore, 
$$
u^{\varepsilon}(x)\geq u(x+y_0-x_0)-\frac{|x_0-y_0|^q}{q\varepsilon^{q-1}}
$$
in a neighborhood of $x_0$. It follows that $v(x)\coloneqq\varphi(x+x_0-y_0)+\frac{|x_0-y_0|^q}{q\varepsilon^{q-1}}$ touches $u(x)$ from above at $y_0$, hence $-\Delta_p \varphi(x_0)=-\Delta_p v(y_0)\leq0$.

\textbf{Step 5.} We now conclude the $C^{1,\alpha_0}$ regularity for $p \ge 2$.

By property (4) in \textbf{Step 1}, $u^{\varepsilon}$ is semi-convex; therefore it is twice differentiable a.e. in $B_{1-r(\varepsilon)}^+ \cup B_{1-r(\varepsilon)}^-$. Consequently, $-\Delta_p u^{\varepsilon}(x) \le 0$ holds pointwise a.e. in $B_{1-r(\varepsilon)}$. Since $u_{\varepsilon}$ is Lipschitz, it follows that $-\Delta_p u^{\varepsilon} \le 0$ in the weak sense in $B_{1-r(\varepsilon)}$, and hence also in the viscosity sense. Taking the limit, we obtain from Definition \ref{vis-ell-def-2} that $-\Delta_p u \le 0$ in $B_{1}$ in the viscosity sense. Applying the same argument to $u_{\varepsilon}$ yields $\Delta_p u = 0$ in $B_{1}$. Therefore, by standard regularity theory for weak solutions or viscosity solutions, we conclude $u \in C^{1,\alpha_0}(\overline{B_{1/2}^+})$.

\textbf{Step 6.} In this step, we conclude the $C^{1,\alpha_0}$ regularity for $1 < p < 2$.

As in \textbf{Step 5}, we have $-\Delta_p u^{\varepsilon}(x) \le 0$ pointwise a.e. in $B_{1-r(\varepsilon)} \setminus \{Du^{\varepsilon} = 0\}$. Moreover, for any $\delta > 0$ and any nonnegative $\psi \in C_c^{\infty}(B_{1-r(\varepsilon)})$,
\begin{eqnarray}\label{delta-ineq-1<p<2}
\int_{B_{1}} \left(|Du^{\varepsilon}|^2 + \delta^2\right)^{\frac{p-2}{2}} Du^{\varepsilon} \cdot D\psi \, dx
\le \int_{B_1} -\operatorname{div}\left(\left(|Du^{\varepsilon}|^2 + \delta^2\right)^{\frac{p-2}{2}} Du^{\varepsilon}\right) \psi \, dx.
\end{eqnarray}

Clearly,
$$
\lim_{\delta \to 0} \int_{B_{1}} \left(|Du^{\varepsilon}|^2 + \delta^2\right)^{\frac{p-2}{2}} Du^{\varepsilon} \cdot D\psi \, dx
= \int_{B_1} |Du^{\varepsilon}|^{p-2} Du^{\varepsilon} \cdot D\psi \, dx,
$$
so it suffices to prove
\begin{eqnarray}\label{delta-limsup-eq-1<p<2}
\limsup_{\delta \to 0} \int_{B_1} -\operatorname{div}\left(\left(|Du^{\varepsilon}|^2 + \delta^2\right)^{\frac{p-2}{2}} Du^{\varepsilon}\right) \psi \, dx \le 0.
\end{eqnarray}

Consider a point $x_0$ where both $Du^\varepsilon(x_0)$ and $D^2u^\varepsilon(x_0)$ exist. By property (6), we have $|y - x_0| \le |Du^\varepsilon(x_0)|^{\frac{1}{q-1}} \varepsilon$ for every $y \in Y^\varepsilon(x_0)$. Moreover, by the upper semicontinuity result in property (5), for every $n$ there exists a small radius $\rho_n$ such that for all $x \in B_{\rho_n}(x_0)$ and all $y \in Y^\varepsilon(x)$,
$$
|y - x| \le |Du^\varepsilon(x_0)|^{\frac{1}{q-1}} \varepsilon + \frac{1}{n} =: r_n.
$$
This implies that for every $x \in B_{\rho_n}(x_0)$,
$$
u^\varepsilon(x) = \sup_{y \in B_{r_n}(x)} \left( u(y) - \frac{|x - y|^q}{q\varepsilon^{q-1}} \right).
$$
For each $y \in B_{r_n}(x)$, the function $\varphi_y(x) = u(y) - \frac{|x - y|^q}{q\varepsilon^{q-1}}$ is smooth and satisfies
$$
D^2\varphi_y(x) \ge -\frac{q - 1}{\varepsilon^{q-1}} r_n^{q-2} I.
$$
Since $u^\varepsilon$ is the supremum of the family $\{\varphi_y\}_{y \in B_{r_n}(x)}$, we conclude that $u^\varepsilon$ is semi-convex and
$$
D^2u^\varepsilon(x) \ge -\frac{q - 1}{\varepsilon^{q-1}} r_n^{q-2} I
$$
a.e. in $B_{\rho_n}(x_0)$. Letting $n \to \infty$ yields the estimate
$$
D^2u^\varepsilon(x_0) \ge -\frac{q - 1}{\varepsilon} |Du^\varepsilon(x_0)|^{\frac{q-2}{q-1}} I.
$$
It follows that $D^2u^\varepsilon(x) \ge 0$ a.e. on the set $\{Du^{\varepsilon} = 0\}$, and that
$$
-\operatorname{div}\left(\left(|Du^{\varepsilon}|^2 + \delta^2\right)^{\frac{p-2}{2}} Du^{\varepsilon}\right) \le C_{\varepsilon}
$$
a.e. on $\{Du^{\varepsilon} \neq 0\}$ for some constant $C_{\varepsilon}$ independent of $\delta$.

Denote
$$
N^{\varepsilon} \coloneqq \{ x \in B_{1-r(\varepsilon)} : Du^{\varepsilon}(x) \neq 0,\; D^2u^{\varepsilon}(x) \text{ exists} \}.
$$
By Fatou's lemma,
\begin{eqnarray*}
&&\limsup_{\delta \to 0} \int_{B_1} -\operatorname{div}\left(\left(|Du^{\varepsilon}|^2 + \delta^2\right)^{\frac{p-2}{2}} Du^{\varepsilon}\right) \psi \, dx \\
&\le& \int_{B_1} \limsup_{\delta \to 0} \left( -\operatorname{div}\left(\left(|Du^{\varepsilon}|^2 + \delta^2\right)^{\frac{p-2}{2}} Du^{\varepsilon}\right) \psi \right) dx \\
&=& \int_{N^{\varepsilon}} \limsup_{\delta \to 0} \left( -\operatorname{div}\left(\left(|Du^{\varepsilon}|^2 + \delta^2\right)^{\frac{p-2}{2}} Du^{\varepsilon}\right) \psi \right) dx \\
&=& \int_{N^{\varepsilon}} -\Delta_p u^{\varepsilon} \cdot \psi \, dx \le 0.
\end{eqnarray*}
Letting $\delta \to 0$, we obtain
$$
\int_{B_1} |Du^{\varepsilon}|^{p-2} Du^{\varepsilon} \cdot D\psi \, dx \le 0,
$$
hence $-\Delta_p u^{\varepsilon} \le 0$ in the weak sense in $B_{1-r(\varepsilon)}$. The same argument applied to $u_{\varepsilon}$ yields $\Delta_p u = 0$ in $B_{1}$, and therefore $u \in C^{1,\alpha_0}(\overline{B_{1/2}^+})$.
\end{proof}

We now consider the general case:
\begin{theorem}\label{thm-global-general-elliptic}
 Let $\Omega$ be a $C^1$ domain with $0 \in \partial \Omega_1$, $\varepsilon\geq0$, $p >1$, and $u$ be a solution of  
$$\begin{cases} 
\operatorname{div}(v^{p-2}Du) = f & \text{in } \Omega_1, \\ 
u_{\nu} = g & \text{on } \partial \Omega_1, 
\end{cases}$$
where $f \in C(\overline{\Omega})$, and $g\in C^{\alpha}(\partial\Omega)$ for some $\alpha\in(0,\alpha_0)$. Then $u \in C^{1,\alpha}(\overline{\Omega_{1/2}})$ and  
$$\|u\|_{C^{1,\alpha}(\overline{\Omega_{1/2}})} \leq C \left(\|u\|_{L^{\infty}(\overline{\Omega_1})} + \|f\|_{L^{\infty}(\overline{\Omega_1})}^{\frac{1}{p-1}} + \|g\|_{C^{\alpha}(\partial \Omega_1)}\right),$$
where $C$ depends only on $n, p, \alpha$, $\|g\|_{C^{\alpha}(\partial \Omega_1)}$ and the $C^1$ modulus of $\partial \Omega_1$.   
\end{theorem}
\begin{proof}
Rewrite $\operatorname{div}(v^{p-2}Du)$ as
$$
v^{p-2}\left(\Delta u+(p-2)\frac{D^2uDu\otimes Du}{v^2}\right)\coloneqq v^{p-2}G_{\varepsilon}\left(Du,D^2u\right).
$$
It is easy to verify that $G_{\varepsilon}$ is uniformly elliptic with $\varepsilon$-independent elliptic coefficients $(\lambda,\Lambda)$, where $\lambda=\min\{1,p-1\}$, $\Lambda=\max\{1, p-1\}$.

We only prove for $1<p<2$. For the case $p>2$, the condition 
$$
\frac{\|f\|_{L^{\infty}(\Omega_1)}}{\left(|q|^{2}+|\varepsilon|^{2}\right)^{\frac{p-2}{2}}}\leq \varepsilon_0
$$ 
in \textbf{Step 2} and \textbf{Step 3} can be omitted, and the proof proceeds analogously.

\textbf{Step 1.} 
We first define $u_1(x)\coloneqq u(x)-g(0)x_n-u(0)$ to flatten the boundary condition at $0$. Set $q_1\coloneqq-g(0)e_n$, then $u_1$ satisfies
$$
\begin{cases}
\left(|Du_1-q_1|^2+\varepsilon^2\right)^{\frac{p-2}{2}}G_{\varepsilon}\left(Du_1-q_1,D^2u_1\right)=f_1  &\text{in } \Omega_1,\\[2pt]
\nu\cdot Du_1=g_1 &\text{on } \partial\Omega_1,
\end{cases}
$$
where $f_1=f$, $g_1(x)=g(x)-g(0)\nu_n$, and satisfies $u_1(0)=g_1(0)=0$.

Next, by scaling
$$
\bar{u}(x)\coloneqq\frac{u_1(sx)}{K}
$$
with $s>0$ small and $K>0$ large, we find that $\bar{u}$ satisfies
$$
\begin{cases}
\left(|D\bar u-\bar q|^2+\bar\varepsilon^2\right)^{\frac{p-2}{2}}G_{\bar\varepsilon}(D\bar u-\bar q ,D^2\bar u)=\bar f  & \text{in } (\frac{\Omega}{s})\cap B_1,\\[4pt]
\nu\cdot D\bar u=\bar g  & \text{on } \partial(\frac{\Omega}{s})\cap B_1,
\end{cases}
$$
where $\bar q\coloneqq\frac{sq_1}{K}$, $\bar \varepsilon\coloneqq\frac{s\varepsilon}{K}$, $\bar f(x)=\frac{s^p}{K^{p-1}}f(x)$, and $\bar g(x)=\frac{s}{K}g_1(sx)$.

After choosing suitable coordinates, we may assume that $\Omega_1$ is given by the graph of a $C^1$ function $\varphi_\Omega$ with $\varphi_\Omega(0) = 0$ and $D\varphi_\Omega(0) = 0$. Since $\varphi_\Omega \in C^1$, for any given $\varepsilon_0>0$, there exists a small $s > 0$ depending on the $C^1$ modulus of $\varphi_\Omega$ such that $|D\varphi_\Omega(x)| \leq \varepsilon_0$ for all $|x| \leq s$. Define $\tilde{\varphi}(x) \coloneqq \varphi_{\frac{1}{s}\Omega}(x) = \frac{\varphi(sx)}{s}$, then $\|\tilde{\varphi}\|_{C^1} \leq \varepsilon_0$.

\textbf{Step 2.} Suppose $\varepsilon>0$ and let $u$ be a solution of 
$$\begin{cases}
\left(|Du-q|^2+\varepsilon^2\right)^{\frac{p-2}{2}} G_{\varepsilon}\left(Du-q,D^2u\right) = f &\text{ in } \Omega_1, \\[2pt]
\nu \cdot Du = g &\text{ on } \partial \Omega_1,
\end{cases}$$
with $q \in \mathbb{R}^n$ and $\|u\|_{C^0} \leq 1$. In this step, we claim that for any given $\delta > 0$, there exists $\varepsilon_0 = \varepsilon_0(\delta, n, p) > 0$ such that if
$$\|f\|_{L^\infty(\Omega_1)} , \|g\|_{L^\infty(T_1)} ,  \|\varphi\|_{C^1(T_1)} \leq \varepsilon_0,\quad |q|^2+|\varepsilon|^2\geq \varepsilon_0^{-1},\quad \text{and} \quad \frac{\|f\|_{L^{\infty}(\Omega_1)}}{\left(|q|^{2}+|\varepsilon|^{2}\right)^{\frac{p-2}{2}}}\leq \varepsilon_0,$$
then there exists a function $h \in C^{1, \alpha_0}(\overline{B_{3/4}^+})$ such that $\|h\|_{C^{1,\alpha_0}(\overline{B_{3/4}^+})}\leq C $ (independent of $q$ and $\varepsilon$), $\langle Dh(0),e_n\rangle=0$, and $\|u - h\|_{L^\infty(\Omega_{1/2})} \leq \delta$.

The proof of \textbf{Step 2} is based on a compact argument, and we can also deal with the singular case $1<p<2$.

Assume the assertion in \textbf{Step 2} is false. Then there exists a $\delta>0$ and sequences $u_k$, $f_k$, $g_k$, $\Omega_k$, $q_k$, $\varepsilon_k$ such that $u_k$ satisfies $|u_k|\leq 1$ and
$$\begin{cases}
\left(|Du_k-q_k|^2+\varepsilon_k^2\right)^{\frac{p-2}{2}} G_{\varepsilon_k}(Du_k-q_k,D^2u_k) = f_k & \text{in } (\Omega_k)_1, \\[4pt]
\nu \cdot Du_k = g_k & \text{on } \partial (\Omega_k)_1,
\end{cases}$$
with $\|f_k\|_{L^\infty}, \|g_k\|_{L^\infty}, \|\varphi\|_{C^1}\leq \frac{1}{k}$, $|q_k|^2+\varepsilon_k^2\geq k$, $\frac{\|f_k\|_{L^{\infty}}}{(|q_k|^2+\varepsilon_k^2)^{\frac{p-2}{2}}}\leq \frac{1}{k}$,  but $\|u_k-h_k\|_{L^{\infty}((\Omega_k)_{1/2})}\geq \delta$ for $h_k$ satisfying the corresponding conditions. We now show that $u_k$ satisfies the conditions in Theorem 3.1 of \cite{BKO2025CVPDE}, that is, let
$$
P_{\Lambda, \lambda}^+ (Du, D^2u) = \Lambda \, \operatorname{tr} D^2 u^+ - \lambda \, \operatorname{tr} D^2 u^- + \Lambda |Du|,
$$
$$
P_{\Lambda, \lambda}^- (Du, D^2u) = \lambda \, \operatorname{tr} D^2 u^+ - \Lambda \, \operatorname{tr} D^2 u^- - \Lambda |Du|.
$$
We aim to show that there exists a constant $C_0$ such that
$$
P^{-}(D^2 u_k, Du_k) \leq C_0, \quad P^{+}(D^2 u_k, Du_k) \geq -C_0 \quad \text{in } \{ |Du_k - q_k| > 1 \} \cap (\Omega_k)_1.
$$
We only prove for $P^+$. For any $\varphi$ that touches $u_k$ from below at $x_0\in \{ |Du_k - q_k| > 1 \} \cap (\Omega_k)_1$, we have 
\begin{eqnarray*}
P^+(D\varphi,D^2\varphi)&\geq& G_{\varepsilon_k}(D\varphi-q_k,D^2\varphi)+\Lambda|D\varphi|=f_k\left(|D\varphi-q_k|^2+\varepsilon_k^2\right)^{\frac{2-p}{2}}+\Lambda|D\varphi|\\
&\geq&-2^{\frac{2-p}{2}}\|f_k\|_{L^{\infty}}\left(|D\varphi|^2+|q_k|^2+\varepsilon_k^2\right)^{\frac{2-p}{2}}+\Lambda|D\varphi|\geq-C_p,
\end{eqnarray*}
where we have used the condition $\frac{\|f\|_{L^{\infty}(\Omega_1)}}{\left(|q|^{2}+|\varepsilon|^{2}\right)^{\frac{p-2}{2}}}\leq \varepsilon_0$ and the Cauchy inequality in the last inequality.

Therefore, by Theorem 3.1 of \cite{BKO2025CVPDE}, a subsequence of $u_k$ converges locally uniformly to a function $u_{\infty}\in C(\overline{B_1^+})$ locally uniformly, and we may assume $\frac{q_k}{\sqrt{|q_k|^2+|\varepsilon_k|^2}}\to E\in B_1$. Rewriting the equation satisfied by $u_k$ as
$$
\left(\frac{|Du_k-q_k|^2+\varepsilon_k^2}{|q_k|^2+\varepsilon_k^2}\right)^{\frac{p-2}{2}}\left(\Delta u_k+(p-2)\frac{\frac{u_{k,i}-q_{k,i}}{\sqrt{|q_k|^2+\varepsilon_k^2}}\cdot\frac{u_{k,j}-q_{k,j}}{\sqrt{|q_k|^2+\varepsilon_k^2}}}{\frac{|Du_k-q_k|^2+\varepsilon_k^2}{|q_k|^2+\varepsilon_k^2}}u_{k,ij}\right)=\frac{f_k}{\left(|q_k|^2+\varepsilon_k^2\right)^{\frac{p-2}{2}}},
$$
by a standard stability argument we deduce that $u_{\infty}$ is a viscosity solution of
$$
\begin{cases}
\Delta u_{\infty}+(p-2)E_iE_ju_{\infty,ij}=0\quad &\text{in } B_1^+,\\[2pt]
\nu\cdot Du_{\infty}=0\quad &\text{on } T_1.
\end{cases}
$$
From the regularity results in \cite{MS2006CPDE}, it follows that $u_{\infty}\in C^{1,\alpha_0}(\overline{B_{3/4}^+})$ and $\langle Du_{\infty}(0),e_n\rangle=0$, which leads to a contradiction. Hence \textbf{Step 2} is proved.

\textbf{Step 3.} In this step, we will prove for (maybe smaller) $\varepsilon_0>0$, the assertions in \textbf{Step 2} remain valid under the restriction $|\langle q,e_n\rangle|\leq \varepsilon_0$ instead of $|q|^2+\varepsilon^2\geq\varepsilon_0^{-1}$.

In view of \textbf{Step 2}, we may also assume that $|q|^2+\varepsilon^2\leq \varepsilon_0^{-1}$. Assume the assertion in \textbf{Step 3} is false. Then there exist $\delta>0$ and sequences $u_k$, $f_k$, $g_k$, $\Omega_k$, $q_k$, $\varepsilon_k$ such that $u_k$ satisfies $|u_k|\leq 1$ and
$$\begin{cases}
\left(|Du_k-q_k|^2+\varepsilon_k^2\right)^{\frac{p-2}{2}} G_{\varepsilon_k}(Du_k-q_k,D^2u_k) = f_k & \text{in } (\Omega_k)_1, \\[4pt]
\nu \cdot Du_k = g_k & \text{on } \partial (\Omega_k)_1,
\end{cases}$$
with $\|f_k\|_{L^\infty}, \|g_k\|_{L^\infty}, \|\varphi\|_{C^1}, |\langle q,e_n\rangle|\leq \frac{1}{k}$, $|q_k|^2+\varepsilon_k^2\leq \varepsilon_0^{-1}$, $\frac{\|f_k\|_{L^{\infty}}}{(|q_k|^2+\varepsilon_k^2)^{\frac{p-2}{2}}}\leq \frac{1}{k}$, but\\ $\|u_k-h_k\|_{L^{\infty}((\Omega_k)_{1/2})}\geq \delta$ for $h_k$ satisfying the corresponding conditions. By Theorem 3.1 of \cite{BKO2025CVPDE}, a subsequence of $u_k$ converges locally uniformly to a function $u_{0}\in C(\overline{B_1^+})$ locally uniformly, and we may assume $q_k\to q_0$ for some $q_0$ with $\langle q_0,e_n\rangle=0,  |q_0|^2\leq \varepsilon_0^{-1}$ and $\varepsilon_k\to \varepsilon_1\geq 0$. 

By Definition \ref{vis-ell-def-2} and a standard stability argument, $u_0$ is a viscosity solution of
$$
\begin{cases}
\operatorname{div}\left((|Du_0-q_0|^2+\varepsilon_1^2)^{\frac{p-2}{2}}(Du-q_0)\right)=0    &\text{in } B_1^+,\\[2pt]
\nu\cdot Du_0=0  &\text{on } T_1.
\end{cases}
$$
Denote $\tilde{u}(x)\coloneqq u_0(x)-q_0\cdot x$; then $\tilde{u}$ satisfies
$$
\begin{cases}
\operatorname{div}\left((|D\tilde{u}|^2+\varepsilon_1^2)^{\frac{p-2}{2}}D\tilde{u}\right)=0    &\text{in } B_1^+,\\[2pt]
\nu\cdot D\tilde{u}=0  &\text{on } T_1,
\end{cases}
$$
in the viscosity sense. By Theorem \ref{thm-global-flat-elliptic} (which remains valid for $\varepsilon_1>0$), we obtain $\tilde{u}\in C^{1,\alpha_0}(\overline{B_{3/4}^+})$, $\langle D\tilde{u}(0),e_n\rangle=0$. Consequently, $u_0\in C^{1,\alpha_0}(\overline{B_{3/4}^+})$ and satisfies 
$\|u_0\|_{C^{1,\alpha_0}(\overline{B_{3/4}^+})}\leq C(\varepsilon_0,p,n)$ together with 
$\langle Du_0(0),e_n\rangle=0$, which yields a contradiction. 
Thus, \textbf{Step~3} is proved.

\textbf{Step 4.}  We claim that under the assumptions in \textbf{Step 2} or \textbf{Step 3}, there exist $0<r_0<\frac{1}{2}$ and constant $\varepsilon_0>0$ such that if
$$\|f\|_{L^\infty(\Omega_1)} , \|g\|_{L^\infty(T_1)} ,  \|\varphi\|_{C^1(T_1)} \leq \varepsilon_0,~ |\langle q,e_n\rangle|\leq \varepsilon_0 \text{ or }|q|^2+\varepsilon^2\geq\varepsilon_0^{-1},~ \frac{\|f\|_{L^{\infty}(\Omega_1)}}{(|q|^2+\varepsilon^2)^{\frac{p-2}{2}}}\leq \varepsilon_0,$$
then there exists an affine function $l=u(0)+b\cdot x$ such that
$$
\|u-l\|_{L^{\infty}(\Omega_{r_0})}\leq r_0^{1+\alpha},\quad \langle e_n,b\rangle=0,\quad \text{and} \quad |b|\leq C_1.
$$

For given $\delta>0$, let $\varepsilon_0>0$ be the constant in \textbf{Step 2} and \textbf{Step 3}. Then there exists a function $h\in C^{1,\alpha_0}$ satisfying $\|u-h\|_{L^{\infty}}\leq \delta$. From the $C^{1,\alpha_0}$ estimates for $h$, letting $\tilde{l}(x)=h(0)+Dh(0)\cdot x$, we obtain
$$
|h(x)-\tilde{l}(x)|\leq C_1|x|^{1+\alpha_0}, \quad \nu\cdot Dh(0)=0,\quad \text{and} \quad |Dh(0)|\leq C_1
$$
for some $C_1$ independent of $q$.

Choose $0<r_0<\frac{1}{2}$ sufficiently small such that $C_1 r_0^{1+\alpha_0}\leq \frac{1}{3}r_0^{1+\alpha}$, and set $\delta=\frac{1}{3}r_0^{1+\alpha_0}$. Define $l(x)=u(0)+Dh(0)\cdot x$; we have
\begin{eqnarray*}
\|u-l\|_{L^{\infty}(\Omega_{r_0})}&\leq& \|u-h\|_{L^{\infty}(\Omega_{r_0})}+\|h-\tilde{l}\|_{L^{\infty}(\Omega_{r_0})}+|u(0)-h(0)|\\
&\leq& 2\delta+C_1 r_0^{1+\alpha_0}\leq r_0^{1+\alpha}.
\end{eqnarray*}
Hence \textbf{Step 4} is proved, and from now on, $\varepsilon_0$, $\delta$, $r_0$ and $C_1$ are fixed. Moreover, we can define $A=A(r_0, \varepsilon_0)=A(n, p, \alpha)$ be a positive integer such that 
\begin{eqnarray}\label{def-A}
r_0^A\leq\frac{1}{2}\varepsilon_0^{\frac{3}{2}}.
\end{eqnarray}

\textbf{Step 5.} The case $|q_1| = 0$ can be handled by a similar but simpler argument; we therefore assume $|q_1| > 0$. Let $\varepsilon_0$, $\delta$, $r_0$, $C_1$ and $A$ be the constants in \textbf{Step 4}. We can choose $s$ in \textbf{Step 1} sufficiently small and $K$ in \textbf{Step 1} sufficiently large such that $\frac{s}{K}$, $\frac{s^p}{K^{p-1}}$, $\frac{s^{1+\alpha}}{K}$ and $\frac{s^{p-1}}{K^{p-2}}$ are small enough. Consequently, we have $\|\bar u\|_{C^0} \leq \frac{1}{4}$,  $\|\bar f\|_{L^\infty} \leq \frac{\varepsilon_0}{2(C_1+1)}$, $\|\bar g\|_{C^\alpha} \leq \frac{1}{4} \varepsilon_0$, and $\frac{\|\bar f\|_{L^{\infty}}}{\left(|\bar q|^2+\bar\varepsilon^2\right)^{\frac{p-2}{2}}}\leq \varepsilon_0$.
Moreover, we can choose $s,K$ appropriately such that 
\begin{eqnarray}\label{def-k0}
|\bar q|=\frac{s}{K}|q_1|=r_0^{k_0\alpha}a_0 
\end{eqnarray}
for some $\frac{\varepsilon_0}{2}\leq a_0\leq \varepsilon_0$ and some integer $k_0\geq0$. From now on $s$ and $K$ are fixed.

For such $\bar u$, we \textbf{claim} that with the above $\varepsilon_0$, $\delta$, $r_0$, $C_1$, $A$ and $k_0$, for each $k=0,1,...$ there exists an affine function $l_k(x)=b_k\cdot x$ such that
$$
\|\bar u-l_k\|_{L^{\infty}(\Omega_{r_0^k})}\leq r_0^{k(1+\alpha)},~ \langle e_n,b_k\rangle=0, \text{ and } |b_k-b_{k+1}|\leq C_1 r_0^{k\alpha}, \text{ if } k\leq k_0+1,
$$
and
$$
\|\bar u-l_k\|_{L^{\infty}(\Omega_{r_0^{k+A}})}\leq r_0^{(k+A)(1+\alpha)},~ \langle e_n,b_k\rangle=0,\text{ and } |b_k-b_{k+1}|\leq C_1 r_0^{(k+A)\alpha}, \text{ if } k\geq k_0+2.
$$

For $k=0$, the assertions clearly hold by choosing $l_0=0$ and by \textbf{Step 1}. We now proceed by induction. Assume the claim holds with some $k\leq k_0$, and consider
$$
v(x)\coloneqq\frac{(\bar u-l_k)(r_0^kx)}{r_0^{k(1+\alpha)}}.
$$
Then $|v|\leq 1$ in $(\frac{1}{r_0^k}\Omega)_1$, $v(0)=0$, and $v$ satisfies
$$
\begin{cases}
\left(|Dv-q_k|^2+|\varepsilon_k|^2\right)^{\frac{p-2}{2}}G_{\varepsilon_k}\left(Dv-q_k,D^2v\right)=f_k  &\text{in } (\frac{1}{r_0^k}\Omega)_1,\\[4pt]
\nu\cdot Dv=g_k
&\text{on } \partial(\frac{1}{r_0^k}\Omega)_1, 
\end{cases}
$$
where $q_k=r_0^{-k\alpha}(\bar q-b_k)$, $\varepsilon_k=r_0^{-k\alpha}\bar\varepsilon$, $f_k(x)=r_0^{k(1-\alpha(p-1))}\bar f(x)$ and $g_k(x)=r_0^{-k\alpha}(\bar g(r_0^kx)-b_k\cdot \nu(r_0^kx))$.

For $q_k$, since $k\leq k_0$, then we have $|\langle q_{k}, e_n\rangle|\leq\varepsilon_0$, so $q_k$ satisfies the assumption of \textbf{Step 3}. We next check that $v$ satisfies the assumptions of \textbf{Step 4}. Since $\alpha(p-1)\leq 1$, we have $\|f_k\|_{L^{\infty}}\leq \|\bar f\|_{L^{\infty}}\leq \varepsilon_0$. Because $\bar g\in C^{1,\alpha}$, we obtain $\|g_k\|_{C^\alpha}\leq \varepsilon_0$. Moreover, define $\varphi_k(x)\coloneqq\varphi_{\frac{1}{r_0^k}\Omega}(x)=\frac{\varphi(r_0^kx)}{r_0^k}$, then $\varphi_k(0)=0$, $D\varphi_k(0)=0$, and $\|\varphi_k\|_{C^1}=\|\varphi\|_{C^1}\leq \varepsilon_0$. For the last assumption, we note that
$$
\frac{f_k}{\left(|q_k|^2+\varepsilon_k^2\right)^{\frac{p-2}{2}}}=r_0^{k(1-\alpha)}\frac{\bar f}{\left(|\bar q-b_k|^2+\bar\varepsilon^2\right)^{\frac{p-2}{2}}}\leq \varepsilon_0.
$$
Therefore, we may apply the result of \textbf{Step 4} to obtain an affine function $\tilde{l}(x)=\tilde{b}\cdot x$ such that
$$
\|v-\tilde{l}\|_{L^{\infty}\left((\frac{1}{r_0^k}\Omega)_{r_0}\right)}\leq r_0^{1+\alpha},\quad \langle e_n,\tilde{b}\rangle=0,\quad \text{and}\quad |\tilde{b}|\leq C_1.
$$
Now set $l_{k+1}(x)=l_k(x)+r_0^{k(1+\alpha)}\tilde{l}(r_0^{-k}x)$. Then $l_{k+1}$ satisfies the required properties for index $k+1$, and  the induction holds for $k\leq k_0+1$. 

\

Next, suppose the induction holds for some $k\geq k_0+1$, consider
$$
v(x)\coloneqq\frac{(\bar u-l_k)(r_0^{k+A}x)}{r_0^{(k+A)(1+\alpha)}},
$$
where $A$ is defined in \eqref{def-A}.

Then $|v|\leq 1$ in $(\frac{1}{r_0^{k+A}}\Omega)_{r_0^{-A}}$, $v(0)=0$, and $v$ satisfies
$$
\begin{cases}
\left(|Dv-q_k|^2+|\varepsilon_k|^2\right)^{\frac{p-2}{2}}G_{\varepsilon_k}\left(Dv-q_k,D^2v\right)=f_k  &\text{in } (\frac{1}{r_0^{k+A}}\Omega)_1,\\[4pt]
\nu\cdot Dv=g_k
&\text{on } \partial(\frac{1}{r_0^{k+A}}\Omega)_1, 
\end{cases}
$$
where $q_k=r_0^{-(k+A)\alpha}(\bar q-b_k)$, $\varepsilon_k=r_0^{-(k+A)\alpha}\bar\varepsilon$, $f_k(x)=r_0^{(k+A)(1-\alpha(p-1))}\bar f(x)$ and $g_k(x)=r_0^{-(k+A)\alpha}(\bar g(r_0^{k+A}x)-b_k\cdot \nu(r_0^{k+A}x))$.

For $q_k$, by \eqref{def-k0}, we have $\frac{\varepsilon_0}{2}\leq|\langle q_{k_0}, e_n\rangle|\leq \varepsilon_0$. Together with \eqref{def-A}, it follows that $|q_k|\geq |\langle q_k, e_n\rangle|\geq |\langle q_{k_0+1},e_n\rangle|\geq  \varepsilon_0^{-\frac{1}{2}}$, and $q_k$ satisfies the assumption of \textbf{Step 2}. With a similar argument as the case $k\leq k_0$, we can check that $v_k$, $g_k$, $\varphi_k$ satisfy the assumptions of \textbf{Step 4}, so there exists an affine function $\tilde{l}(x)=\tilde{b}\cdot x$ such that
$$
\|v-\tilde{l}\|_{L^{\infty}\left((\frac{1}{r_0^{k+A}}\Omega)_{r_0}\right)}\leq r_0^{1+\alpha},\quad \langle e_n,\tilde{b}\rangle=0,\quad \text{and}\quad |\tilde{b}|\leq C_1.
$$
Now set $l_{k+1}(x)=l_k(x)+r_0^{(k+A)(1+\alpha)}\tilde{l}(r_0^{-(k+A)}x)$. Recall that $k\geq k_0+1$, we have $l_{k+1}$ satisfies the required properties for index $k+1\geq k_0+2$, and the induction holds for $k+1$. Thus \textbf{Step 5} is proved.

\textbf{Step 6.} Based on \textbf{Step 5}, a standard argument shows that $b_k$ converges to some limit $b$, and the corresponding affine function $l(x)=b\cdot x$ is the affine approximation of $u$ at $0$. We now prove that
\begin{eqnarray}\label{C1alpha}
\sup_{x\in\Omega_r}|u(x)-l(x)|\leq C r^{1+\alpha}
\end{eqnarray}
for all sufficiently small $r>0$, where $C$ is a constant depending only on $n$, $p$, $\alpha$.

For any fixed $0<r<r_0$, there exists a unique $B\in\mathbb N$ such that
\begin{eqnarray*}
r_0^{B+1}< r \leq r_0^B.
\end{eqnarray*}
We distinguish three cases.

\textbf{Case 1.} $B\leq k_0+1$. Then
\begin{eqnarray*}
\sup_{x\in\Omega_r}|u(x)-l(x)|&\leq& \sup_{x\in\Omega_{r_0^B}}|u(x)-l(x)|\\
&\leq& \sup_{x\in\Omega_{r_0^B}}|u(x)-l_B(x)|+\sup_{x\in\Omega_{r_0^B}}|l_B(x)-l(x)|\\
&\leq& r_0^{B(1+\alpha)}+\frac{C_1}{1-r_0^\alpha}r_0^{B(1+\alpha)}~ \leq~ C r^{1+\alpha}. 
\end{eqnarray*}

\textbf{Case 2.} $k_0+2\leq B \leq k_0+A+1$. Then
\begin{eqnarray*}
\sup_{x\in\Omega_r}|u(x)-l(x)|&\leq& \sup_{x\in\Omega_{r_0^B}}|u(x)-l(x)|\\
&\leq& \sup_{x\in\Omega_{r_0^B}}|u(x)-l_{k_0+1}(x)|+\sup_{x\in\Omega_{r_0^B}}|l_{k_0+1}(x)-l(x)|\\
&\leq& r_0^{(k_0+1)(1+\alpha)}+\frac{C_1}{1-r_0^\alpha}r_0^{B+(k_0+1)\alpha}~ \leq~ C r^{1+\alpha}. 
\end{eqnarray*}

\textbf{Case 3.} $k_0+A+2\leq B$. Then
\begin{eqnarray*}
\sup_{x\in\Omega_r}|u(x)-l(x)|&\leq& \sup_{x\in\Omega_{r_0^B}}|u(x)-l(x)|\\
&\leq& \sup_{x\in\Omega_{r_0^B}}|u(x)-l_{B-A}(x)|+\sup_{x\in\Omega_{r_0^B}}|l_{B-A}(x)-l(x)|\\
&\leq& r_0^{B(1+\alpha)}+\frac{C_1}{1-r_0^\alpha}r_0^{B(1+\alpha)}~ \leq~ C r^{1+\alpha}. 
\end{eqnarray*}

Thus $u$ is $C^{1,\alpha}$ at the origin. By a standard covering argument, we conclude that $u\in C^{1,\alpha}(\overline{\Omega_{1/2}})$.

\end{proof}

\subsection{Parabolic problem}
We now turn to the parabolic problem \eqref{app par eq}.
\begin{equation*}
\left\{\begin{aligned}
& u_t=\operatorname{div}(v^{p-2}Du)+f(x,u) &&\text{in~}\Omega\times[0,\infty),\\
& u(x,0)=u_{0}(x)  &&\text{in~}\Omega,\\
& u_\nu=-\phi_{\varepsilon}(x)v^{1-q} &&\text{on~}\partial\Omega\times[0,\infty),
\end{aligned}\right.
\end{equation*}
where $q=1$ or $q=p-1$. When $q=1$, as mentioned before, since $u_t$ is continuous and bounded, we can treat it as an elliptic problem to derive the global H\"{o}lder regularity of  $Du$ in space for each fixed $t\geq0$. When $q=p-1$, the same result holds by applying the weak methods in \cite{L1988NA}.

For the H\"{o}lder regularity of $Du$ in $t$, according to the appendix of \cite{AP2018CCM} (with a slight modification), it suffices to show that for any $r>0$, $\operatorname{dist}(x_0,\partial\Omega)\geq 2r$ and $t_0\geq 2r^{1+\alpha}$, there exists a vector $q(x_0,t_0)$ such that 
\begin{equation}
\sup\limits_{\overline{Q_r(x_0,t_0)}}|u(x,t)-u(x_0,t_0)-q(x_0,t_0)\cdot(x-x_0)|\leq Cr^{1+\alpha},
\end{equation}
where $Q_r(x_0,t_0)\coloneqq B_r(x_0)\times (t_0-r^{1+\alpha},t_0]$.

Since $u$ is $C^{1,\alpha}$ in $x$ and Lipschitz in $t$, we have
\begin{eqnarray*}
&&\quad|u(x,t)-u(x_0,t_0)-q(x_0,t_0)\cdot(x-x_0)|\\
&&\leq |u(x,t)-u(x,t_0)|+|u(x,t_0)-u(x_0,t_0)-q(x_0,t_0)\cdot (x-x_0)|\\
&&\leq C|t-t_0|+C|x-x_0|^{1+\alpha}\leq Cr^{1+\alpha},
\end{eqnarray*}
hence $Du$ is also globally H\"{o}lder continuous in $t$.\vspace{4pt}

\begin{proof}[\textbf{Proof of Theorem \ref{thm1.1}}]~

$\bullet$ Existence.

Consider the following problem
\begin{equation}\label{eqDuneq0}
\left\{\begin{aligned}
& u_t=\operatorname{div}(v^{p-2}Du)+f(x,u) &&\text{in~}\Omega\times[0,\infty),\\
& u(x,0)=u_0(x)  &&\text{in~}\Omega,\\
& u_\nu=-\phi_{\varepsilon}(x)v^{1-q} &&\text{on~}\partial\Omega\times[0,\infty),
\end{aligned}\right.
\end{equation}
where 
$$
\phi_{\varepsilon}(x)\coloneqq-(|Du_0|^2+\varepsilon^2)^{\frac{q-1}{2}}D_{\nu}u_0(x).
$$
Theorem \ref{thm3.2} yields the existence, uniqueness and global $C^{1}$ estimate for the solution $u_{\varepsilon}$ of \eqref{eqDuneq0}. We now demonstrate the convergence $u_{\varepsilon}\to u$ as $ \varepsilon\to0$.

When $(x_0,t_0)\in\Omega\times[0,\infty)$, by Lemma \ref{lemstb}, $u$ is a viscosity solution of \eqref{par eq} in the interior.

When $(x_0,t_0)\in\partial\Omega\times[0,\infty)$, by Lemma \ref{par-vis-bdy-def-2}, we only need to consider test function $\varphi\in C^2$ with $|D\varphi(x_0,t_0)|\neq0$, therefore, the equation is non-singular. By the standard stability theory for boundary problems for degenerate equations \cite{CIL1992}, $u$ is a viscosity solution of \eqref{par eq} on the boundary.

Together with the $C^{1,\alpha}$ arguments in this and the previous subsections, we have proved that
\begin{itemize}
\item when $q=1$, there exists a viscosity solution $u$ of \eqref{par eq}, which is globally Lipschitz in $(x,t)$, and the spatial gradient $Du$ is globally H\"older in $(x,t)$. To be specific, for any $\alpha\in(0,\alpha_0)$, where $\alpha_0$ is the optimal exponent for interior $C^{1,\alpha_0}$ regularity of weak solutions to elliptic $p$-Laplacian equations (see Section 4.2), 
$$
|u(x,t)-u(x,s)|\leq M(|t-s|+|x-y|),\quad |Du(x,t)-Du(y,s)|\leq M(|x-y|^{\alpha}+|t-s|^{\frac{\alpha}{1+\alpha}}),
$$
for all $t,s\in[0,\infty)$, $x,y\in\overline{\Omega}$.\\
\item when $q=p-1$, the above results hold for $\alpha=\alpha_1
$, where $\alpha_1$ is the optimal exponent for the elliptic conormal problem considered in \cite{L1988NA}.\\
\item when $q\neq 1,p-1$, there exists a viscosity solution $u$ of \eqref{par eq}, which is globally Lipschitz in $(x,t)$, and the spatial gradient $Du$ is locally $C^{\alpha_0}$ in $x$ and globally $C^{\frac{\alpha_0}{1+\alpha_0}}$ in $t$. To be specific, 
$$
|u(x,t)-u(x,s)|\leq M(|t-s|+|x-y|),
$$
for all $t,s\in[0,\infty)$,  $x,y\in\overline{\Omega}$, and
$$
|Du(x,t)-Du(y,s)|\leq M(|x-y|^{\alpha_0}+|t-s|^{\frac{\alpha_0}{1+\alpha_0}}),
$$
for all $t,s\in[0,\infty)$ and $x,y\in\Omega'\subset\subset\Omega$.
\end{itemize}

$\bullet$ Uniqueness.

The uniqueness follows directly from the comparison principle Lemma \ref{lem cpp}.
\end{proof}

\section{Asymptotic behavior of solutions to the approximate equation \eqref{app par eq}}
In this section, for fixed small $\varepsilon>0$, we study the asymptotic behavior of the solutions to equation \eqref{app par eq}. Since all estimates are independent of $\varepsilon$, we can take the limit $\varepsilon\to0$ in the next section. We begin with the elliptic equation
\begin{equation}\label{app ell eq2}
\left\{\begin{aligned}
& -\operatorname{div}(v^{p-2}Du)=-\delta u +\varphi(x) &&\text{in~}\Omega,\\
&v^{q-1} u_\nu=-\phi_{\varepsilon}(x) &&\text{on~}\partial\Omega,
\end{aligned}\right.
\end{equation} 
where $\delta>0$ is another small constant. We have the following existence and uniqueness result, with estimates independent of both $\varepsilon$ and $\delta$.
\begin{theorem}\label{thm5.1}
Assume $\Omega\subset \mathbb R^n$ is a bounded strictly convex $C^{2,\beta}$ domain and $q>0$. Then for any $\varphi(x)$, $\phi_{\varepsilon}(x)\in C^{\infty}(\overline\Omega)$ and $\|\phi_{\varepsilon}\|_{C^2(\overline{\Omega})}\leq L$ , there exists a unique solution $u_{\varepsilon,\delta}$ of \eqref{app ell eq2}. Moreover, 
$$
\sup_{\Omega}|Du_{\varepsilon,\delta}|+\sup_{\Omega}|\delta u_{\varepsilon,\delta}|\leq C,
$$
where $C=C(n,p,q,L,\Omega)$ is independent of both $\varepsilon$ and $\delta$.
\end{theorem}
\begin{proof}
It is straightforward to verify that the boundary condition is oblique. Hence, by the result of Lieberman-Trudinger \cite{LT1986TAMS}, we may apply the Leray–Schauder fixed point theorem to the following family of boundary value problems, parametrized by $\tau\in[0,1]$,
$$\begin{cases}
\tau \left( -\operatorname{div}(v^{p-2}Du) +\delta u-\varphi(x) \right) \\
+ (1 - \tau) \left( -\operatorname{div}(v^{p-2}Du)  + \delta u \right) = 0 & \text{in } \Omega, \\[2pt]
v^{q-1}u_{\nu} = \tau \phi_{\varepsilon}(x) & \text{on } \partial \Omega,
\end{cases}$$
When $\tau = 0$, the function $u \equiv 0$ is the unique solution; we need to find a solution for $\tau = 1$. According to the Leray–Schauder fixed point theorem, the existence of a solution for $\tau=1$ follows once we establish a priori $C^0$ and $C^1$ estimates that are uniform in $\tau\in[0,1]$:
$$\sup_{\Omega}|Du_{\varepsilon,\delta}|+\sup_{\Omega}|\delta u_{\varepsilon,\delta}|\leq C,$$

Since $\delta>0$, the gradient estimate is the same as Theorem \ref{thm3.2}, so we omit it.
    
To obtain the $C^0$ estimate, we must carefully construct a barrier function due to the singularity that arises when $1 < p < 2$.

Define
$$w(x)\coloneqq\begin{cases}
-\left(\varepsilon_0^b-(\varepsilon_0-d)^b+\gamma\right)^\frac{1}{2}\quad & \text{if}\quad 0<d<\varepsilon_0,\\[4pt]
-\left(\varepsilon_0^b+\gamma\right)^\frac{1}{2}\quad& \text{if}\quad \varepsilon_0\leq d,
\end{cases}$$
where $d=d(x,\partial\Omega)$ is the distance function, which is therefore $C^2$ near $\partial\Omega$, and $\varepsilon_0,\gamma>0$ are small constants, $b>0$ is a large constant to be determined later, and all parameters are independent of $\varepsilon$ and $\delta$. 

We claim that for appropriate parameters, $w$ satisfies
\begin{itemize}
\item $w\in C^2(\overline{\Omega})$,
\item $F_{\varepsilon}[w]\leq C$ on $\overline{\Omega}$,
\item $(\varepsilon^2+|Dw|^2)^\frac{q-1}{2}D_{\nu}w< -\sup_{\partial\Omega}|\phi_{\varepsilon}(x)|$ on $\partial\Omega$,
\end{itemize}
where $F_{\varepsilon}[u]\coloneqq\operatorname{div}\left((\varepsilon^2+|Du|^2)^\frac{p-2}{2}Du\right)$, and the constant $C$ is independent of $\varepsilon$ and $\delta$.

Denote $A\coloneqq\varepsilon_0^b-(\varepsilon_0-d)^b+\gamma$, $B\coloneqq\varepsilon_0-d$, $\Omega_{\varepsilon_0}\coloneqq\{x\in\Omega~|~d(x)<\varepsilon_0\}$. In $\Omega_{\varepsilon_0}$, by direct computation we have
$$
w_i=-\frac{b}{2}A^{-\frac{1}{2}}B^{b-1}d_i,\qquad |Dw|=\frac{b}{2}A^{-\frac{1}{2}}B^{b-1},
$$
$$w_{ij}=\frac{1}{4}b^2A^{-\frac{1}{2}}B^{2b-2}d_id_j-\frac{1}{4}bA^{-\frac{1}{2}}B^{b-2}d_id_j-\frac{b}{2}A^{-\frac{1}{2}}B^{b-1}d_{ij}.$$
Hence, for $b > 2$, we have $w \in C^{2}(\overline{\Omega})$ since the second-order derivative of $w$ remains bounded. Furthermore, in $\Omega_{\varepsilon_0}$,
\begin{equation*}
\begin{aligned}
F_\varepsilon[w]=&\Big(\varepsilon^2+\frac{b^2}{4}A^{-1}B^{2b-2}\Big)^{\frac{p-2}{2}}\Big(\frac{b^2}{4}A^{-\frac{3}{2}}B^{2b-2}+\frac{b(b-1)}{2}A^{-\frac{1}{2}}B^{b-2}-\frac{b}{2}A^{-\frac{1}{2}}B^{b-1}\Delta d\Big)\\[2pt]
&+(p-2)\Big(\varepsilon^2+\frac{b^2}{4}A^{-1}B^{2b-2}\Big)^{\frac{p-4}{2}}\Big(\frac{b^4}{16}A^{-\frac{5}{2}}B^{4b-4}-\frac{b^3}{16}A^{-\frac{3}{2}}B^{3b-4}\Big),
\end{aligned}
\end{equation*}
where we have used the fact that $\sum_{i,j}d_{ij}d_id_j = d_{nn} = 0$ near $\partial\Omega$. Thus, with the choice $b > \frac{p}{p-1}$, the expression $F_{\varepsilon}[w]$ remains bounded by some constant $C$ independent of $\varepsilon$. Together with the fact that $F_{\varepsilon}[w] \equiv 0$ in $\Omega\setminus\Omega_{\varepsilon_0}$, this completes the proof of the second assertion.

To verify the last assertion, we note that on $\partial\Omega$, $A=\gamma$ and $B=\varepsilon_0$, whence
$D_{\nu}w = -|Dw| = -\frac{b}{2}\gamma^{-\frac{1}{2}}\varepsilon_0^{b-1}$.
Choosing $\gamma>0$ sufficiently small makes $D_{\nu}w$ negative enough so that, since $q>0$,
$$
\left(\varepsilon^2+|Dw|^2\right)^\frac{q-1}{2}D_{\nu}w< -\sup_{\partial\Omega}|\phi_{\varepsilon}(x)|.
$$
In summary, if we choose $b>\max\{2,\frac{p}{p-1}\}$, take $\varepsilon_0>0$ sufficiently small, and finally take $\gamma>0$ much smaller (depending on $\varepsilon_0$ and $b$), then $w$ satisfies the desired properties.

Based on the barrier function $w$, it follows easily that the minimum of $w-u_{\varepsilon,\delta}$ over $\overline\Omega$ occurs in $\Omega$ at some point $x_0$. Then $Dw(x_0)=Du_{\varepsilon,\delta}(x_0)$ and $D^2w(x_0)\geq D^2u_{\varepsilon,\delta}(x_0)$, hence
$$
C\geq F_{\varepsilon}[w](x_0)\geq F_{\varepsilon}[u_{\varepsilon,\delta}](x_0)=\delta u_{\varepsilon,\delta}(x_0)-\tau\varphi(x_0).
$$
Combining this with the boundedness of $Du_{\varepsilon,\delta}$, we obtain that $\delta u_{\varepsilon,\delta}$ is bounded from above, independently of $\varepsilon$, $\delta$ and $\tau$. The lower bound can be derived in a similar way.

Finally, uniqueness follows by applying Hopf’s lemma and the strong maximum principle to the difference $u_1-u_2$ of any two solutions; see Theorem 4.2 in \cite{WWX2019CPAA} for details.

\end{proof}
Next, we study the elliptic eigenvalue problem \eqref{app ell eq}.

\begin{proof}[\textbf{Proof of Theorem \ref{thm1.3}}]
Let $u_{\varepsilon,\delta}$ be the unique solution of \eqref{app ell eq2} as in Theorem \ref{thm5.1}. We have proved that
$$\sup_{\Omega}|Du_{\varepsilon,\delta}|+\sup_{\Omega}|\delta u_{\varepsilon,\delta}|\leq C.$$
Consider $\omega_{\varepsilon,\delta}=u_{\varepsilon,\delta}-\frac{\int_\Omega u_{\varepsilon,\delta}}{|\Omega|}$; then $\omega_{\varepsilon,\delta}$ satisfies 
\begin{equation}\label{app ell eq3}
\left\{\begin{aligned}
& \operatorname{div}(v^{p-2}D\omega_{\varepsilon,\delta})=\delta \omega_{\varepsilon,\delta} +\delta\frac{\int_\Omega u_{\varepsilon,\delta}}{|\Omega|}+\varphi(x) &&\text{in~}\Omega,\\[2pt]
& v^{q-1}D_{\nu}\omega_{\varepsilon,\delta}=-\phi_{\varepsilon}(x) &&\text{on~}\partial\Omega.
\end{aligned}\right.
\end{equation} 
Since 
$$\sup_{\Omega}|D\omega_{\varepsilon,\delta}|=\sup_{\Omega}|Du_{\varepsilon,\delta}|\leq C,$$
and $\omega_{\varepsilon,\delta}$ has a zero in $\Omega$, we obtain $|\omega_{\varepsilon,\delta}| \leq C$. Together with the previously proved estimates $|\delta u_{\varepsilon,\delta}| \leq C$, $\bigl|\delta\frac{\int_\Omega u_{\varepsilon,\delta}}{|\Omega|}\bigr| \leq C$, and $|D\omega_{\varepsilon,\delta}|\leq C$, the classical Schauder theory yields $|\omega_{\varepsilon,\delta}|_{C^{2,\alpha'}(\overline{\Omega})} \leq C$ for some $\alpha' \in (0,1)$. Passing to the limit $\delta \to 0$, a subsequence of $\omega_{\varepsilon,\delta}$ converges to $u_{\varepsilon}$ in $C^{2,\alpha}$ for any $0 < \alpha < \alpha'$, while $\delta \omega_{\varepsilon,\delta} + \delta\frac{\int_\Omega u_{\varepsilon,\delta}}{|\Omega|} \to \lambda_{\varepsilon}$. The pair $(\lambda_{\varepsilon}, u_{\varepsilon})$ thus solves \eqref{app ell eq}.  Uniqueness follows in the same way as in Theorem \ref{thm5.1}.

\end{proof}

\begin{proof}[\textbf{Proof of Theorem \ref{thm1.4}}]
The proof is the same as in Theorem 5.2 of \cite{WWX2019CPAA}, so we omit it.
    
\end{proof}

\section{Asymptotic behavior of solutions to the equation \eqref{par eq}}
As in Section 5, we first consider the elliptic eigenvalue problem \eqref{ell eq}, which can be viewed as the limit of \eqref{app ell eq}.\vspace{4pt}

\begin{proof}[\textbf{Proof of Theorem \ref{thm1.5}}]~

$\bullet$ Existence.

By the proof of Theorem \ref{thm1.3}, we know that $\lambda_{\varepsilon}$ is uniformly bounded in $\varepsilon$. Hence, by taking the limit of $u_{\varepsilon}$ and using the stability Lemma \ref{lemstb}, we obtain the existence of $u$ in $\Omega$. Moreover, the same arguments as in the proof of Theorem \ref{thm1.1} yield the convergence on $\partial\Omega$ and the global $C^{1,\alpha}$ regularity when $q=1$ or $p-1$. 

$\bullet$ Uniqueness.

When $q=p-1$, integration by parts yields the uniqueness of $\lambda_0$. Moreover, suppose $u$ and $v$ are two solutions of \eqref{ell eq}. Testing the equation 
$$
\operatorname{div}(|Du|^{p-2}Du)-\operatorname{div}(|Dv|^{p-2}Dv)=0
$$  
by $u-v$ and integrating by parts, we obtain $Du=Dv$.

For $p\geq2$, when $q=1$ or $q>0$ with $|\phi(x)|>0$ on $\partial\Omega$, the uniqueness of $\lambda_0$ will be shown in the proof of Theorem \ref{thm1.6} via the comparison principle Lemma \ref{lem cpp}.

\end{proof}

\begin{proof}[\textbf{Proof of Theorem \ref{thm1.6}}] Let $(\omega_0,\lambda_0)$ be a pair of solution in Theorem \ref{thm1.5}. Then for sufficiently large $A>0$,  $u_1(x,t)\coloneqq\omega_0(x)-\lambda_0t-A$ is a subsolution of \eqref{par eq}, and $u_2(x,t)\coloneqq\omega_0(x)-\lambda_0t+A$ is a supersolution of \eqref{par eq}. By Lemma \ref{lem cpp}, we have 
$$
u_2(x,t)=\omega_0(x)-\lambda_0 t+A \geq u(x,t)\geq \omega_0(x)-\lambda_0 t-A=u_1(x,t),
$$
hence
$$
\lim_{t\to\infty}\frac{u(x,t)}{t}=-\lambda_0,
$$
and $\lambda_0$ is unique.
\end{proof}

If the boundary data are zero, we can derive a more delicate asymptotic behavior as in \cite{JKMT2011JMPA, HTZ2019GF}.

\begin{proof}[\textbf{Proof of Theorem \ref{thm1.7}}]

Let $u_{\varepsilon}$ be the solution of \eqref{app par eq2} and denote $v=\sqrt{\varepsilon^2+|Du_\varepsilon|^2}$. We consider the following Lyapunov function
$$
I^{\varepsilon}(t)\coloneqq\int_{\Omega}v^{\frac{p}{2}}\,dx.
$$

By direct calculation,
$$
\frac{\mathrm{d}}{\mathrm{d}t}\int_{\Omega}v^{p}\,dx=p\int_{\Omega}v^{p-2} Du_\varepsilon\cdot Du_{\varepsilon,t}\,dx\\
=-p\int_{\Omega}u_{\varepsilon,t}\,\mathrm{div}(v^{p-2}Du_\varepsilon )\,dx\\
=-p\int_{\Omega}|u_{\varepsilon,t}|^2\, dx.
$$

Integrating in $t$, it follows that
$$
\int_{0}^{T}\int_{\Omega}|u_{\varepsilon,t}|^{2}\,dxdt =\frac{1}{p}\int_{\Omega}(v(x,0)^p-v(x,T)^p)\, dx.
$$

Therefore,
$$
\limsup\limits_{\varepsilon\to 0}\int_{0}^{T}\int_{\Omega}|u_{t}^{\varepsilon}|^{2}\,dxdt\leqslant C,
$$
where $C$ is a constant independent of $\varepsilon\in(0,1)$ and $T>0$. Hence, we know $u_{t}^{\varepsilon}\rightharpoonup u_{t}$ weakly in $L^{2}(\overline{\Omega}\times[0,T])$ as $\varepsilon\to 0$ for each fixed $T>0$.

By lower semi-continuity of weak convergence,
\begin{eqnarray}\label{utbdd}
\int_0^{\infty}\int_{\Omega}|u_t|^2\, dxdt=\lim_{T\to\infty}\int_0^{T}\int_{\Omega}|u_t|^2\, dxdt\leq \lim_{T\to\infty}\liminf_{\varepsilon\to0}\int_0^{T}\int_{\Omega}|u_t^{\varepsilon}|^2\, dxdt\leq C,
\end{eqnarray}
where the constant $C$ is independent of $\varepsilon$ and $T$.

For any sequence $\{t_k\} \to \infty$, by the Arzelà-Ascoli theorem, there exist a subsequence $\{t_{k_j}\}$ and a Lipschitz continuous function $w$ such that
\begin{eqnarray}\label{ukj}
u_{k_j}(x,t) \coloneqq u(x, t + t_{k_j}) \to w(x,t),
\end{eqnarray}
locally uniformly on $\overline{\Omega} \times [0,\infty)$, hence uniformly on $\overline{\Omega} \times [0,T]$ for every $T > 0$. By the arguments in the proof of Theorem \ref{thm1.1}, we know $w$ is a viscosity solution of 
\begin{equation*}
\left\{\begin{aligned}
& w_t=\operatorname{div}(|Dw|^{p-2}Dw) &&\text{in~}\Omega\times[0,\infty),\\
& w_\nu=0 &&\text{on~}\partial\Omega\times[0,\infty).
\end{aligned}\right.
\end{equation*}
Moreover, by \eqref{utbdd}, it follows that
$$
\int_{0}^{1} \int_{\Omega} (u_{k_j})_t^2 \, dx \, dt = \int_{t_{k_j}}^{1 + t_{k_j}} \int_{\Omega} (u_t)^2 \, dx \, dt \to 0,
$$
as $ j \to \infty $, which implies
$$
(u_{k_j})_t \rightharpoonup0,
$$
weakly in $ L^2(\overline{\Omega} \times [0,1]) $ as $ j \to \infty $. Recall that \eqref{ukj} implies
$$
(u_{k_j})_t \rightharpoonup w_t,
$$
weakly in $ L^2(\overline{\Omega} \times [0,1]) $ as $ j \to \infty $. Therefore, $ w_t = 0 $ in the weak sense, and $ w $ is constant in $ t $. Consequently, $ w $ is a viscosity solution of the following elliptic equation
\begin{equation*}
\left\{\begin{aligned}
& \operatorname{div}(|Dw|^{p-2}Dw)=0 &&\text{in~}\Omega,\\
& w_\nu=0 &&\text{on~}\partial\Omega.
\end{aligned}\right.
\end{equation*}
By Theorem \ref{thm-global-general-elliptic}, $w$ is also a weak solution, and thus constant by classical weak theory.

Finally, by Lemma \ref{lem cpp}, it follows easily that $w$ is independent of the choice of the subsequence.
\end{proof}

\section{The prescribed contact angle problem}
In this section, we consider the prescribed contact angle problem, which corresponds to $q=0$ in \eqref{par eq}. Specifically, we consider the following problem
\begin{equation}\label{q=01}
\left\{\begin{aligned}
& u_t=\operatorname{div}(|Du|^{p-2}Du)+f(x,u) &&\text{in~}\Omega\times[0,\infty),\\
& u(x,0)=u_{0}(x)  &&\text{in~}\Omega,\\
& u_\nu=-\phi(x)|Du| &&\text{on~}\partial\Omega\times[0,\infty),
\end{aligned}\right.
\end{equation}
and its approximate problems
\begin{equation}\label{q=02}
\left\{\begin{aligned}
& u_t=\operatorname{div}(v^{p-2}Du)+f(x,u) &&\text{in~}\Omega\times[0,\infty),\\
& u(x,0)=u_{0}(x)  &&\text{in~}\Omega,\\
& u_\nu=-\phi(x)|Du| &&\text{on~}\partial\Omega\times[0,\infty).
\end{aligned}\right.
\end{equation}
However, we can only derive results corresponding to the case $q>0$ under the nearly vertical condition
\begin{equation*}
\|\phi\|_{C^2(\partial\Omega)}\leq \varepsilon_1,
\end{equation*}
where $\varepsilon_1>0$ is a small constant depending only on $n,p,L,\Omega.$

Since the approach is very similar, all results established in the previous sections for the case $q>0$ extend to equations satisfying the nearly vertical condition \eqref{nvc} when $q=0$. Therefore, we only present the proof of the following crucial gradient estimate. The small constant $c>0$ and the large constant $C>0$ are independent of $\varepsilon$, $\varepsilon_1$, $T$, and $|u|_{C^0}$, and may change from line to line. 
\begin{theorem}\label{thm7.1}
Assume $\Omega\subset \mathbb{R}^n$ is a bounded strictly convex $C^{2,\beta}$ domain. If $f$ satisfies \eqref{f} and $\phi$ satisfies \eqref{nvc}, then the solution of \eqref{q=02} satisfies 
$$
\sup_{\overline{\Omega}\times[0,T]}|Du|\leq C(n,p,L,\Omega).
$$
\end{theorem}
\begin{proof}
We choose auxiliary function
$$
\Phi(x,t)=\log w +\alpha h,
$$
where
$$
w=v-\sum_{l=1}^n \phi u_l h_l.
$$

Suppose $\Phi(x,t)$ attains its maximum at $(x_0,t_0)\in \overline{\Omega}\times[0,T]$. Without loss of generality, we may assume $t_0>0$ and $|Du(x_0,t_0)|$ is large enough. \vspace{4pt}

\textbf{Case 1:} $x_0\in \partial\Omega$.

Using the same notation and following an argument similar to that in \textbf{Case 1} of Theorem \ref{thm3.2}, we solve $u_{ni}$ as
\begin{eqnarray*}
u_{ni}=\frac{1}{1-\phi^2|Du|^{-1}v}\left[\phi |Du|^{-1}v\left(\phi_i u_n-\phi\sum_{l}u_l h_{li}\right)-\phi_i |Du|\right].   
\end{eqnarray*}
It follows that
\begin{eqnarray}
0&\geq& \Phi_{\nu}(x_{0},t_{0})
= w^{-1}v^{-1}\sum_{k=1}^{n-1}u_{k}(u_{nk}+\sum_{i=1}^{n-1}u_{i}\tilde{\kappa}_{ik})+\frac{1}{w}[\phi_{n}u_{n}-\phi\sum_{l=1}^{n}u_{l}h_{ln}]-\alpha\notag\\
&=& \frac{w^{-1}v^{-1}}{1-\phi^{2}|Du|^{-1}v}\sum_{k=1}^{n-1}u_{k}\big{[}\phi|Du|^{-1}v(-\phi\sum_{l=1}^{n}u_{l}h_{lk}+\phi_{k}u_{n})-|Du|\phi_{k}\big{]}\notag\\
&&+\frac{1}{w}v^{-1}\sum_{k,i=1}^{n-1}u_{k}u_{i}\kappa_{ik}+\frac{1}{w}[\phi_{n}u_{n}-\phi\sum_{l=1}^{n}u_{l}h_{ln}]-\alpha\label{q=0bdy est}\\
&\geq&\frac{\kappa_0}{2}-\alpha-C\varepsilon_1.\notag
\end{eqnarray}
It follows that $|Du(x_0,t_0)|\leq C$ provided $\varepsilon_1>0$ is small enough and selecting $\alpha>0$ sufficiently small.\vspace{4pt}

\textbf{Case 2:} $x_0\in\Omega$.

 We choose coordinates such that $|Du|=u_1$, and the matrix $\{u_{\alpha\beta}\}_{\alpha,\beta>1}$ is diagonal at that point. Since $\Phi$ attains its maximum at $(x_0,t_0)$, we have
\begin{eqnarray*}
0&\leq& \Phi_t(x_0,t_0)=w^{-1}\sum_{k=1}^n(v^{-1}u_k-\phi h_k)u_{tk},\\ 
0&=&\Phi_i(x_0,t_0)=w^{-1}w_i+\alpha h_i,
\end{eqnarray*}
and
\begin{eqnarray*}
0\geq \Phi_{ii}(x_0,t_0)=w^{-1}w_{ii}-\left|\frac{w_i}{w}\right|^2+ \alpha h_{ii}.
\end{eqnarray*}
Hence
\begin{eqnarray*}
0&\geq&a^{ii}\Phi_{ii}(x_0,t_0)-\Phi_t(x_0,t_0)\\
&=&w^{-1}a^{ii}w_{ii}-a^{ii}\left|\frac{w_i}{w}\right|^2+\alpha a^{ii}h_{ii}-w^{-1}\sum_{k=1}^n(v^{-1}u_k-\phi h_k)u_{tk}\\
&=&w^{-1}a^{ii}w_{ii}+A a^{ii}\left|\frac{w_i}{w}\right|^2-w^{-1}\sum_{k=1}^n(v^{-1}u_k-\phi h_k)u_{tk}\\[2pt]
&&+\alpha a^{ii}h_{ii}-(A+1)\alpha^{2}a^{ii}|h_i|^2\\[4pt]
&=& \mathrm{I}+\mathrm{II},
\end{eqnarray*}
where $\mathrm{I}$ denotes the three terms in the third line, and $\mathrm{II}$ denotes the two terms in the fourth line. The large constant $A$ is independent of $\varepsilon$, $T$, $|u|_{C^0}$ and will be chosen later.

By direct computation, we have
\begin{eqnarray*}
w_i=v^{-1}\sum_ku_ku_{ki}-\phi_i\sum_ku_kh_k-\phi \sum_k u_{ki}h_k-\phi \sum_ku_kh_{ki},
\end{eqnarray*}
and
\begin{eqnarray*}
w_{ii}&=&-v^{-3}\left(\sum_k u_ku_{ki}\right)^2+v^{-1}\sum_{k}u_{ki}^2+v^{-1}\sum_ku_ku_{kii}\\
&&-\phi_{ii}\sum_k u_kh_k-2\phi_i\sum_k u_{ki}h_k-2\phi_{i}\sum_ku_kh_{ki}\\
&&-\phi \sum_k u_{kii}h_k-2\phi \sum_k u_{ki}h_{ki}-\phi \sum_k u_k h_{kii}.
\end{eqnarray*}
Denote $M=\max|D^2u(x_0,t_0)|$, $M_1=\max_i|u_{1i}(x_0,t_0)|$, $M_2=\max_{\alpha\neq1}|u_{\alpha\alpha}(x_0,t_0)|$. Note that
\begin{eqnarray*}
w_i&=&\sum_k(v^{-1}u_k-\phi h_k)u_{ki}-(\phi_i h_1+\phi h_{1i})u_1\\
&=&(v^{-1}u_1-\phi h_1)u_{1i}-\sum_{k\neq1}\phi h_k u_{ki}-(\phi_i h_1+\phi h_{1i})u_1,
\end{eqnarray*}
hence 
$$
|w_i|^2\geq \frac{1}{2}u_{1i}^2-C\varepsilon_1M^2-C\varepsilon_1v^2,
$$
and
\begin{eqnarray}\label{q=0w-2aiiwi2}
w^{-2}a^{ii}|w_i|^2\geq cv^{p-4}M_1^{2} -C\varepsilon_1v^{p-4}M_2^2-C\varepsilon_1v^{p-2}.
\end{eqnarray}

Now compute the remaining part of $\mathrm{I}$.
\begin{eqnarray*}
&&a^{ii}w_{ii}-\sum_k v^{-1}(u_k-\phi h_k)u_{tk}\\
&\geq&c v^{p-3}M_2^2+\sum_k(v^{-1}u_k-\phi h_k)(a^{ij}u_{ijk}-u_{tk})-C\varepsilon_1v^{p-2}M-C\varepsilon_1v^{p-1}\\
&\geq& c v^{p-3}M_2^2-C\varepsilon_1v^{p-2}M-C\varepsilon_1v^{p-1}+\sum_k (v^{-1}u_k-\phi h_k)(-f_k-f_u u_k-a^{ij}_{~,k}u_{ij}),
\end{eqnarray*}
where we have used differentiating equation \eqref{app par eq}. Recall that $f_u \leq 0$; from the computation \eqref{aijkuij} and the Cauchy inequality, we further have
\begin{eqnarray*}
&&a^{ii}w_{ii}-\sum_k v^{-1}(u_k-\phi h_k)u_{tk}\\
&\geq& c v^{p-3}M_2^2-C\varepsilon_1v^{p-2}M-C\varepsilon_1v^{p-1}-C-C\varepsilon_1\sum_{k\neq1}\left|a^{ij}_{~, k}u_{ij}\right|-C\left|a^{ij}_{~,1}u_{ij}\right|\\[2pt]
&\geq& c v^{p-3}M_2^2-C\varepsilon_1v^{p-2}M-C\varepsilon_1v^{p-1}-C-C\varepsilon_1v^{p-3}M_1^2-Cv^{p-3}M_1M.\label{q=0aiiwii}
\end{eqnarray*}

Hence
\begin{eqnarray}
&&w^{-1}a^{ii}w_{ii}-\sum_k w^{-1}v^{-1}(u_k-\phi h_k)u_{tk}\notag\\
&\geq& c v^{p-4}M_2^2-C\varepsilon_1v^{p-3}M-C\varepsilon_1v^{p-2}-Cv^{-1}-C\varepsilon_1v^{p-4}M_1^2-Cv^{p-4}M_1M.\label{q=0w-1aiiwii}
\end{eqnarray}
Recalling \eqref{q=0w-2aiiwi2} and choosing $A=\varepsilon_1^{-\frac{1}{3}}$ with $\varepsilon_1>0$ sufficiently small, we have
\begin{eqnarray}\label{q=0I}
\mathrm{I}\geq -C\varepsilon_1^{\frac{2}{3}}v^{p-2}-Cv^{-1}.
\end{eqnarray}
Finally, we select $\alpha=\varepsilon_1^{\frac{1}{2}}>0$ small enough. It follows from the strict convexity of $h$ and the fact that $v$ is sufficiently large that
\begin{eqnarray}\label{q=0II}
\mathrm{II}\geq c\varepsilon_1^{\frac{1}{2}}v^{p-2}.
\end{eqnarray}
Combining \eqref{q=0I} and \eqref{q=0II} yields
$$
0\geq a^{ii}\Phi_{ii}(x_0,t_0)-\Phi_t(x_0,t_0)\geq (c\varepsilon_1^{\frac{1}{2}}-C\varepsilon_1^{\frac{2}{3}})v^{p-2}-Cv^{-1}.
$$
Therefore, $|Du(x_0,t_0)|\leq C$ provided $\varepsilon_1>0$ is sufficiently small.

Together with \textbf{Case 1}, we obtain
$$\sup_{\overline{\Omega}\times[0,T]}|Du|\leq C(n,p,L,\Omega).$$
\end{proof}

\begin{proof}[\textbf{Proof of Theorem \ref{thm1.8}}]

We only show the approximating procedure, since, based on the crucial gradient estimate in Theorem \ref{thm7.1}, the other results are very similar to those for the case $q>0$. 

Since the boundary condition is already Lipschitz in $Du$, we do not need to modify the boundary condition. Consider approximate problems \eqref{q=02}, Theorem \ref{thm7.1} yields the existence, uniqueness and global gradient estimate for the solution $u_{\varepsilon}$ of \eqref{q=02}. Then, by the same arguments as in the proof of Theorem \ref{thm1.1}, we know that, up to a subsequence, $u_{\varepsilon}\to u$ and $u$ is a viscosity solution of \eqref{q=01}. Moreover, uniqueness follows directly from Lemma \ref{lem cpp}.
    
\end{proof}

\section{Non-convex case}
In this section, we consider the parabolic $p$-Laplacian equations on non-convex domains. We first introduce some constants which describe the geometric properties of $\Omega$.

For a given point $y \in \partial \Omega$, we define $K_0(y)$ as the supremum of the radii of interior tangent balls at $y$, and define $C_0(y)$ as the maximum eigenvalue of the matrix $-\kappa_{ij}(y)$, where $\kappa_{ij}$ is the second fundamental form of $\partial\Omega$. We further denote
\begin{equation}\label{C0K0}
\begin{cases}
C_0 = \sup\{C_0(y) : y \in \partial \Omega\}, \\[2pt]
K_0 = \inf\{K_0(y) : y \in \partial \Omega\}.
\end{cases}
\end{equation}
In view of the strictly convex case, we may assume $C_{0}\geq 0$.
\subsection{The case $q>0$}

Following the ideas of \cite{JKMT2011JMPA,J2023CVPDE}, and in order to satisfy the Lipschitz regularity of right-hand side in $Du$, we denote $\tilde p\coloneqq\max\{2,p\}$ and consider the following capillary-type problem
\begin{equation}\label{nc q>01}
\left\{\begin{aligned}
& u_t=\operatorname{div}(|Du|^{p-2}Du)+a(x,u)|Du|^{\tilde p-1}+f(x,u) &&\text{in~}\Omega\times[0,\infty),\\
& u(x,0)=u_0(x)  &&\text{in~}\Omega,\\
& |Du|^{q-1}u_\nu=-\phi(x) &&\text{on~}\partial\Omega\times[0,\infty),
\end{aligned}\right.
\end{equation}
and the approximate problems
\begin{equation}\label{nc q>02}
\left\{\begin{aligned}
& u_t=\operatorname{div}(v^{p-2}Du)+a(x,u)v^{\tilde p-1}+f(x,u) &&\text{in~}\Omega\times[0,\infty),\\
& u(x,0)=u_{0}(x) &&\text{in~}\Omega,\\
& u_\nu=-\phi_{\varepsilon}(x)v^{1-q} &&\text{on~}\partial\Omega\times[0,\infty),
\end{aligned}\right.
\end{equation}
where $a(x,u)$ is a forcing function satisfying 
\begin{equation}\label{a1}
 a\in C^{1,\beta}(\overline{\Omega}\times \mathbb{R}),\quad a_{z}(x,z)\leq0
\end{equation}
and the following forcing condition holds for some $c_1=c_1(n, p, q, L, \Omega)>0$ small and $C_1=C_1(n, p, q, L, \Omega)>0$ large:
\begin{align}\label{a2}
\begin{cases}
c_1a(x,u)^2-|Da(x,u)|-C_1|a(x,u)|\gg 1 \quad &\text{if } p\geq2,\\
|a(x,u)|>0 \quad &\text{if } 1<p<2.
\end{cases}
\end{align}
\begin{remark}\label{rmk2}
Since the proofs in previous sections extend analogously to equations with a forcing term $a(x,u)$, all preceding results for strictly convex domains remain valid provided the global gradient estimate holds. Hence, we focus on establishing the following crucial gradient estimate. Moreover, as the case $1<p<2$ is more straightforward, we present the proof solely for $p\geq2$. 
\end{remark}
In the following, the small constant $c>0$ and the large constant $C>0$ are independent of $\varepsilon$, $T$, and $|u|_{C^0}$, and may change from line to line. 
\begin{theorem}\label{thm8.1}
Assume $\Omega\subset \mathbb{R}^n$ is a bounded $C^{2,\beta}$ domain, $p>1$, $q>0$. If $f$ satisfies \eqref{f} and the forcing term $a(x,u)$ satisfies \eqref{a1} and \eqref{a2}, then the solution of \eqref{nc q>02} satisfies 
$$
\sup_{\overline{\Omega}\times[0,T]}|Du|\leq C(n,p,q,L,\Omega).
$$
\end{theorem}
\begin{proof}
By abuse of notation, we write $\phi$ in place of $\phi_{\varepsilon}$, since they enjoy similar properties. Following the ideas of \cite{JKMT2011JMPA} and \cite{J2023CVPDE}, we choose auxiliary function
$$
w=v^{q+1}-(q+1)\sum_{l=1}^n \phi u_l h_l.
$$
Suppose $w(x,t)$ attains its maximum at $(x_0,t_0)\in \overline{\Omega}\times[0,T]$.

Without loss of generality, we may assume $t_0>0$ and $|Du(x_0,t_0)|$ is large enough. We use the same notation as in the proof of Theorem \ref{thm3.2}.\vspace{4pt}

\textbf{Case 1:} $x_0\in\Omega$.

 We choose coordinates such that $|Du|=u_1$ and the matrix $\{u_{\alpha\beta}\}_{\alpha,\beta>1}$ is diagonal at that point. Since $w$ attains its maximum at $(x_0,t_0)$, we have
\begin{eqnarray*}
0&\leq& \frac{w_t(x_0,t_0)}{q+1}=\sum_{k=1}^n(v^{q-1}u_k-\phi h_k)u_{tk},\\
0&=&\frac{w_i(x_0,t_0)}{q+1}=v^{q-1}\sum_k u_ku_{ki}-\phi_i\sum_k u_kh_k-\phi \sum_k u_{ki}h_k-\phi \sum_k u_kh_{ki},
\end{eqnarray*}
and
\begin{eqnarray*}
0~\geq~ \frac{w_{ii}(x_0,t_0)}{q+1}&=&(q-1)v^{q-3}(u_1u_{1i})^2+v^{q-1}\sum_k u_{ki}^2+v^{q-1}\sum_k u_ku_{kii}\\
&&-\phi_{ii}\sum_k u_kh_k-2\phi_i\sum_k u_{ki}h_k-2\phi_{i}\sum_k u_kh_{ki}\\
&&-\phi \sum_k u_{kii}h_k-2\phi \sum_k u_{ki}h_{ki}-\phi \sum_k u_k h_{kii}.
\end{eqnarray*}
Hence
\begin{eqnarray}
0&\geq&(q+1)^{-1}\left(a^{ii}w_{ii}(x_0,t_0)-w_t(x_0,t_0)\right)\notag\\[8pt]
&=&(q-1)v^{q-3}u_1^2a^{ii}u_{1i}^2+v^{q-1}a^{ii}\sum_k u_{ki}^{2}\notag\\
&&+\sum_k(v^{q-1}u_k-\phi h_k)(a^{ij}u_{ijk}-u_{kt})\notag\\
&&-u_1 a^{ii}(\phi h_1)_{ii}-2a^{ii}\sum_k(\phi h_k)_i u_{ki}\label{nc-q>0-case1-I+II+III}\\[-4pt]
&=& \mathrm{I} +\mathrm{II} +\mathrm{III},\notag
\end{eqnarray}
where $\mathrm{I}$ denotes the two terms in the second line, $\mathrm{II}$ denotes the term in the third line, and $\mathrm{III}$ denotes the two terms in the fourth line.

Denote $M=\max|D^2u(x_0,t_0)|$, $M_1=\max_{\alpha\neq1} |u_{1\alpha}(x_0,t_0)|$, $M_2=\max_{\alpha\neq1}|u_{\alpha\alpha}(x_0,t_0)|$. Before handling these terms, we claim the following assertion:

\textbf{Claim 1.} $\sum_{\alpha\neq1}|u_{\alpha\alpha}|^{2}\geq ca(x,u)^2v^2$.

Suppose \textbf{Claim 1} is false. From the equation  
$$u_t=\operatorname{div}(|Du|^{p-2}Du)+a(x,u)|Du|^{p-1}+f(x,u)$$
and the boundedness of $|u_{t}|$, we have 
$$
|u_{11}|\geq cv.
$$
Then, using $0=w_1$, there exists $l>1$ such that
$$
|u_{1l}|\geq cv^{1+q}.
$$
Furthermore, from $0=w_l$, we obtain
$$
|u_{ll}|\geq cv^{1+2q},
$$
which yields a contradiction. Hence, \textbf{Claim 1} holds. 

Now returning to \eqref{nc-q>0-case1-I+II+III}, by direct computation, we have
\begin{eqnarray*}
\mathrm{I}~\geq~ c v^{p+q-3}M^2, \quad \mathrm{III}~\geq~ -Cv^{p-1}-Cv^{p-2}M,
\end{eqnarray*}
and
\begin{eqnarray*}
\mathrm{II}&=&\sum_k \left(v^{q-1}u_k-\phi h_k\right)\bigl(-f_k-f_u u_k-a_k v^{p-1}-a_u u_k v^{p-1}\\
&&\hspace{100pt}-(p-1)a(x,u)v^{p-3}u_1 u_{1k}-a^{ij}_{~,k}u_{ij}\bigr)\\[6pt]
&\geq& -a_u v^{p+q}-|Da|v^{p+q-1}-C\left(v^{q}+v^{p-1}\right)\\[6pt]
&&-\sum_k(v^{q-1}u_k-\phi h_k)\left((p-1)a(x,u)v^{p-3}u_1u_{1k}+a^{ij}_{~,k}u_{ij}\right).
\end{eqnarray*}

Recalling the critical equation $0=w_i$ and the computation \eqref{aijkuij}, we have
\begin{eqnarray*}
(p-1)a(x,u)v^{p-3}u_1\sum_k(v^{q-1}u_k-\phi h_k)u_{1k}~\leq~ C|a(x,u)|v^{p-1},
\end{eqnarray*}
and
\begin{eqnarray}
\sum_k(v^{q-1}u_k-\phi h_k)a^{ij}_{~,k}u_{ij}&=&(p-2)v^{p-4}u_1^2\Delta u(\phi h_1)_1+(p-2)(p-4)v^{p-6}u_1^4u_{11}(\phi h_1)_1\notag\\
&&+2(p-2)v^{p-4}u_1^2\sum_i u_{1i}(\phi h_1)_i~\leq~ Cv^{p-2}M.\label{nc-q>0-aijkuij}
\end{eqnarray}
Therefore, we have
\begin{eqnarray*}
\mathrm{II}\geq-|Da|v^{p+q-1}-C(v^q+v^{p-1})-C|a(x,u)|v^{p-1}-Cv^{p-2}M.
\end{eqnarray*}

Hence, by \textbf{Claim 1} and the Cauchy inequality, we derive
\begin{eqnarray*}
0&\geq& (q+1)^{-1}\left(a^{ii}w_{ii}(x_0,t_0)-w_t(x_0,t_0)\right)=\mathrm{I}+\mathrm{II}+\mathrm{III}\\[4pt]
&\geq& cv^{p+q-3}M^{2}-|Da|v^{p+q-1}-C|a(x,u)|v^{p-1}-Cv^{p-2}M-C(v^q+v^{p-1})\\[4pt]
&\geq&\left[c\cdot a(x,u)^2-|Da|-C\right]v^{p+q-1}-C\left(v^{q}+v^{p-1}\right).
\end{eqnarray*}
It follows from condition \eqref{a2} that $|Du(x_0,t_0)|\leq C.$\vspace{4pt}

\textbf{Case 2:} $x_0\in \partial\Omega$.

By the same argument as in \textbf{Case 1} of the proof of Theorem \ref{thm3.2}, we obtain
\begin{equation}\label{nc-q>0-case2-wn}
    w_{\nu}\geq -\left(C_{0}+\frac{\delta_0}{4}\right)w.
\end{equation}

Let $B = B(x_c, K_0)$ be the open ball with center $x_c \coloneqq x_0 + K_0 \vec{\nu}(x_0)$ such that $B \subseteq \Omega$ and $\overline{B} \cap (\mathbb{R}^n \setminus \Omega) = \{x_0\}$. Denote $\psi=\rho w$, where $\rho$ is a multiplier as in \cite{JKMT2011JMPA,J2023CVPDE}
$$
\rho(x) \coloneqq -\frac{L_0}{2K_0} \left|x - x_c\right|^2 + \frac{L_0K_0}{2} + 1.
$$
If we choose $L_0>C_0$, then clearly
$$
\psi_{\nu}(x_0,t_0)>0.
$$
Moreover, by the choice of $\rho$, we have
$$
\psi(z,t)\leq w(z,t)\leq w(z_0,t_0)=\psi(x_0,t_0) \quad \text{for } (z,t)\in(\overline\Omega\backslash B)\times[0,T],
$$
and therefore,
$$
\max_{\overline\Omega\times[0,T]}\psi= \max_{\overline B\times[0,T]}\psi>\psi(x_0,t_0)=w(x_0,t_0).
$$
It follows that the maximum of $\psi$ in $\overline{B}\times [0,T]$ must occur at a point $(x_1,t_1)\in B\times(0,T]$. 

We may assume that $t_1>0$, and choose coordinates such that $|Du|=u_1$ and the matrix $\{u_{\alpha\beta}\}_{\alpha,\beta>1}$ is diagonal at $(x_1,t_1)$. Since $\psi$ attains its local maximum, we have
\begin{eqnarray*}
0\leq \psi_t(x_1,t_1)&=&(q+1)\rho(x_1)\sum_{k=1}^n(v^{q-1}u_k-\phi h_k)u_{tk},\\
0=\psi_i(x_1,t_1)&=&(q+1)\rho(x_1) \Big[v^{q-1}\sum_k u_ku_{ki}-\phi_i\sum_k u_kh_k\\
&&\hspace{70pt}-\phi \sum_k u_{ki}h_k-\phi \sum_k u_kh_{ki}\Big]-\frac{L_0}{K_0}(x_{1,i}-x_{c,i})w,
\end{eqnarray*}
and
\begin{eqnarray*}
0\geq \psi_{ii}(x_1,t_1)&=&(q+1)\rho(x_1)\Big[(q-1)v^{q-3}(u_1u_{1i})^2+v^{q-1}\sum_{k}u_{ki}^2+v^{q-1}\sum_ku_ku_{kii}\\
&&-\phi_{ii}\sum_k u_kh_k-2\phi_i\sum_k u_{ki}h_k-2\phi_{i}\sum_k u_kh_{ki}-\phi \sum_k u_{kii}h_k\\
&&-2\phi \sum_k u_{ki}h_{ki}-\phi \sum_k u_k h_{kii}\Big]-\frac{2L_0}{K_0}w_i\left(x_{1,i}-x_{c,i}\right)-\frac{L_0}{K_0}w.
\end{eqnarray*}
Hence
\begin{eqnarray}
0&\geq& \frac{1}{(q+1)\rho}\left(a^{ii}\psi_{ii}(x_1,t_1)-\psi_{t}(x_1,t_1)\right)\notag\\[4pt]
&=&\Big[(q-1)v^{q-3}u_1^2a^{ii}u_{1i}^2+v^{q-1}a^{ii}\sum_ku_{ki}^2+\sum_k \left(v^{q-1}u_k-\phi h_k\right)\left(a^{ij}u_{ijk}-u_{kt}\right)\notag\\
&&-u_1a^{ii}(\phi h_1)_{ii}-2a^{ii}\sum_k (\phi h_k)_iu_{ki} \Big]-\frac{L_0}{(q+1)K_0\rho}\left(v^{q+1}-(q+1)\phi u_1h_1\right)\sum_ia^{ii}\notag\\
&&-\frac{2(q+1)L_0}{(q+1)K_0\rho}\left(x_{1,i}-x_{c,i}\right)a^{ii}\left(v^{q-1}u_{1}u_{1i}-(\phi h_1)_iu_1-\phi \sum_k  h_k u_{ki}\right)\label{nc-q>0-case2-I+II+III}\\[4pt]
&=&\mathrm{I}+\mathrm{II}+\mathrm{III},\notag
\end{eqnarray}
where
$$
\mathrm{I}=(q-1)v^{q-3}u_1^2a^{ii}u_{1i}^2+v^{q-1}a^{ii}\sum_ku_{ki}^2,
$$
$$
\mathrm{II}=\sum_k \left(v^{q-1}u_k-\phi h_k\right)\left(a^{ij}u_{ijk}-u_{kt}\right),
$$
and $\mathrm{III}$ denotes all the other terms.

Denote $M=\max|D^2u(x_1,t_1)|$, $M_1=\max_{\alpha\neq1} |u_{1\alpha}(x_1,t_1)|$, $M_2=\max_{\alpha\neq1}|u_{\alpha\alpha}(x_1,t_1)|$. Similar to \textbf{Case 1}, we can prove 
\begin{equation}\label{nc-q>0-case2-uaa}
\sum_{\alpha\neq1}|u_{\alpha\alpha}|^2\geq ca(x,u)^2v^2.
\end{equation}

Now returning to \eqref{nc-q>0-case2-I+II+III}, by direct computation, we have
$$
\mathrm{I}\geq cv^{p+q-3}M^2,\quad \mathrm{III}\geq -Cv^{p+q-1}-Cv^{p+q-2}M,
$$
and 
\begin{eqnarray*}
\mathrm{II}&=&\sum_k \left(v^{q-1}u_k-\phi h_k\right)\bigl(-f_k-f_u u_k-a_k v^{p-1}-a_u u_k v^{p-1}\\
&&\hspace{100pt}-(p-1)a(x,u)v^{p-3}u_1 u_{1k}-a^{ij}_{~,k}u_{ij}\bigr)\\[6pt]
&\geq& -a_u v^{p+q}-|Da|v^{p+q-1}-C\left(v^{q}+v^{p-1}\right)\\[6pt]
&&-\sum_k(v^{q-1}u_k-\phi h_k)\left((p-1)a(x,u)v^{p-3}u_1u_{1k}+a^{ij}_{~,k}u_{ij}\right).
\end{eqnarray*}

Recalling the critical equation $0=\psi_i$ and the computation \eqref{aijkuij}, by a computation similar to that in \textbf{Case 1}, we have
\begin{eqnarray*}
(p-1)a(x,u)v^{p-3}u_1\sum_k(v^{q-1}u_k-\phi h_k)u_{1k}~\leq~ C|a(x,u)|v^{p+q-1},
\end{eqnarray*}
and
\begin{eqnarray*}
\sum_k(v^{q-1}u_k-\phi h_k)a^{ij}_{~,k}u_{ij}\leq~ Cv^{p+q-2}M.
\end{eqnarray*}
Therefore, we have
\begin{eqnarray*}
\mathrm{II}\geq-|Da|v^{p+q-1}-C(v^q+v^{p-1})-C|a(x,u)|v^{p+q-1}-Cv^{p+q-2}M.
\end{eqnarray*}

Hence, by \eqref{nc-q>0-case2-uaa} and the Cauchy inequality, we derive
\begin{eqnarray*}
0&\geq& \frac{1}{(q+1)\rho}\left(a^{ii}\psi_{ii}(x_1,t_1)-\psi_t(x_1,t_1)\right)=\mathrm{I}+\mathrm{II}+\mathrm{III}\\[4pt]
&\geq& cv^{p+q-3}M^{2}-|Da|v^{p+q-1}-C|a(x,u)|v^{p+q-1}-Cv^{p+q-2}M-Cv^{p+q-1}\\[4pt]
&\geq&\left[c\cdot a(x,u)^2-|Da|-C|a(x,u)|\right]v^{p+q-1}.
\end{eqnarray*}
It follows from condition \eqref{a2} that $|Du(x_1,t_1)|\leq C$, and therefore $|Du(x_0,t_0)|\leq C$.

\end{proof}

\subsection{The case $q=0$} 
We consider the following prescribed contact angle problem
\begin{equation}\label{nc q=01}
\left\{\begin{aligned}
& u_t=\operatorname{div}(|Du|^{p-2}Du)+a(x,u)|Du|^{\tilde p-1}+f(x,u) &&\text{in~}\Omega\times[0,\infty),\\
& u(x,0)=u_0(x)  &&\text{in~}\Omega,\\
& |Du|^{-1}u_\nu=-\phi(x) &&\text{on~}\partial\Omega\times[0,\infty),
\end{aligned}\right.
\end{equation}
and the approximate problems
\begin{equation}\label{nc q=02}
\left\{\begin{aligned}
& u_t=\operatorname{div}(v^{p-2}Du)+a(x,u)v^{\tilde p-1}+f(x,u) &&\text{in~}\Omega\times[0,\infty),\\
& u(x,0)=u_{0}(x) &&\text{in~}\Omega,\\
& u_\nu=-\phi(x)|Du| &&\text{on~}\partial\Omega\times[0,\infty),
\end{aligned}\right.
\end{equation}
where the forcing function $a(x,u)$ satisfies \eqref{a1} and \eqref{a2}.
\begin{remark}
As in Remark~\ref{rmk2}, we prove the crucial gradient estimate only for $p \geq 2$. Moreover, thanks to the presence of the forcing term $a(x,u)$, the nearly vertical condition \eqref{nvc} is not required.
\end{remark}
In the following, the small constant $c>0$ and the large constant $C>0$ are independent of $\varepsilon$, $T$, and $|u|_{C^0}$, and may change from line to line. 
\begin{theorem}\label{thm8.2}
Assume $\Omega\subset \mathbb{R}^n$ is a bounded $C^{2,\beta}$ domain and  $p>1$. If $f$ satisfies \eqref{f}, $\phi$ satisfies \eqref{phi}, the forcing term $a(x,u)$ satisfies \eqref{a1} and \eqref{a2}, then the solution of \eqref{nc q=02} satisfies 
$$
\sup_{\overline{\Omega}\times[0,T]}|Du|\leq C(n,p,L,\Omega).
$$
\end{theorem}
\begin{proof}
We choose auxiliary function
$$w=v-\sum_{k=1}^n \phi u_k h_k.$$
Suppose $w(x,t)$ attains its maximum at $(x_0,t_0)\in \overline{\Omega}\times[0,T]$.

Without loss of generality, we may assume $t_0>0$ and $|Du(x_0,t_0)|$ is large enough. We use the same notation as in the proof of Theorem \ref{thm3.2}.\vspace{4pt}

\textbf{Case 1:} $x_0\in\Omega$.

 We choose coordinates such that $|Du|=u_1$ and the matrix $\{u_{\alpha\beta}\}_{\alpha,\beta>1}$ is diagonal at that point. Since $w$ attains its maximum at $(x_0,t_0)$, we have
\begin{eqnarray*}
0&\leq& w_t(x_0,t_0)=\sum_{k=1}^n(v^{-1}u_k-\phi h_k)u_{tk},\\
0&=&w_i(x_0,t_0)=v^{-1}\sum_k u_ku_{ki}-\phi_i\sum_k u_kh_k-\phi \sum_k u_{ki}h_k-\phi \sum_k u_kh_{ki},
\end{eqnarray*}
and
\begin{eqnarray*}
0&\geq& w_{ii}(x_0,t_0)=-v^{-3}\left(u_1u_{1i}\right)^2+v^{-1}\sum_{k}u_{ki}^2+v^{-1}\sum_k u_ku_{kii}-\phi_{ii}\sum_k u_kh_k\\
&&-2\phi_i\sum_k u_{ki}h_k-2\phi_{i}\sum_k u_kh_{ki}-\phi \sum_k u_{kii}h_k-2\phi \sum_k u_{ki}h_{ki}-\phi\sum_k  u_k h_{kii}.
\end{eqnarray*}
Hence
\begin{eqnarray}
0&\geq&a^{ii}w_{ii}(x_0,t_0)-w_t(x_0,t_0)\notag\\[4pt]
&=&-v^{-3}u_1^2a^{ii}u_{1i}^2+v^{-1}a^{ii}\sum_k u_{ki}^{2}\notag\\
&&+\sum_k \left(v^{-1}u_k-\phi h_k\right)\left(a^{ij}u_{ijk}-u_{kt}\right)\notag\\
&&-u_1 a^{ii}(\phi h_1)_{ii}-2a^{ii}\sum_k (\phi h_k)_i u_{ki}\label{nc-q=0-case1-I+II+III}\\[-4pt]
&=& \mathrm{I} +\mathrm{II} +\mathrm{III},\notag
\end{eqnarray}
where $\mathrm{I}$ denotes the two terms in the second line, $\mathrm{II}$ denotes the term in the third line, and $\mathrm{III}$ denotes the two terms in the fourth line.

Denote $M=\max|D^2u(x_0,t_0)|$, $M_1=\max_{\alpha\neq1} |u_{1\alpha}(x_0,t_0)|$, $M_2=\max_{\alpha\neq1}|u_{\alpha\alpha}(x_0,t_0)|$.  As in the previous subsection, we can prove the following claim for $v$ and $|a(x,u)|$ sufficiently large:\vspace{4pt}

\textbf{Claim 1.} $\sum_{\alpha\neq1}|u_{\alpha\alpha}|^{2}\geq ca(x,u)^2v^2$.\vspace{4pt}

Moreover, from the equation
$$u_t=\operatorname{div}(|Du|^{p-2}Du)+a(x,u)|Du|^{p-1}+f(x,u),$$
the boundedness of $|u_{t}|$ and the Cauchy inequality, we have
$$
\sum_{\alpha\neq1}|u_{\alpha\alpha}|^2\geq c|\sum_{\alpha\neq1}u_{\alpha\alpha}|^2=c\left|(u_t-f)v^{2-p}-\frac{\tilde{v}^2}{v^2}u_{11}-a(x,u)v\right|^2\geq c|u_{11}|^2-Ca(x,u)^2v^2.
$$
Combining to \textbf{Claim 1}, we further have 
\begin{eqnarray}\label{nc-q=0-case1-uaa}
\sum_{\alpha\neq1}|u_{\alpha\alpha}|^{2}\geq c|u_{11}|^2+ca(x,u)^2v^2.
\end{eqnarray}

Now returning to \eqref{nc-q=0-case1-I+II+III}, by \eqref{nc-q=0-case1-uaa} and direct computation, we have
\begin{eqnarray*}
\mathrm{I}\geq cv^{p-3}M^2+ca(x,u)^2v^{p-1},\quad \mathrm{III}\geq -Cv^{p-2}-Cv^{p-2}M,
\end{eqnarray*}
and
\begin{eqnarray*}
\mathrm{II}&=&\sum_k \left(v^{-1}u_k-\phi h_k\right)\bigl(-f_k-f_u u_k-a_k v^{p-1}-a_u u_k v^{p-1}\\
&&\hspace{100pt}-(p-1)a(x,u)v^{p-3}u_1 u_{1k}-a^{ij}_{~,k}u_{ij}\bigr)\\[6pt]
&\geq& -|Da|v^{p-1}-Cv^{p-1}-\sum_k(v^{-1}u_k-\phi h_k)\left((p-1)a(x,u)v^{p-3}u_1u_{1k}+a^{ij}_{~,k}u_{ij}\right).
\end{eqnarray*}

Recalling the critical equation $0=w_i$ and the computation \eqref{aijkuij}, we have
\begin{eqnarray*}
(p-1)a(x,u)v^{p-3}u_1\sum_k(v^{-1}u_k-\phi h_k)u_{1k}~\leq~ C|a(x,u)|v^{p-1},
\end{eqnarray*}
and
\begin{eqnarray}
\sum_k(v^{-1}u_k-\phi h_k)a^{ij}_{~,k}u_{ij}&=&(p-2)v^{p-4}u_1^2\Delta u(\phi h_1)_1+(p-2)(p-4)v^{p-6}u_1^4u_{11}(\phi h_1)_1\notag\\
&&+2(p-2)v^{p-4}u_1^2\sum_i u_{1i}(\phi h_1)_i~\leq~ Cv^{p-2}M.\label{nc-q=0-case1-aijkuij}
\end{eqnarray}
Therefore, we have
\begin{eqnarray*}
\mathrm{II}\geq-|Da|v^{p-1}-Cv^{p-1}-C|a(x,u)|v^{p-1}-Cv^{p-2}M.
\end{eqnarray*}

Hence, by \textbf{Claim 1}, \eqref{nc-q=0-case1-uaa} and the Cauchy inequality, we derive
\begin{eqnarray*}
0&\geq& a^{ii}w_{ii}(x_0,t_0)-w_t(x_0,t_0)=\mathrm{I}+\mathrm{II}+\mathrm{III}\\[4pt]
&\geq& cv^{p-3}M^{2}+ca(x,u)^2v^{p-1}-|Da|v^{p-1}-C|a(x,u)|v^{p-1}-Cv^{p-2}M-Cv^{p-1}\\[4pt]
&\geq&\left[c\cdot a(x,u)^2-|Da|-C|a(x,u)|-C\right]v^{p-1}.
\end{eqnarray*}
It follows from condition \eqref{a2} that $|Du(x_0,t_0)|\leq C.$\vspace{4pt}

\textbf{Case 2:} $x_0\in \partial\Omega$.

By the similar argument as in \textbf{Case 1} of the proof of Theorem \ref{thm8.1}, we obtain
\begin{equation}\label{nc-q=0-case2-wn}
w_{\nu}\geq -(C_{0}+C)w.
\end{equation}

Let $B = B(x_c, K_0)$ be the open ball with center $x_c \coloneqq x_0 + K_0 \vec{\nu}(x_0)$ such that $B \subseteq \Omega$ and $\overline{B} \cap (\mathbb{R}^n \setminus \Omega) = \{x_0\}$. Denote $\psi=\rho w$, where $\rho$ is a multiplier 
$$
\rho(x) \coloneqq -\frac{L_0}{2K_0} |x - x_c|^2 + \frac{L_0K_0}{2} + 1.
$$
If we choose $L_0>C_0+C\varepsilon_1$, then clearly
$$
\psi_{\nu}(x_0,t_0)>0.
$$
Moreover, the maximum of $\psi$ in $\overline{B}\times [0,T]$ must occur at a point $(x_1,t_1)\in B\times(0,T]$. 

We may assume that $t_1>0$, and choose coordinates such that $|Du|=u_1$ and the matrix $\{u_{\alpha\beta}\}_{\alpha,\beta>1}$ is diagonal at $(x_1,t_1)$. Since $\psi$ attains its local maximum at $(x_1,t_1)$, we have
\begin{eqnarray*}
0&\leq& \psi_t(x_1,t_1)=\rho(x_1)\sum_{k=1}^n\left(v^{-1}u_k-\phi h_k\right)u_{tk},\\
0&=&\psi_i(x_1,t_1)=\rho(x_1)\Bigl[v^{-1}\sum_k u_ku_{ki}-\phi_i\sum_ku_kh_k\\
&&\hspace{100pt}-\phi \sum_k u_{ki}h_k-\phi \sum_k u_kh_{ki}\Bigr]-\frac{L_0}{K_0}(x_{1,i}-x_{c,i})w,
\end{eqnarray*}
and
\begin{eqnarray*}
0&\geq& \psi_{ii}(x_1,t_1)=\rho(x_1)\Bigl[-v^{-3}(u_1u_{1i})^2+v^{-1}\sum_{k}u_{ki}^2+v^{-1}\sum_k u_ku_{kii}\\
&&-\phi_{ii}\sum_k u_kh_k-2\phi_i\sum_k u_{ki}h_k-2\phi_{i}\sum_k u_kh_{ki}-\phi \sum_k u_{kii}h_k\\
&&-2\phi \sum_k u_{ki}h_{ki}-\phi \sum_k u_k h_{kii}\Bigr]-\frac{2L_0}{K_0}w_i(x_{1,i}-x_{c,i})-\frac{L_0}{K_0}w.
\end{eqnarray*}
Hence
\begin{eqnarray}
0&\geq& \rho^{-1}\left(a^{ii}\psi_{ii}(x_1,t_1)-\psi_{t}(x_1,t_1)\right)\notag\\[2pt]
&=&\Big[-v^{-3}u_1^2a^{ii}u_{1i}^2+v^{-1}a^{ii}\sum_ku_{ki}^2+\sum_k \left(v^{-1}u_k-\phi h_k\right)\left(a^{ij}u_{ijk}-u_{kt}\right)\notag\\
&&-u_1a^{ii}(\phi h_1)_{ii}-2a^{ii}\sum_k (\phi h_k)_iu_{ki} \Big]-\frac{L_0}{K_0\rho}\left(v-\phi u_1h_1\right)\sum_ia^{ii}\notag\\
&&-\frac{2L_0}{K_0\rho}\left(x_{1,i}-x_{c,i}\right)a^{ii}\left(v^{-1}u_{1}u_{1i}-(\phi h_1)_iu_1-\phi \sum_k  h_k u_{ki}\right)\label{nc-q=0-case2-I+II+III}\\
&=&\mathrm{I}+\mathrm{II}+\mathrm{III},\notag
\end{eqnarray}
where
$$
\mathrm{I}=-v^{-3}u_1^2a^{ii}u_{1i}^2+v^{-1}a^{ii}\sum_ku_{ki}^2,
$$
$$
\mathrm{II}=\sum_k \left(v^{-1}u_k-\phi h_k\right)\left(a^{ij}u_{ijk}-u_{kt}\right),
$$
and $\mathrm{III}$ denotes all the other terms.

Denote $M=\max|D^2u(x_1,t_1)|$, $M_1=\max_{\alpha\neq1} |u_{1\alpha}(x_1,t_1)|$, $M_2=\max_{\alpha\neq1}|u_{\alpha\alpha}(x_1,t_1)|$. Similar to \textbf{Case 1}, we can prove 
$$
\sum_{\alpha\neq1}|u_{\alpha\alpha}|^2\geq ca(x,u)^2v^2,
$$
and
\begin{eqnarray}\label{nc-q=0-case2-uaa}
\sum_{\alpha\neq1}|u_{\alpha\alpha}|^{2}\geq c|u_{11}|^2+ca(x,u)^2v^2.
\end{eqnarray}
Now returning to \eqref{nc-q=0-case2-I+II+III}, by \eqref{nc-q=0-case2-uaa} and direct computation, we have
$$
\mathrm{I}\geq cv^{p-3}M^2+ca(x,u)^2v^{p-1},\quad \mathrm{III}\geq -Cv^{p-1}-Cv^{p-2}M,
$$
and 
\begin{eqnarray*}
\mathrm{II}&=&\sum_k \left(v^{-1}u_k-\phi h_k\right)\bigl(-f_k-f_u u_k-a_k v^{p-1}-a_u u_k v^{p-1}\\
&&\hspace{100pt}-(p-1)a(x,u)v^{p-3}u_1 u_{1k}-a^{ij}_{~,k}u_{ij}\bigr)\\[6pt]
&\geq&-|Da|v^{p-1}-Cv^{p-1}-\sum_k(v^{-1}u_k-\phi h_k)\left((p-1)a(x,u)v^{p-3}u_1u_{1k}+a^{ij}_{~,k}u_{ij}\right).
\end{eqnarray*}

Recalling the critical equation $0=\psi_i$ and the computation \eqref{aijkuij}, by a computation similar to that in \textbf{Case 1}, we have
\begin{eqnarray*}
(p-1)a(x,u)v^{p-3}u_1\sum_k(v^{-1}u_k-\phi h_k)u_{1k}~\leq~ C|a(x,u)|v^{p-1},
\end{eqnarray*}
and
\begin{eqnarray*}
\sum_k(v^{-1}u_k-\phi h_k)a^{ij}_{~,k}u_{ij}\leq~ Cv^{p-2}M.
\end{eqnarray*}
Therefore, we have
\begin{eqnarray*}
\mathrm{II}\geq-|Da|v^{p-1}-Cv^{p-1}-C|a(x,u)|v^{p-1}-Cv^{p-2}M.
\end{eqnarray*}

Hence, by \eqref{nc-q=0-case2-uaa} and the Cauchy inequality, we derive
\begin{eqnarray*}
0&\geq& \rho^{-1}\left(a^{ii}\psi_{ii}(x_1,t_1)-\psi_t(x_1,t_1)\right)=\mathrm{I}+\mathrm{II}+\mathrm{III}\\[4pt]
&\geq& cv^{p-3}M^{2}+ca(x,u)^2v^{p-1}-|Da|v^{p-1}-C|a(x,u)|v^{p-1}-Cv^{p-2}M-Cv^{p-1}\\[4pt]
&\geq&\left[c\cdot a(x,u)^2-|Da|-C|a(x,u)|-C\right]v^{p-1}.
\end{eqnarray*}
It follows from condition \eqref{a2} that $|Du(x_1,t_1)|\leq C$, and therefore $|Du(x_0,t_0)|\leq C$.

\end{proof}

\begin{proof}[\textbf{Proof of Theorem \ref{thm1.9}}]
Based on the crucial gradient estimates established in Theorem \ref{thm8.1} and Theorem \ref{thm8.2}, the proofs from Theorem \ref{thm1.1} to Theorem \ref{thm1.6} are identical to those in the convex case and are therefore omitted. For Theorem \ref{thm1.7}, according to Hopf's lemma, we have $\|u_{\varepsilon}\|_{L^{\infty}(\overline{\Omega}\times[0,\infty))}\leqslant\|u_{0}\|_{L^{\infty}(\overline{\Omega})}$. Therefore, the same arguments yield $\lim_{t\to\infty}u(x,t)=w(x)$, and $w$ satisfies
\begin{equation*}
\left\{\begin{aligned}
& \operatorname{div}(|Dw|^{p-2}Dw)+a(x,w)|Dw|^{p-1}=0 &&\text{in~}\Omega,\\
& w_\nu=0 &&\text{on~}\partial\Omega.
\end{aligned}\right.
\end{equation*}
By condition \eqref{a2}, we have $|a(x,w)|>0$; integrating this equation then yields that $w$ is constant.
\end{proof}

\section*{Acknowledgements} The authors would like to thank Prof. Xi-Nan Ma for his advice, constant support, and
encouragement. The first author was supported by NSFC grant No. 12301257.

\noindent \textbf{Research ethics:} Not applicable.

\noindent \textbf{Informed consent:} Not applicable.

\noindent \textbf{Author contributions:} All authors have accepted responsibility for the entire content of this manuscript and approved its submission.

\noindent \textbf{Use of Large Language Models, AI and Machine Learning Tools:} None declared.

\noindent \textbf{Conflict of interest:} The authors state no conflict of interest.

\noindent \textbf{Data availability:} Not applicable.


\begin{thebibliography}{00}

\bibitem{AW1994CVPDE} S. Altschuler and L. Wu, Translating surfaces of the non-parametric mean curvature flow with prescribed contact angle, Calc. Var. Partial Differential Equations { 2} (1994), no.~1, 101--111.

\bibitem{APPT2022AIM} P. Andrade, D. Pellegrino, E. Pimentel, E. Teixeira, C1-regularity for degenerate diffusion equations, Adv. Math. 409, part B (2022) 108667, 34 pp.

\bibitem{AS2022CVPDE} P. Andrade and M. Santos, Improved regularity for the parabolic normalized $p$-Laplace equation, Calc. Var. Partial Differential Equations { 61} (2022), no.~5, Paper No. 196, 13 pp.

\bibitem{AC2009IMJ} B. Andrews and J. Clutterbuck, Time-interior gradient estimates for quasilinear parabolic
equations, Indiana Univ. Math. J., 58 (2009), 351--380.

\bibitem{A2026arXiv} C. Antonini, Local and global $C^{1,\beta}$-regularity for uniformly elliptic quasilinear equations of $p$-Laplace and Orlicz-Laplace type, arXiv: 2601.07140.

\bibitem{ACF2023MA} C. Antonini, G. Ciraolo and A. Farina, Interior regularity results for inhomogeneous anisotropic quasilinear equations, Math. Ann. { 387} (2023), no.~3-4, 1745--1776.

\bibitem{ART2015CVPDE} D. Ara\'ujo, G. Ricarte and E. Teixeira, Geometric gradient estimates for solutions to degenerate elliptic equations, Calc. Var. Partial Differential Equations { 53} (2015), no.~3-4, 605--625.

\bibitem{A2016DIE} A. Attouchi, Gradient estimate and a Liouville theorem for a $p$-Laplacian evolution equation with a gradient nonlinearity, Differential Integral Equations { 29} (2016), no.~1-2, 137--150.

\bibitem{A2020NA} A. Attouchi, Local regularity for quasi-linear parabolic equations in non-divergence form, Nonlinear Anal. { 199} (2020), 112051, 28 pp.

\bibitem{AB2015JMPA} A. Attouchi and G. Barles, Global continuation beyond singularities on the boundary for a degenerate diffusive Hamilton-Jacobi equation, J. Math. Pures Appl. (9) { 104} (2015), no.~2, 383--402.

\bibitem{AP2018CCM} A. Attouchi and M. Parviainen, H\"older regularity for the gradient of the inhomogeneous parabolic normalized $p$-Laplacian, Commun. Contemp. Math. { 20} (2018), no.~4, 1750035, 27 pp.

\bibitem{APR2017JMPA} A. Attouchi, M. Parviainen and E. Ruosteenoja, $C^{1,\alpha}$ regularity for the normalized $p$-Poisson problem, J. Math. Pures Appl. (9) { 108} (2017), no.~4, 553--591.

\bibitem{AR2018JDE} A. Attouchi and E. Ruosteenoja, Remarks on regularity for $p$-Laplacian type equations in non-divergence form, J. Differential Equations { 265} (2018), no.~5, 1922--1961.

\bibitem{AR2020DCDS} A. Attouchi and E. Ruosteenoja, Gradient regularity for a singular parabolic equation in non-divergence form, Discrete Contin. Dyn. Syst. { 40} (2020), no.~10, 5955--5972.

\bibitem{BBLL2024} S. Baasandorj, S. Byun, K. Lee and S. Lee, $C^{1,\alpha}$-regularity for a class of degenerate/singular fully nonlinear elliptic equations, Interfaces Free Bound. 26 (2) (2024) 189–215.

\bibitem{BV2022PA} A. Banerjee and R. Verma, $C^{1,\alpha}$ regularity for degenerate fully nonlinear elliptic equations with Neumann boundary conditions, Potential Anal. { 57} (2022), no.~3, 327--365.

\bibitem{BCESS1997} M. Bardi, M. Crandall, L. Evans, H. Soner, and P. Souganidis, { Viscosity solutions and applications}, { Springer-Verlag, Berlin; Centro Internazionale Matematico Estivo (C.I.M.E.), Florence} (1997) x+259 pp.

\bibitem{B1993JDE} G. Barles, Fully nonlinear Neumann type boundary conditions for second-order elliptic and parabolic equations, J. Differential Equations { 106} (1993), no.~1, 90--106.

\bibitem{B1999JDE} G. Barles, Nonlinear Neumann boundary conditions for quasilinear degenerate elliptic equations and applications, J. Differential Equations { 154} (1999), no.~1, 191--224.

\bibitem{BC2026NODEA} G. Barles and E. Chasseigne, Some comparison results for first-order Hamilton-Jacobi equations and second-order fully nonlinear parabolic equations with Ventcell boundary conditions, NoDEA Nonlinear Differential Equations Appl. { 33} (2026), no.~1, Paper No. 30, 50 pp.

\bibitem{BRS2026NA} J. Bessa, G. Ricarte and P. Silva, Optimal gradient regularity to degenerate fully nonlinear elliptic models with oblique boundary condition, Nonlinear Anal. { 262} (2026), Paper No. 113919, 16 pp.

\bibitem{BD2006ADE} I. Birindelli and F. Demengel, First eigenvalue and maximum principle for fully nonlinear singular operators, Adv. Differential Equations { 11} (2006), no.~1, 91--119.

\bibitem{BD2007DCDS} I. Birindelli and F. Demengel, The Dirichlet problem for singular fully nonlinear operators, Discrete Contin. Dyn. Syst. { 2007}, Dynamical systems and differential equations. Proceedings of the 6th AIMS International Conference, 110--121.

\bibitem{BD2010JDE} I. Birindelli and F. Demengel, Regularity and uniqueness of the first eigenfunction for singular fully nonlinear operators, J. Differential Equations { 249} (2010), no.~5, 1089--1110.

\bibitem{BDL2022NA} I. Birindelli, F. Demengel and F. Leoni, Mixed boundary value problems for fully nonlinear degenerate or singular equations, Nonlinear Anal. { 223} (2022), Paper No. 113006, 22 pp.

\bibitem{BKO2025CVPDE} S. Byun, H. Kim and J. Oh, $C^{1,\alpha }$ regularity for degenerate fully nonlinear elliptic equations with oblique boundary conditions on $C^1$ domains, Calc. Var. Partial Differential Equations { 64} (2025), no.~5, Paper No. 174, 20 pp.

\bibitem{CC1995} L. Caffarelli and X. Cabr\'e, { Fully nonlinear elliptic equations}, American Mathematical Society Colloquium Publications, 43, Amer. Math. Soc., Providence, RI, 1995.

\bibitem{CM2021RMI} G. Chatzigeorgiou and E. Milakis, Regularity for fully nonlinear parabolic equations with oblique boundary data, Rev. Mat. Iberoam. { 37} (2021), no.~2, 775--820.

\bibitem{CGG1991JDG} Y. Chen, Y. Giga and S. Goto, Uniqueness and existence of viscosity solutions of generalized mean curvature flow equations, J. Differential Geom. { 33} (1991), no.~3, 749--786.

\bibitem{CIL1992} M. Crandall, H. Ishii and P. Lions, User's guide to viscosity solutions of second order partial differential equations, {Bull. of Amer. Soc.} {27}  1 (1992), 1-67.

\bibitem{DFQ2010CVPDE} G. D\'avila, P. Felmer and A. Quaas, Harnack inequality for singular fully nonlinear operators and some existence results, Calc. Var. Partial Differential Equations { 39} (2010), no.~3-4, 557--578.

\bibitem{D2011PA} F. Demengel, Existence's results for parabolic problems related to fully non linear operators degenerate or singular, Potential Anal. { 35} (2011), no.~1, 1--38.

\bibitem{D1983NA} E. DiBenedetto, {$\mathcal{C}^{1+\alpha}$-local regularity of weak solutions of degenerate elliptic equations}, Nonlinear Anal. {7} (1983), 827-850.

\bibitem{D1993} E. DiBenedetto, { Degenerate parabolic equations}, Universitext, Springer, New York, 1993.

\bibitem{DF1984JRAM} E. DiBenedetto and A. Friedman, Regularity of solutions of nonlinear degenerate parabolic systems, J. Reine Angew. Math. { 349} (1984), 83--128.

\bibitem{DF1985JRAM} E. DiBenedetto and A. Friedman, H\"older estimates for nonlinear degenerate parabolic systems, J. Reine Angew. Math. { 357} (1985), 1--22.

\bibitem{ES1991JDG} L. Evans and J. Spruck, Motion of level sets by mean curvature. I, J. Differential Geom. { 33} (1991), no.~3, 635--681.

\bibitem{FRZ2024MA} Y. Fang, V. R\u adulescu and C. Zhang, Equivalence of weak and viscosity solutions for the nonhomogeneous double phase equation, Math. Ann. { 388} (2024), no.~3, 2519--2559.

\bibitem{FRZ2021BLMS} Y. Fang, V. R\u adulescu and C. Zhang, Regularity of solutions to degenerate fully nonlinear elliptic equations with variable exponent, Bull. Lond. Math. Soc. { 53} (2021), no.~6, 1863--1878.

\bibitem{FZ2021JDE} Y. Fang and C. Zhang, Gradient H\"older regularity for parabolic normalized $p( x, t)$-Laplace equation, J. Differential Equations { 295} (2021), 211--232.

\bibitem{FZ2022MMAS} Y. Fang and C. Zhang, Equivalence between viscosity and weak solutions for the parabolic equations with nonstandard growth, Math. Methods Appl. Sci. { 45} (2022), no.~14, 8430--8449.

\bibitem{FZ2023CVPDE} Y. Fang and C. Zhang, Regularity for quasi-linear parabolic equations with nonhomogeneous degeneracy or singularity, Calc. Var. Partial Differential Equations { 62} (2023), no.~1, Paper No. 2, 46 pp.

\bibitem{F2021RE} C. De~Filippis, Regularity for solutions of fully nonlinear elliptic equations with nonhomogeneous degeneracy, Proc. Roy. Soc. Edinburgh Sect. A { 151} (2021), no.~1, 110--132.

\bibitem{FL2014JDG} A. Fraser and M. Li, Compactness of the space of embedded minimal surfaces with free boundary in three-manifolds with nonnegative Ricci curvature and convex boundary, J. Differential Geom. { 96} (2014), no.~2, 183--200.

\bibitem{FS2016Invention} A. Fraser and R. Schoen, Sharp eigenvalue bounds and minimal surfaces in the ball, Invent. Math. { 203} (2016), no.~3, 823--890.

\bibitem{GLX2024JFA} Z. Gao, B. Lou and J. Xu, Uniform gradient bounds and convergence of mean curvature flows in a cylinder, J. Funct. Anal. { 286} (2024), no.~5, Paper No. 110283, 28 pp.

\bibitem{GMWW2021JMS} Z. Gao, X. Ma, P. Wang and L. Weng, Nonparametric mean curvature flow with nearly vertical contact angle condition, J. Math. Study { 54} (2021), no.~1, 28--55.

\bibitem{G1976Pisa} C. Gerhardt, Global regularity of the solutions to the capillarity problem, Ann. Sci. Norm.
Sup. Piss Ser. (4), 3 (1976), 157--175.

\bibitem{GOS1999JDE} Y. Giga, M. Ohnuma and M. Sato, On the strong maximum principle and the large time behavior of generalized mean curvature flow with the Neumann boundary condition, J. Differential Equations { 154} (1999), no.~1, 107--131.

\bibitem{GS1993DIE} Y. Giga and M. Sato, Neumann problem for singular degenerate parabolic equations, Differential Integral Equations { 6} (1993), no.~6, 1217--1230.

\bibitem{HTZ2019GF} Y. Giga, H. Tran and L. Zhang, On obstacle problem for mean curvature flow with driving force, Geom. Flows { 4} (2019), no.~1, 9--29.

\bibitem{G1996} B. Guan, Mean curvature motion of non-parametric hypersurfaces with contact angle condition, in Elliptic and parabolic methods in geometry (Minneapolis, MN, 1994), AK Peters,
Wellesley, MA,(1996), 47--56.

\bibitem{H1984JDG} G. Huisken, Flow by mean curvature of convex surfaces into spheres, J. Differential Geom.,
20 (1984), 237--266.

\bibitem{H1986Invention} G. Huisken, Contracting convex hypersurfaces in Riemannian manifolds by their mean curvature, Invent. Math., 84 (1986), 463--480.

\bibitem{H1989JDE} G. Huisken, Non-parametric mean curvature evolution with boundary conditions, J. Differ.
Equations., 77 (1989), 369--378.

\bibitem{IJS2019ANA} C. Imbert, T. Jin and L. Silvestre, H\"older gradient estimates for a class of singular or degenerate parabolic equations, Adv. Nonlinear Anal. { 8} (2019), no.~1, 845--867.

\bibitem{IS2013AIM} C. Imbert and L. Silvestre, $C^{1,\alpha}$ regularity of solutions of some degenerate fully non-linear elliptic equations, Adv. Math. { 233} (2013), 196--206.

\bibitem{IL1990JDE} H. Ishii and P. Lions, Viscosity solutions of fully nonlinear second-order elliptic partial differential equations, J. Differential Equations { 83} (1990), no.~1, 26--78.

\bibitem{J2023CVPDE} J. Jang, Capillary-type boundary value problems of mean curvature flows with force and transport terms on a bounded domain, Calc. Var. Partial Differential Equations { 62} (2023), no.~3, Paper No. 108, 55 pp.

\bibitem{JKMT2011JMPA} J.Jang, D.Kwon, H.Mitake and H.Tran, Level-set forced mean curvature flow with the Neumann boundary condition, J. Math. Pures Appl. (9) { 168} (2022), 143--167.

\bibitem{JS2017JMPA} T. Jin and L. Silvestre, H\"older gradient estimates for parabolic homogeneous $p$-Laplacian equations, J. Math. Pures Appl. (9) { 108} (2017), no.~1, 63--87.

\bibitem{JJ2012CPDE} V. Julin and P. Juutinen, {A new proof for the equivalence of weak and viscosity solutions for the $p$-Laplace equation}. Communications in PDE {37} 5 (2012), 934-946.

\bibitem{JLM2001SIAMMA} P. Juutinen, P. Lindqvist, and J. Manfredi, \textit{On the equivalence of viscosity solutions and weak solutions for a quasilinear equation}, SIAM J. Math. Anal. {33} 3 (2001), 699-717.

\bibitem{K1988CPDE} N. Korevaar, Maximum principle gradient estimates for the capillary problem, Comm. in
Partial Differential Equations, 13 (1988), 1--31.

\bibitem{LLY2024JMPA} K. Lee, S. Lee and H. Yun, $C^{1,\alpha}$-regularity for solutions of degenerate/singular fully nonlinear parabolic equations, J. Math. Pures Appl. (9) { 181} (2024), 152--189.

\bibitem{LY2024JDE} S. Lee and H. Yun, $C^{1,\alpha}$-regularity for functions in solution classes and its application to parabolic normalized $p$-Laplace equations, J. Differential Equations { 378} (2024), 539--558.

\bibitem{LY2025CVPDE} K. Lee and H. Yun, Boundary regularity for viscosity solutions of fully nonlinear degenerate/singular parabolic equations, Calc. Var. Partial Differential Equations { 64} (2025), no.~1, Paper No. 25, 32 pp.

\bibitem{L1983IMJ} J. Lewis, Regularity of the derivatives of solutions to certain degenerate elliptic equations, Indiana Univ. Math. J. { 32} (1983), no.~6, 849--858.

\bibitem{LZ2018ARMA} D. Li and K. Zhang, Regularity for fully nonlinear elliptic equations with oblique boundary conditions, Arch. Ration. Mech. Anal. { 228} (2018), no.~3, 923--967.

\bibitem{L1988CPDE} G. Lieberman, Gradient estimates for capillary-type problems via the maximum principle,
Commun. in Partial Differential Equations, 13 (1988), 33--59.

\bibitem{L1988NA} G. Lieberman, Boundary regularity for solutions of degenerate elliptic equations, Nonlinear Anal. { 12} (1988), no.~11, 1203--1219.

\bibitem{L1990NA} G. Lieberman, Boundary regularity for solutions of degenerate parabolic equations, Nonlinear Anal. { 14} (1990), no.~6, 501--524.

\bibitem{L2013} G. Lieberman, Oblique Boundary Value Problems for Elliptic Equations, World Scientific
Publishing Co. Pte. Ltd., Hackensack, NJ, 2013.

\bibitem{LT1986TAMS} G. Lieberman and N. Trudinger, Nonlinear oblique boundary value problems for nonlinear elliptic equations, Trans. Amer. Math. Soc. { 295} (1986), no.~2, 509--546.

\bibitem{MWW2018JFA} X. Ma, P. Wang and W. Wei, Constant mean curvature surfaces and mean curvature flow with non-zero Neumann boundary conditions on strictly convex domains, J. Funct. Anal. { 274} (2018), no.~1, 252--277.

\bibitem{MX2016CPAAAIM} X. Ma and J. Xu, Gradient estimates of mean curvature equations with Neumann
boundary condition, Advances in Mathematics, 290 (2016), 1010--1039.

\bibitem{MO2019ANA} M. Medina and P. Ochoa, On viscosity and weak solutions for non-homogeneous $p$-Laplace equations, Adv. Nonlinear Anal. { 8} (2019), no.~1, 468--481.

\bibitem{MWW2025AIM} X. Mei, G. Wang and L. Weng, The capillary Minkowski problem, Adv. Math. { 469} (2025), Paper No. 110230, 29 pp.

\bibitem{MWWX2025MathZ} X. Mei, G. Wang, L. Weng, C. Xia, Alexandrov-Fenchel inequalities for convex hypersurfaces in the half-space with capillary boundary II, Math. Z. { 310} (2025), no.~4, Paper No. 71, 17 pp.

\bibitem{MW2023JGA} X. Mei and L. Weng, A constrained mean curvature type flow for capillary boundary hypersurfaces in space forms, J. Geom. Anal. { 33} (2023), no.~6, Paper No. 195, 28 pp.

\bibitem{MS2006CPDE} E. Milakis and L. Silvestre, Regularity for fully nonlinear elliptic equations with Neumann boundary data, Comm. Partial Differential Equations { 31} (2006), no.~7-9, 1227--1252.

\bibitem{MT2017NODEA} M. Mizuno and K. Takasao, Gradient estimates for mean curvature flow with Neumann boundary conditions, NoDEA Nonlinear Differential Equations Appl. { 24} (2017), no.~4, Paper No. 32, 24 pp.

\bibitem{MV2023BMS} A. Mohammed and A. Vitolo, The effects of nonlinear perturbation terms on comparison principles for the $p$-Laplacian, Bull. Math. Sci. { 14} (2024), no.~2, Paper No. 2450005, 30 pp.

\bibitem{MV2024BMS} A. Mohammed and A. Vitolo, Remarks on comparison principles for $p$-Laplacian with extension to $(p,q)$-Laplacian, Bull. Math. Sci. { 14} (2024), no.~3, Paper No. 2450011, 18 pp.

\bibitem{OS1997CPDE} M. Ohnuma and K. Sato, Singular degenerate parabolic equations with applications to the $p$-Laplace diffusion equation, Comm. Partial Differential Equations { 22} (1997), no.~3-4, 381--411.

\bibitem{P2008JMPA} S. Patrizi, The Neumann problem for singular fully nonlinear operators, J. Math. Pures Appl. (9) { 90} (2008), no.~3, 286--311.

\bibitem{PT2021JDE} D. Prazeres and E. Topp, Interior regularity results for fractional elliptic equations that degenerate with the gradient, J. Differential Equations { 300} (2021), 814--829.

\bibitem{PS2007} P. Pucci and J. Serrin, Maximum principles for elliptic partial differential equations, in { Handbook of differential equations: stationary partial differential equations. Vol. IV}, 355--483, Handb. Differ. Equ., Elsevier/North-Holland, Amsterdam.

\bibitem{R2020NA} G. Ricarte, Optimal $C^{1,\alpha}$ regularity for degenerate fully nonlinear elliptic equations with Neumann boundary condition, Nonlinear Anal. { 198} (2020), 111867, 13 pp. 

\bibitem{SWX2022JDG} J. Scheuer, G. Wang and C. Xia, Alexandrov-Fenchel inequalities for convex hypersurfaces with free boundary in a ball, J. Differential Geom. { 120} (2022), no.~2, 345--373.

\bibitem{S2021JEE} J. Siltakoski, Equivalence of viscosity and weak solutions for a $p$-parabolic equation, J. Evol. Equ. { 21} (2021), no.~2, 2047--2080.

\bibitem{S2022JMAA} J. Siltakoski, H\"older gradient regularity for the inhomogeneous normalized $p(x)$-Laplace equation, J. Math. Anal. Appl. { 513} (2022), no.~1, Paper No. 126187, 27 pp.

\bibitem{SR2020CVPDE} J. da~Silva and G. Ricarte, Geometric gradient estimates for fully nonlinear models with non-homogeneous degeneracy and applications, Calc. Var. Partial Differential Equations { 59} (2020), no.~5, Paper No. 161, 33 pp.

\bibitem{SS1976ARMA} L. Simon and J. Spruck, Existence and regularity of a capillary surface with prescribed contact
angle, Arch. Rational Mech. Anal., 61 (1976), 19--34.

\bibitem{S1975CPAM} J. Spruck, On the existence of a capillary surface with prescribed contact angle, Comm. Pure
Appl. Math., 28 (1975), 189--200.

\bibitem{T1984JDE} P. Tolksdorf, Regularity for a more general class of quasilinear elliptic equations, J. Differential Equations { 51} (1984), no.~1, 126--15.

\bibitem{U1973} N. Ural’tseva, The solvability of the capillary problem, (Russian) Vestnik Leningrad. Univ.
No. 19 Mat. Meh. Astronom.Vyp., 4 (1973), 54--64.

\bibitem{WWX2024MA} G. Wang, L. Weng and C. Xia, Alexandrov-Fenchel inequalities for convex hypersurfaces in the half-space with capillary boundary, Math. Ann. { 388} (2024), no.~2, 2121--2154.

\bibitem{WWX2024JFA} G. Wang, L. Weng and C. Xia, A Minkowski-type inequality for capillary hypersurfaces in a half-space, J. Funct. Anal. { 287} (2024), no.~4, Paper No. 110496, 22 pp.

\bibitem{WWX2019CPAA} J. Wang, W. Wei and J. Xu, Translating solutions of non-parametric mean curvature flows with capillary-type boundary value problems, Commun. Pure Appl. Anal. { 18} (2019), no.~6, 3243--3265.

\bibitem{WYJ2025PA} J. Wang, Y. Yin and F. Jiang, Regularity of solutions to degenerate normalized $p$-Laplacian equation with general variable exponents, Potential Anal. { 63} (2025), no.~4, 1963--2000.

\bibitem{X2016CPAA} J. Xu, A new proof of gradient estimates for mean curvature equations with oblique boundary conditions, Commun. Pure Appl. Anal., 15 (2016), 1719--1742.














\end{thebibliography}
\end{document}